\documentclass{article}

\usepackage[T1]{fontenc} 
\usepackage[english]{babel}

\usepackage{mathtools,amssymb} 
\usepackage{mathrsfs} 
\usepackage{bbm} 
\usepackage{dsfont} 
\usepackage[only,llbracket,rrbracket]{stmaryrd}

\usepackage[a4paper,top=2cm,bottom=2cm,left=3cm,right=3cm,marginparwidth=1.75cm]{geometry}

\usepackage{graphicx}
\usepackage{tikz}

\usepackage[colorlinks = true,
                linkcolor = red,
                urlcolor  = blue,
                citecolor = red,
                anchorcolor = magenta]{hyperref}
\usepackage[amsmath,thmmarks,hyperref]{ntheorem} 
\usepackage[nameinlink]{cleveref} 

\usepackage{tcolorbox}
\usepackage[shortlabels]{enumitem}

\usepackage[sortcites=true]{biblatex} 
\usepackage{csquotes}
\bibliography{biblio.bib}

\AtEveryBibitem{
\clearfield{url}
\clearfield{urldate}
\clearfield{urlyear}
}

\newcommand{\R}{\mathbb R}

\newcommand{\Z}{\mathbb Z}
\newcommand{\N}{\mathbb N}

\newcommand{\1}{\mathbb 1}

\newcommand{\lllbracket}{ [\![ }

\newcommand{\dd}{\mathrm{d}}

\renewcommand{\P}{\mathbb P}
\newcommand{\E}{\mathbb E}

\newcommand{\gen}{\mathrm{Gen}}

\newcommand{\norm}{\mathrm{Norm}}

\newcommand{\bcp}{\mathrm{BCap}}
\newcommand{\esc}{\mathrm{Esc}}

\newcommand{\Tree}{\mathcal T}

\renewcommand{\1}{\mathds 1}
\renewcommand{\lllbracket}{\llbracket}

\providecommand{\tightlist}{%
  \setlength{\itemsep}{0pt}\setlength{\parskip}{0pt}} 

\theoremstyle{break}

\newtheorem{definition}{Definition}[section]
\newtheorem{lemma}[definition]{Lemma}
\newtheorem{proposition}[definition]{Proposition}
\newtheorem{theorem}[definition]{Theorem}
\newtheorem{corollary}[definition]{Corollary}
\newtheorem{remark}[definition]{Remark}

\newtheorem{assumption}{Assumption}

\newtheorem{claim}[definition]{Claim}

\theoremstyle{empty}
\theoremheaderfont{\bfseries}
\theorembodyfont{\normalfont}
\theoremseparator{}
\theoremsymbol{\(\blacksquare\)}
\newtheorem{proof}{Proof}

\AddToHook{env/lemma/begin}{\crefalias{definition}{lemma}}
\AddToHook{env/proposition/begin}{\crefalias{definition}{proposition}}
\AddToHook{env/theorem/begin}{\crefalias{definition}{theorem}}
\AddToHook{env/corollary/begin}{\crefalias{definition}{corollary}}
\AddToHook{env/remark/begin}{\crefalias{definition}{remark}}
\AddToHook{env/claim/begin}{\crefalias{definition}{claim}}

\providecommand{\tightlist}{%
  \setlength{\itemsep}{0pt}\setlength{\parskip}{0pt}}

\usepackage{bookmark}
\IfFileExists{xurl.sty}{\usepackage{xurl}}{} 
\usepackage{tabularx}
\usepackage{booktabs}
\usepackage[labelfont=bf,format=plain,justification=raggedright,singlelinecheck=false]{caption}

\title{Recurrence and capacity of stable branching random walks}
\author{Antoine Aurillard\footnote{CMAP, École polytechnique \& ICJ, Université Lyon 1, France. \textit{email:} antoine.aurillard@polytechnique.edu}}
\date{}

\begin{document}
\maketitle

\begin{abstract}
    We study the linear growth rate of the range of size-conditioned Branching Random Walks (BRW) when the offspring distribution $\mu$ is critical and attracted to an $\alpha$-stable law. This is done via the infinite invariant BRW introduced by Le Gall \& Lin and a new criterion which relates this growth rate of the range to a notion of dimension of the underlying tree in a general way. Then, in the transient case (that is, when the range does grow linearly), we extend the notion of branching capacity to this $\alpha$-stable case. We show that it is still related to the asymptotic probability that a BRW (or its infinite version) reaches a distant set in $\Z^d$, and we estimate the $\alpha$-stable branching capacity of balls.
\end{abstract}

\section{Introduction}\label{introduction}

Consider a centered tree-indexed in \(\Z^{d}\). To do so, start with either a Bienaymé tree \(T=\Tree\) with a critical offspring distribution \(\mu\) or its size-conditioned version \(T=\Tree_{n}\), which is a \(\mu\)-Bienaymé tree conditioned to have \(n+1\) vertices. Then, assign independently to each edge of $T$ a jump in $\Z^d$ sampled according to a centered distribution \(\theta\). The resulting walk indexed by $T$, simply called a $T$-walk in the rest of this paper, is obtained by declaring that the root of $T$ is at $0 \in \Z^d$ while every other $u \in T$ is at the position of its parent plus the jump from this parent to $u$. Note that in the case $T=\Tree$, we recover the classical model of critical Branching Random Walk (BRW). The goal of this paper is to understand how the global shape of $T$ affects the spread of the corresponding $T$-walk. More precisely, we focus on the probability to reach a distant set of positions in \(\Z^{d}\) and on the range of those walks, \emph{i.e.} the number of positions \(x \in \Z^{d}\) reached by the walks and denoted by \(R^{\theta}(T)\).

This question on tree-indexed walks has mainly been studied when \(\mu\), in
addition of having mean \(1\), has a finite variance
\(\sigma_{\mu}^2 < +\infty\). In particular, Le Gall \& Lin \autocite{LeGallLin2016RangeTreeindexedRandom} studied the linear or sublinear growth rate of $R^{\theta}(\Tree_n)$ by relating it to the transience or recurrence of an infinite invariant version of a $\Tree$-walk. However, their results are partial in the broader setting where \(\mu\) may have an infinite variance but is in the domain of
attraction of an \(\alpha\)-stable distribution,\footnote{We denote this by $\mu \in \mathrm{dom}(\alpha)$ and provide some background in \Cref{offspring-attracted-to-an-alpha-stable-distribution}} while the underlying trees \(\Tree\) and \(\Tree_{n}\) are still well understood in this context. Indeed, it is known that the large scale geometry of \(\Tree\) and
\(\Tree_{n}\) primarily depends on the parameter \(\alpha \in [1,2]\), in
particular they are \(\frac{\alpha}{\alpha-1}\)-dimensional in the sense
of Maillard \autocite{Maillard2016MaximumTreeindexedRandom}.

In this paper, we study the impact of the geometric characteristics of Bienaymé trees in an $\alpha$-stable domain on the behaviour of the corresponding tree-indexed walks. Concretely, we first complete the
previous work of Le Gall \& Lin \autocite{LeGallLin2016RangeTreeindexedRandom} in the setting of a critical offspring distribution $\mu \in \mathrm{dom}(\alpha)$. We use an explicit connection between the
dimension of \(\Tree\) and the range of a \(\Tree\)-walk to show that the
recurrence or transience of the infinite walk in \(\Z^{d}\) (almost) only depends
on this dimension \(\frac{\alpha}{\alpha-1}\). More precisely, this walk is
recurrent whenever \(d < 2 \frac{\alpha}{\alpha-1}\) and transient
whenever \(d > 2 \frac{\alpha}{\alpha-1}\). We also briefly
discuss some new behaviour appearing at the critical dimension
\(d = 2\frac{\alpha}{\alpha-1}\), as well as the limiting case
\(\alpha=1\). Those results are actually based on a more general one, whose main ingredient comes from
\autocite{Maillard2016MaximumTreeindexedRandom}: given any tree-indexed random walk where the underlying tree is $D$-dimensional in some sense specified later, the critical dimension separating recurrent
and transient behaviour is at least \(2D\).

In a second part, we extend
the notion of branching capacity for transient branching random walk
introduced by Zhu \autocite{Zhu2017CriticalBranchingRandom} in the finite
variance case. Interestingly, in the $\alpha$-stable case the geometry of the
underlying tree seems to have little influence on the probability that a
\(\Tree\)-walk starting from \(x \in \Z^{d}\) far away from the origin
reaches some fixed finite set \(A\). Indeed, whatever is \(\alpha\) this
probability is actually comparable to the same one as for a classical
random walk. However, when \(A\) gets large as well, the geometric
properties of \(\Tree\) come back to play a role into the asymptotic
probability of reaching \(A\).

Let us fix once and for all our main assumptions on the offspring distribution $\mu$ and the jump distribution $\theta$.

\begin{assumption}
  \begin{itemize}
    \item $\mu$ is critical, \textit{i.e.} $\sum_{k \in \N} k\mu(k)=1$ and $\mu \neq \delta_1$. 
    
    \item $\mu \in \mathrm{dom}(\alpha)$ for $\alpha \in [1,2]$. Actually, we will explicitly consider the case $\alpha=1$ in \Cref{cor:RecCauchy} but apart from this, we always assume that $\alpha >1$. A critical $\mu \in \mathrm{dom}(\alpha)$ where $\alpha \in (1,2]$ is characterized\footnote{see \Cref{annexe-regularly-varying-functions} for a reminder on attraction to a stable distribution and regularly varying functions.} by the existence of a slowly varying function \(L\) such that the generating function \(\gen_{\mu}\) of \(\mu\) satisfies
\begin{equation}\label{eq:GenMuStable}{
 \gen_{\mu}(1-s)-(1-s) = s^{\alpha}L\left( \frac{1}{s} \right).
}\end{equation}
In particular, the finite variance setting is included in the case \(\alpha=2\).

    \item $\theta$ is centered in $\Z^d$ and has an invertible covariance matrix $\Sigma_{\theta}^{2}$.

  \end{itemize}
\end{assumption}

\subsection{Recurrence and transience of stable BRW}\label{recurrence-and-transience-of-stable-brw}

\subsubsection{The invariant branching random walk}\label{the-infinite-walk-associated-to-branching-random-walk}

Le Gall \& Lin \autocite{LeGallLin2016RangeTreeindexedRandom} initiated the study of
the range of a finite but large branching random walk and defined an
important tool: they introduced an infinite tree, denoted here as
\(\Tree_{\mathrm{inv}}\), which enjoys an invariance property. This
enables to study a \(\Tree_{\mathrm{inv}}\)-indexed walk with an
approach based on ergodic theory already standard for classical random
walk. At the same time \(\Tree_{\mathrm{inv}}\) is closely linked to
\(\Tree\) and \(\Tree_{n}\). Thanks to this, the
results concerning \(\Tree_{\mathrm{inv}}\)-indexed walk have been transfered to
\(\Tree_{n}\)-walk.

In later works, \textit{e.g.} Zhu  \autocite{Zhu2017CriticalBranchingRandom} and
Asselah, Schapira \& Sousi \autocite{AsselahSchapiraSousi2025LocalTimesCapacity}, slightly
different versions of this invariant tree have been used. Here we shall work
with the one constructed as a pointed local limit of finite but large
\(\mu\)-Bienaymé trees, which we denote by \(\Tree_{\infty}\). This tree
will be properly introduced in \Cref{invariant-trees-and-walks}, but
informally speaking one may construct it  as follows. First pick a vertex uniformly
at random in \(\Tree_{n}\). We will focus on the neighbourhood of this point, so we call it the \textit{cursor} and denote it by $c_0$. Then take \(n\) large, and observe that the distance between the root $\varnothing$ and the cursor $c_0$ tends to $+\infty$ almost surely. Seen from the cursor $c_0$, the root \(\varnothing\) is sent at infinity and the genealogical path
from \(\varnothing\) to \(c_0\) is stretched into a semi-infinite path
\(\ldots \to c_{-2} \to c_{-1} \to c_{0}\) called the spine. Thus the limit tree $\Tree_{\infty}$ contains this spine and some finite trees grafted to it. This tree
\(\Tree_{\infty}\) is to \(\Tree_{\mathrm{inv}}\) what a bi-infinite
line \(\Z\) is to \(\N\): after ordering \(\Tree_{\infty}\) with a
depth-first search starting from infinity (because the root $\varnothing$ has been sent at infinity), \(\Tree_{\infty}\) can be split into the \emph{future} \(\Tree_{+}\) containing all vertices coming after \(c_0\) in
this depth-first search and the \emph{past} \(\Tree_{-}\) containing
those before \(c_0\), and \(\Tree_{\mathrm{inv}}\) (almost)
identifies as \(\Tree_{+}\), see \Cref{fig:futureAndPast}.

\begin{figure}[!h]
  \begin{center}
    \includegraphics[height=7cm]{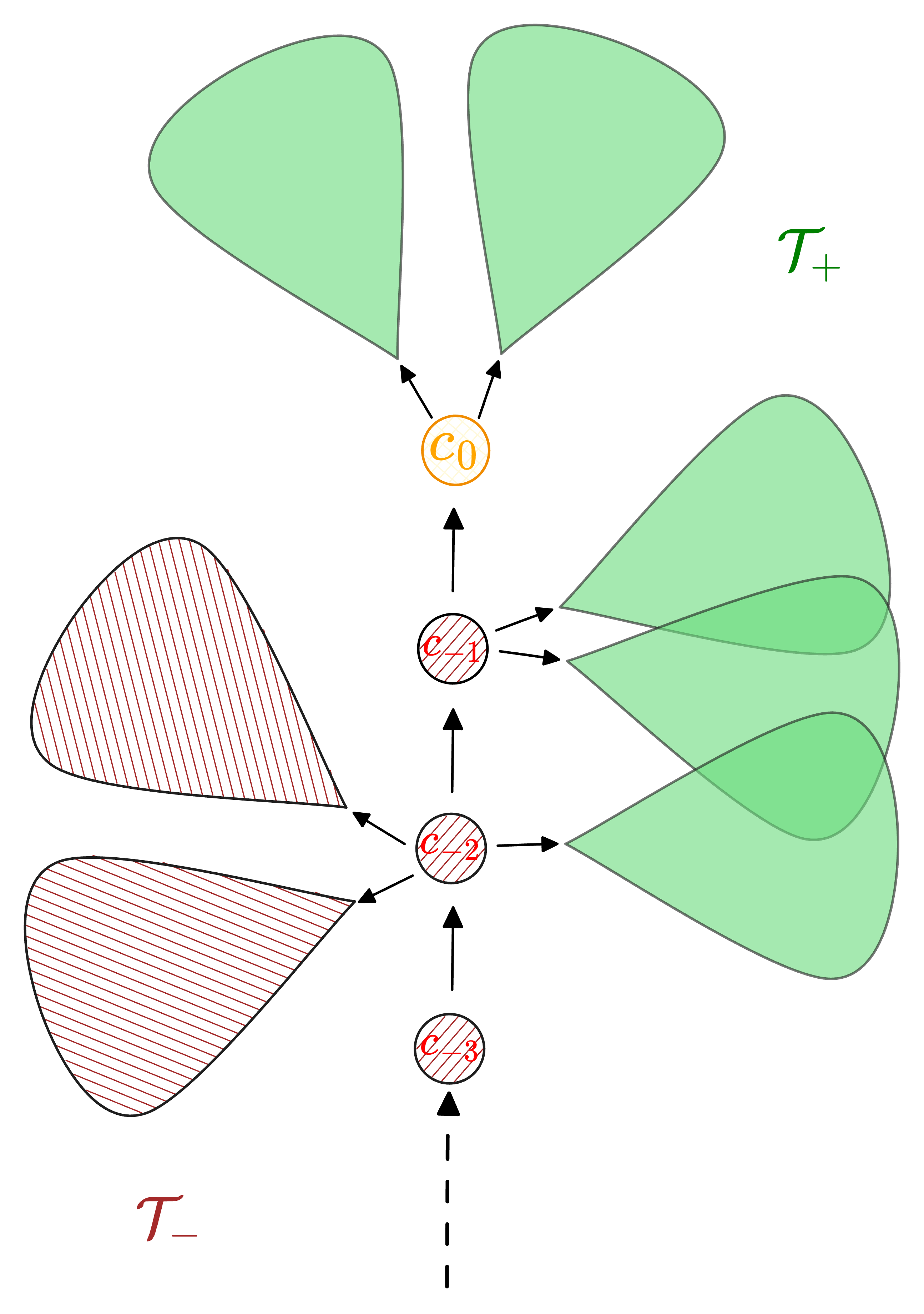}
    \caption{The tree \(\Tree_{\infty}\), with the future in full green and the past in hatched brown. \(\Tree_{\mathrm{inv}}\) is the tree containing the future and the infinite spine. }\label{fig:futureAndPast}
  \end{center}
\end{figure}

In this setting, one of the main results in \autocite{LeGallLin2016RangeTreeindexedRandom} may be
written as follows: assuming that $\mu \in \mathrm{dom}(\alpha)$ with \(\alpha \in (1,2]\), and without assumption on $\theta$, we have

\begin{equation}\label{eq:LinkFiniteInfinite}{
\frac{R^{\theta}(\Tree_{n})}{n} \xrightarrow[n \to \infty]{(\P)} \P \bigl( \esc_{+} \bigr),
}\end{equation}
where \(\esc_{+}= \bigl(W_{\Tree_{\infty}}(u) \neq 0\ \forall u \in \Tree_{+} \bigr)\)
is the event that a \(\Tree_{\infty}\)-walk \(W_{\Tree_{\infty}}\)
starting\footnote{As defined in \Cref{tree-indexed-walk-and-brw}, the starting position for a $\Tree_{\infty}$-walk is by convention the position of $c_0$.} from \(0\) does not come back to its initial position \emph{in the
future}.

In \eqref{eq:LinkFiniteInfinite}, \(\P(\esc_{+})\) clearly is the branching equivalent
of the probability that a classical random walk does not come back to
its initial position, which characterizes the recurrent or transient
behaviour of this walk. As \Cref{prop:DichotomyRecTransienceAdlous} stresses, the exact same
dichotomy holds for a \(\Tree_{\infty}\)-walk. Consequently, we will say that
\(\Tree_{\infty}\)-walks and by extension \(\Tree_{n}\)-walks are
\emph{transient} when \(\P(\esc_{+})>0\), \emph{i.e.} \(R^{\theta}(\Tree_{n})\)
is of order \(n\) with high probability, otherwise they are said to be
\emph{recurrent}.

\begin{proposition}[Dichotomy recurrence/transience for $\Tree_{\infty}$-walk]\label{prop:DichotomyRecTransienceAdlous}
Assume that \(\theta\) is not supported by a strict subgroup of \(\Z^d\). Consider the \emph{total escape} event
\(\esc= \bigl(W_{\Tree_{\infty}}(u) \neq 0\ \forall u \in \Tree_{\infty}\!\setminus\! \{ c_0 \} \bigr)\).

\begin{enumerate}[(i)]
\item
  If \(\P(\esc)=0\) then almost surely, for every \(x \in \Z^d\),
  \(\sum_{u \in \Tree_{\infty}} \1_{W_{\Tree_{\infty}}(u)=x} = +\infty\)
  \emph{i.e.} \(x\) is visited infinitely often by
  \(W_{\Tree_{\infty}}\).
\item
  If \(\P(\esc)>0\) then almost surely, for every \(x \in \Z^d\),
  \(\sum_{u \in \Tree_{\infty}} \1_{W_{\Tree_{\infty}}(u)=x} < +\infty\).
\end{enumerate}
Moreover, \[\P(\esc)=0 \text{ if and only if } \P(\esc_{+})=0,\] where
\(\esc_{+}= \bigl(W_{\Tree_{\infty}}(u) \neq 0\ \forall u \in \Tree_{+} \bigr)\)
as in \eqref{eq:LinkFiniteInfinite}, and in the recurrent case we actually have that almost surely, for every \(x \in \Z^d\),
  \(\sum_{u \in \Tree_{+}} \1_{W_{\Tree_{\infty}}(u)=x} = +\infty\).

\end{proposition}

\begin{remark}
\Cref{prop:DichotomyRecTransienceAdlous} actually holds without assuming that \(\theta\) is centered or has a finite second moment.
\end{remark}

\subsubsection{Classification of the recurrent behaviour}\label{classification-of-the-recurrent-behaviour}

The second main result of \autocite{LeGallLin2016RangeTreeindexedRandom}
deals with the classification of the recurrent/transient behaviour of
tree-indexed walks. With the additional assumption that $\theta$ has finite moment of order $d-1$, they obtain

\begin{enumerate}[(i)]
\item When \(\mu\) has a finite variance, then a
\(\Tree_{n}\)-walk (or a \(\Tree_{\infty}\)-walk) in \(\Z^{d}\) is
transient if and only if \(d >4\).
\item When $\mu \in \mathrm{dom}(\alpha)$ with \(\alpha \in (1,2]\), then
\(d>\frac{2\alpha}{\alpha-1}\) implies transience.
\end{enumerate}

This result is based on a sufficient criterion implying
\(\P(\esc_{+})>0\) which is recalled in \Cref{le-gall-and-lins-transience-criterion-at-critical-dimension}. However the second
moment method used to establish recurrence for \(d\leq 4\) when the
variance of $\mu$ is finite fails as soon as we relax this assumption.

Here we address the remaining question in the \(\alpha\)-stable case,
which is to know if there is recurrence as soon as
\(d \leq  \frac{2\alpha}{\alpha-1}\), with a quite different approach.
We directly work with the finite trees \(\Tree_{n}\) and we bound from above
\(R^{\theta}(\Tree_{n})\) thanks to a connection between the maximum of a
\(\Tree_{n}\)-walk and the dimension of \(\Tree_{n}\) adapted from Maillard
\autocite{Maillard2016MaximumTreeindexedRandom}. 

Actually, this result is quite general. It applies to the branching random walks studied here but also to \(T_{n}\)-walks
where \((T_{n})_{n}\) is any sequence of trees with size \(n+1\) that are \emph{at least \(D\)-dimensional} for some \(D >1\) in the sense of Maillard, \textit{i.e.} such that their height $H(T_n)$ satisfies the following assumption.

\begin{assumption}[Dimension of trees]\label{Assum:MaillardDimension}
\begin{itemize}
    \item There is $D>1$ such that
\begin{equation*}{
\forall \varepsilon>0, \P(H(T_n) \geq n^{\frac{1}{D} + \varepsilon}) \xrightarrow[n \to +\infty]{}0, \text{ or equivalently } \forall \varepsilon>0, \frac{H(T_n)}{n^{\frac{1}{D}+\varepsilon}} \xrightarrow[n \to +\infty]{(\P)} 0.
}\end{equation*}
\end{itemize}

\end{assumption}
We first give the general recurrence criterion which may be applied as soon as \Cref{Assum:MaillardDimension} holds.

\begin{theorem}[Recurrence criterion]\label{thm:RecurrenceCriterion}

Consider \(T_{n}\)-walks in \(\Z^{d}\), where \((T_{n})\) is any sequence of trees with size \(n+1\) satisfying
\Cref{Assum:MaillardDimension} for some $D>1$. Also assume that the centered jump distribution \(\theta\) has finite (absolute) moment of order $\delta$ for all $\delta < 2D$. As soon as \(d<2D\), the \(T_{n}\)-walks are recurrent:
\begin{equation*}{
\frac{R^{\theta}(T_{n})}{n} \xrightarrow[n \to +\infty]{(\P)} 0. 
}\end{equation*}

\end{theorem}

This enables us to answer
 positively to the question of recurrence in the $\alpha$-stable case.

\begin{corollary}[Recurrence of BRW in low dimension]\label{cor:RecLowDim}
Recall that \(\mu \in \mathrm{dom}(\alpha)\) for \(\alpha \in (1,2]\), and assume that $\theta$ has finite moment of order $\delta$ for all $\delta < \frac{2\alpha}{\alpha-1}$.
When \(d < \frac{2\alpha}{\alpha-1}\), a \(\Tree_{n}\)-walk in \(\Z^d\)
with $\theta$-jumps is recurrent:

\begin{equation*}{
\frac{R^{\theta}(\Tree_{n})}{n} \xrightarrow[n \to +\infty]{(\P)} 0.
}\end{equation*}
Consequently, with the same assumptions we have \(\P(\esc_{+})=\P(\esc)=0\).
\end{corollary}

As an illustration of the generality of \Cref{thm:RecurrenceCriterion}, we also apply it to
cover a limit case of stable critical Bienaymé trees where \(\alpha=1\).

\subsubsection{\texorpdfstring{The limiting case \(\alpha=1\)}{The limiting case \textbackslash alpha=1}}\label{the-limiting-case-alpha1}

In this subsection only, we make a different assumption on the
offspring distribution \(\mu\):

\begin{assumption}\label{Assum:Cauchy}
\begin{itemize}
  \item \(\mu\) is critical and such that \(n \mapsto \mu(n)\) is
\((-2)\)-varying.
\end{itemize}
\end{assumption}

This assumption implies that \(\mu\) is attracted to a spectrally
positive Cauchy distribution, which is stable of index \(\alpha=1\). In
this setting, the geometry of \(\Tree_{n}\) is quite different
(\emph{e.g.} \autocite{KortchemskiRichier2019CondensationCriticalCauchy}
revealed that a condensation phenomenon occurs, \emph{i.e.} for large
\(n\), \(\Tree_{n}\) looks like a giant star-shape tree) and much less
is known compared to the previous setting. Nonetheless Addario-Berry, Donderwinkel \& Kortchemski
\autocite{Addario-BerryDonderwinkelKortchemski2025CriticalTreesAre}
gave an estimate on the height of \(\Tree_{n}\) which indicates
that those trees are \(\infty\)-dimensional in the sense of \Cref{Assum:MaillardDimension}. With \Cref{thm:RecurrenceCriterion} in hands, we get

\begin{corollary}[No dichotomy in the Cauchy case]\label{cor:RecCauchy}

When $\mu$ satisfies \Cref{Assum:Cauchy} and $\theta$ has finite moment of order $\delta$ for all $\delta >0$, then a
\(\Tree_{n}\)-walk with $\theta$-jumps in \(\Z^{d}\) is
recurrent, whatever the dimension \(d\) is:

\begin{equation*}{
\frac{R^{\theta}(\Tree_{n})}{n} \xrightarrow[n \to +\infty]{(\P)} 0.
}\end{equation*}

\end{corollary}

However, note that \eqref{eq:LinkFiniteInfinite} is not established in this new setting,
hence \(\Tree_{\infty}\) can still be defined as previously but
\Cref{cor:RecCauchy} does not prove that a \(\Tree_{\infty}\)-walk is recurrent
in any dimension (yet, we expect this to be true).

\subsubsection{At the critical dimension}\label{at-the-critical-dimension}

Whenever \(\frac{2\alpha}{\alpha-1}\) is an integer, \emph{i.e.}
\(\alpha= \frac{m}{m-2}\) for some \(m\geq 4\), it still remains to
understand the behaviour of a \(\Tree_{n}\)-walk in
\(\Z^{\frac{2\alpha}{\alpha-1}}\). To do so, we combine Le Gall \& Lin's
criterion and our approach to bound from above \(R^{\theta}(\Tree_{n})\) with some
sharper results based on scaling limits of \(\Tree_{n}\)-walks. In
the Gaussian case \(\alpha=2\), we prove that even though the variance may be infinite, \(\Tree_{n}\)-walks are
still recurrent in \(\Z^{4}\).

\begin{proposition}[Recurrence at the critical dimension for $\alpha=2$]\label{prop:RecCriticalGaussian}

Assume that \(\alpha=2\), \emph{i.e.} \(\mu\) is attracted to a
Gaussian distribution, and that $\theta$, defined here on $\Z^4$, has a finite moment of order $\delta>4$. A \(\Tree_{n}\)-walk in \(\Z^4\) with $\theta$-jumps  is recurrent:

\begin{equation*}{
\frac{R^{\theta}(\Tree_{n})}{n} \xrightarrow[n \to +\infty]{(\P)} 0.
}\end{equation*}

\end{proposition}

\begin{remark}
Actually, we do not fully need the existence of a moment of order $\delta >4$ for \Cref{prop:RecCriticalGaussian} to hold. We require the slightly weaker but more technical \Cref{Assum:MomentAssumJumpMarzouk}, which ensures that $\Tree_n$-walks has the Brownian snake indexed by an $\alpha$-stable Lévy tree as a scaling limit according to \autocite[Theorem 1]{Marzouk2020ScalingLimitsDiscrete}.
\end{remark}

On the contrary, when \(\alpha <2\) and \(\frac{2\alpha}{\alpha-1}\) is
an integer, we do not have a complete classification but we know that
there is no universal behaviour anymore. Indeed, we have already mentioned that a critical 
offspring distribution $\mu \in \mathrm{dom}(\alpha)$ is characterized by \eqref{eq:GenMuStable}. In particular, $\mu$ is not fully characterized by $\alpha$, one must also consider a slowly varying function \(L\) to understand $\mu$ properly. We observe that depending on $L$, we may actually have recurrent or transient
\(\Tree_{n}\)-walk in \(\Z^{\frac{2\alpha}{\alpha-1}}\) !

\begin{assumption}[A sufficient condition for transience]\label{Assum:SufficientConditionTransience}
\begin{itemize}
  \item \(\mu\) satisfies \eqref{eq:GenMuStable} with a slowly varying function \(L\)
such that
\end{itemize}

\begin{equation*}{
 \sum_{n \geq 1} \frac{L(n^{\frac{2}{\alpha-1}})}{n} < +\infty.
 }\end{equation*}

\end{assumption}

\begin{proposition}[no universal behaviour at the critical dimension]\label{prop:NoUniversalityAtCriticality}

Assume that \(\alpha > 2 \text{ and } \frac{2\alpha}{\alpha-1} \in \N\). Also assume that $\theta$, defined here on $\Z^{\frac{2\alpha}{\alpha-1}}$, has a finite moment of order $\delta > \frac{2\alpha}{\alpha-1}$. Consider \(\Tree_{n}\)-walk in \(\Z^{\frac{2\alpha}{\alpha-1}}\) with $\theta$-jumps.

\begin{enumerate}[(i)]
  \item 
If \Cref{Assum:SufficientConditionTransience} holds, then this \(\Tree_{n}\)-walk is transient:
\begin{equation*}{
\frac{R^{\theta}(\Tree_{n})}{n} \xrightarrow[n \to +\infty]{(\P)} \P(\esc_{+}) > 0.
}\end{equation*}
  \item If \(L(n) \to +\infty\) then this \(\Tree_{n}\)-walk is
recurrent:
\begin{equation*}{
\frac{R^{\theta}(\Tree_{n})}{n} \xrightarrow[n \to +\infty]{(\P)}  0.
}\end{equation*}
\end{enumerate}

\end{proposition}

\begin{remark}
Here again, we may replace the moment assumption on $\theta$ by \Cref{Assum:MomentAssumJumpMarzouk}.
\end{remark}

We stress that \Cref{prop:NoUniversalityAtCriticality} does not cover all the possible situations,
such as when \(L(n)\) converges in \((0,+\infty)\), and we already know
from the case \(\alpha=2\) that the condition \(L(n) \to +\infty\) is
not necessary to get recurrence. To check whether the sufficient
condition \Cref{Assum:SufficientConditionTransience} is necessary, one would need precise estimates
on the visiting probability of a position \(x\) by a \(\Tree\)-walk
starting from \(0\). However, as we will see in the next section with
the notion of branching capacity, we were only able to get such
estimates at the critical dimension \(d=\frac{2\alpha}{\alpha-1}\) when
we already assume \Cref{Assum:SufficientConditionTransience}.

\subsection{Branching capacity for transient stable BRW}\label{capacity-of-transient-stable-brw}

Zhu \autocite{Zhu2017CriticalBranchingRandom} pursued the study of
\(\Tree\)-indexed walk in the transient setting \(d > 4\) when \(\mu\)
has a finite variance \(\sigma_{\mu}^{2} < +\infty\). In particular, he introduced
the notion of branching capacity \(\bcp(A)\) of a finite set
\(A\) and revealed that in some sense, a transient \(\Tree\)-walk behaves
just as a classic random walk \((S_{n})_{n \in \N}\) with the same jump
distribution \(\theta\). One must simply replace the notion of
(Newtonian) capacity associated with \((S_{n})_{n \in \N}\) by the
branching capacity associated with the \(\Tree\)-walk. This notion of branching capacity has been further studied in \autocite{Schapira2023BranchingCapacityRandom,BaiDelmasHu2024BranchingCapacityRandom,Baran2025IntersectionsBranchingRandom,BaiDelmasHu2024BranchingCapacityBrownian,AsselahSchapiraSousi2025DerivativeFormulaCapacities,AsselahOkadaSchapiraSousi2023BranchingRandomWalks} and has application to local times \autocite{AsselahSchapiraSousi2025LocalTimesCapacity,AsselahSchapira2024TimeSpentBall} as well as to a branching variant of Sznitman's random interlacements \autocite{Zhu2018BranchingInterlacementsTreeindexed, Schapira2025NonTrivialityPercolation,ProcacciaZhang2016ConnectivityPropertiesBranching,Vanhaelewyn2026MultipointConnectivityBranching}. Here our extension to the $\alpha$-stable case is essentially based on Zhu's work, but we expect that several results of this literature could be extended as well to this setting.

The first step of this extension is to observe that the definition of
branching capacity actually also makes sense when \(\mu \in \mathrm{dom}(\alpha)\) (where \(\alpha \in (1,2]\)) as soon
as \(d> \frac{2\alpha}{\alpha-1}\), or \(d=\frac{2\alpha}{\alpha-1}\)
and \Cref{Assum:SufficientConditionTransience} holds to ensure that we are in the transient case. In this section we show that in the same sense as earlier a
transient \(\Tree\)-walk behaves as a transient random walk
\((S_{n})_{n \in \N}\). We also show that \(\bcp\) is useful to understand
visiting probabilities for a \(\Tree_{\infty}\)-walk. Finally, we give some estimates on \(\bcp(A)\) for
\(A\) a large ball which stress that the geometry of \(\Tree\) does affect
\(\bcp\).

To make a clear analogy with a transient random walk \((S_{n})_{n}\)
with the same jump distribution \(\theta\), let us give some basic facts
on the classical notion of Newtonian capacity. We refer to
\autocite[Section~6.5]{LawlerLimic2010RandomWalkModern} for more details
in the case of finite-range random walks.

Consider a finite subset \(A \subset \Z^{d}\), and define
\(H_{A} = \inf\{ n \geq 0 : S_{n} \in A\}\) and
\(H_{A}^{+}=\inf \{  n \geq  1 : S_{n} \in A\}\). The capacity of \(A\)
(which implicitly depends on \(\theta\)) is defined as

\begin{equation}\label{eq:newtonianCapacity}{
\mathrm{Cap}(A)=\sum_{a \in A}\P_{a}(H_{A}^{+} = +\infty),
}\end{equation}
where \(\P_{a}\) indicates that the walk starts from \(a\). This quantity
is positive as soon as \((S_{n})_{n}\) is transient, \emph{i.e.}
\(d > 2\). Moreover, it is well known that, under some finite moment
assumption on \(\theta\), the Green function \(g_{\theta}(y,x) = \sum_{n \in \N} \P_y(S_n =x)\) associated
to \((S_{n})_{n}\) satisfies

\begin{equation}\label{eq:GreenEstimates}{
g_{\theta}(x):= g_{\theta}(0,x) \sim_{x \to \infty} \frac{C_{\theta}}{\lVert x \rVert^{d-2} },
}\end{equation}
where
\(C_{\theta} = \frac{\Gamma\left(\frac{d-2}{2} \right)}{2 \sqrt{ \det (\pi\Sigma_{\theta}^{2}) }}\)
and \(\lVert  \cdot \rVert\) is defined from the Euclidian norm
\(\lvert \cdot \rvert\) by
\(\lVert x \rVert = \lvert \Sigma_{\theta}^{-1}x \rvert\) (see for
instance \autocite{LawlerLimic2010RandomWalkModern} where \(\theta\) is
assumed to have a finite moment of order \(d\), or \autocite[Theorem
2]{Uchiyama1998GreensFunctionsRandom} where a finite moment of order
\(d-2\) is required). In the following, we will simply keep the moment
assumption made by Le Gall \& Lin in \autocite{LeGallLin2016RangeTreeindexedRandom} to
ensure transience, namely that \(\theta\) has a finite moment of order
\(d-1\), which is sufficient to have \eqref{eq:GreenEstimates}. When \eqref{eq:GreenEstimates}
holds, \(\mathrm{Cap}(A)\) can be interpreted as the size of \(A\) seen
from a random walker in \(\Z^{d}\) starting far away from the origin
thanks to the following result:

\begin{equation}\label{eq:newtonianCapacityInterpretation}{
\lim_{ x \to \infty } \frac{\P_{x}(H_{A} < +\infty)}{g_{\theta}(x)} = \mathrm{Cap}(A).
}\end{equation}

\subsubsection{Branching capacity as a measure of the visiting probability}\label{branching-capacity-as-a-measure-of-the-visiting-probability}

In the context of \(\Tree\)-walk where $\mu \in \mathrm{dom}(\alpha)$, as explained above the notion of
branching capacity of a finite set \(A \subset \Z^{d}\) defined by Zhu still makes sense and we keep
the same definition. It is a branching equivalent of \eqref{eq:newtonianCapacity} where
the escape event \((H_{A}^{+}=+\infty)\) from \(a \in A\) is replaced by
the event of a \(\Tree_{\infty}\)-walk \(W_{\Tree_{\infty}}^{a}\) (where
the superscript \(a\) indicates the starting position
\(W_{\Tree_{\infty}}^{a}(c_0)=a\)) never coming back to \(A\) in
the future \(\Tree_{+}\). We adopt the notation of
\autocite{AsselahSchapiraSousi2025LocalTimesCapacity}: for any
\(\mathcal{V} \subset \Tree_{\infty}\) we set
\(\mathcal{V}^{x} = \{ W_{\Tree_{\infty}}^{x}(u), u \in \mathcal{V} \}\). In a similar way, for a \(T\)-walk \( W_{T}^{x} \), we set \(T^{x} = \{ W_{T}^{x}(u), u \in T \}\).

\begin{definition}[Branching capacity from escape probabilities]\label{def:BranchingCapacity}

The branching capacity (associated with a \(\Tree\)-walk) of a finite
subset \(A \subset \Z^d\) is defined as

\begin{equation*}{
\bcp(A) = \sum_{a \in A} \P(\Tree_{+}^{a} \cap A = \emptyset).
}\end{equation*}
It implicitly depends on \(\mu\) and \(\theta\).

\end{definition}

From this definition, one can argue as in \autocite[Lemma~2.4]{Schapira2025NonTrivialityPercolation} to get that
\(A \mapsto \bcp(A)\) is still increasing in \(A\) in this
\(\alpha\)-stable setting. As it is also clearly invariant by
translation of \(A\), this branching capacity behaves as a volume, as
long as it is non-zero. This last point is satisfied if and only if the
corresponding \(\Tree_{\infty}\)-walk is transient \emph{i.e.}
\(\P(\esc_{+})>0\), since we clearly have for all \(A \subset \Z^{d}\),
\(\P(\esc_{+})=\bcp(\{ 0 \}) \leq  \bcp(A) \leq  \#A\P(\esc_{+})\). Thus the branching capacity may give us some useful information
about \(\Tree\)-walks only in the transient case, when
\(d > \frac{2\alpha}{\alpha-1}\) or \(d=\frac{2\alpha}{\alpha-1}\) and
\Cref{Assum:SufficientConditionTransience} holds.

Assuming that \(\mu\) has a finite variance,
Zhu established a branching
equivalent of \eqref{eq:newtonianCapacityInterpretation} which gives an interpretation of \(\bcp(A)\)
as the size of \(A\) seen from a \emph{branching walker} in
\(\Z^{d}\) starting far away from the origin. At the same time it tells
us that in this transient setting a branching walker and a classical
walker starting from \(0\) have roughly the same probability of reaching
some distant position \(x \in \Z^{d}\). This actually holds as soon as
$\mu \in \mathrm{dom}(\alpha)$, and we also slightly relax the moment assumption on $\theta$.

\begin{theorem}[Branching capacity and asymptotics of visiting probability]\label{thm:BranchingCapacityVisitingProba}

Consider a \(\Tree\)-walk in \(\Z^{d}\), and assume that 
\begin{enumerate}[(i)]
  \item \(d> \frac{2\alpha}{\alpha-1}\), or \(d=\frac{2\alpha}{\alpha-1}\) and
\Cref{Assum:SufficientConditionTransience} holds;
\item the jump distribution \(\theta\) has a finite
moment of order \(d-1\);
\end{enumerate}

Then for all finite \(A \subset \Z^d\) we have

\begin{equation*}{
 \P(\Tree^{x} \cap A \neq \emptyset) \sim_{x \to \infty} \bcp(A)g_{\theta}(x).
}\end{equation*}

\end{theorem}

Just as Zhu did in \autocite{Zhu2017CriticalBranchingRandom}, we actually prove
a more precise result, \Cref{prop:AsymptoticsVisitingProba}, which also takes into account the
first and the last vertex of \(\Tree\) (for the depth-first order) that
lies in \(A\). Considering the last one will be linked to \(\Tree_{+}\),
while taking into account the first vertex that lies in \(A\) will be
linked to \(\Tree_{-}\) instead. As a consequence we get that one
can also defines \(\bcp(A)\) with \(\Tree_{-}\) instead of \(\Tree_{+}\)
in \Cref{def:BranchingCapacity}.

\begin{proposition}[branching capacity defined by the past]\label{prop:BranchingCapacityFromPast}

For all finite sets \(A \subset \Z^{d}\),

\begin{equation*}{
\bcp(A):=\sum_{a \in A} \P(\Tree_{+}^{a} \cap A = \emptyset) = \sum_{a \in A} \P(\Tree_{-}^{a} \cap A = \emptyset).
}\end{equation*}

\end{proposition}

The proof of \Cref{prop:AsymptoticsVisitingProba}, which directly gives \Cref{thm:BranchingCapacityVisitingProba} and
\Cref{prop:BranchingCapacityFromPast}, is still based on
Zhu's crucial observation that $\P(\Tree^{x} \cap A \neq \emptyset)$ may be rewritten as the Green
function of a classical random walk with some killing rate
related to a special tree-indexed walk. The fact that \(\mu\) no longer
has a finite variance mainly impacts this killing rate and thus the
comparison between this special Green function and \(g_{\theta}\). While we believe that Zhu's method could still work in this setting, we propose a seemingly more straightforward way to compare the special Green function and \(g_{\theta}\) (see \Cref{lem:GreenFunctionApproximation} and its proof). 

\subsubsection{Branching capacity and infinite tree indexed walk}\label{branching-capacity-and-infinite-tree-indexed-walk}

We can also link the branching capacity to visiting probabilities of a
\(\Tree_{\infty}\)-walk. First we introduce

\begin{equation}\label{eq:AlternativeGreenFunction}{
G(x) = \E(\#\{ k \geq 1: \exists u \in \mathcal{ B}_{k,-} \text{ such that } W_{\Tree_{\infty}}^{0}(u) = x \}),
}\end{equation}
where \(\mathcal{B}_{k,-}\) is the finite subtree (called a \emph{bush})
of \(\Tree_{-}\) grafted to the \(k^{\text{th}}\) vertex \(c_{-k}\) on
the spine of \(\Tree_{\infty}\). Our aim is to replace the more natural notion
of Green function of a \(\Tree_{-}\)-indexed random walk
\(\E(\# \{  u \in \Tree_{-}: W_{\Tree_{\infty}}^{0}(u)=x \})\), as this
last expectation is infinite as soon as \(\sigma_{\mu}^{2} =+\infty\).
From \Cref{thm:BranchingCapacityVisitingProba}, we can get a precise estimate for \(G\), which still
depends on \(\mu\) through the slowly varying function \(L\).

\begin{proposition}\label{prop:AlternativeGreenFunction}

Assume that \(d> \frac{2\alpha}{\alpha-1}\) and \(\theta\) has a finite
moment of order \(d-1\), then we have

\begin{equation*}{
G(x) \sim_{x \to \infty} C_{\mu,\theta} \lVert x \rVert^{d} g_{\theta}(x)^{\alpha}L(g_{\theta}(x)^{-1}),
}\end{equation*}
where \(C_{\mu,\theta} = \bcp(\{ 0 \})^{\alpha-1}\times  \int_{\R^d} \frac{1}{\Vert z\Vert^{d-2}} \frac{1}{\Vert u-z\Vert^{(\alpha-1)(d-2)}}\dd z\)
with \(u\) any unitary vector (for \(\lVert \cdot \rVert\)). In particular, \(G\) is \((2\alpha-d(\alpha-1))\)-varying.

\end{proposition}

By decomposing \(\Tree_{-}\) and re-using the Green function
approximation needed for \Cref{prop:AsymptoticsVisitingProba}, it turns out that \(G\) gives
the correct asymptotic for the visiting probability of a
\(\Tree_{\infty}\)-walk, as stated in the following result.

\begin{theorem}\label{thm:VisitingProbabilityInfinite}

With the assumptions of \Cref{prop:AlternativeGreenFunction}, for all finite \(A \subset \Z^d\),
as \(\lVert x \rVert\to +\infty\) we have

\begin{equation*}{
 \P(\Tree_{+}^{x} \cap A \neq \emptyset) \sim \P(\Tree_{-}^{x} \cap A \neq \emptyset) \sim \left(\frac{\bcp(A)}{\bcp(\{ 0 \})}\right)^{\alpha-1} G(x).
}\end{equation*}

\end{theorem}

\subsubsection{\texorpdfstring{\(\alpha\)-stable branching capacity of balls}{\textbackslash alpha-stable branching capacity of balls}}\label{alpha-stable-branching-capacity-of-balls}

Finally, for practical uses of the branching capacity, one needs some
estimates on \(\bcp(A)\) for sets \(A\) of interest. In
\autocite{Zhu2017CriticalBranchingRandom}, Zhu gave such estimates for large
\(m\)-dimensional ball into \(\Z^{d}\) when \(\mu\) has a finite
variance, however his approach relies on a second-moment method which 
cannot be extended to the case \(\sigma_{\mu}^{2} = +\infty\). We propose a
different approach to get some estimates, and based on the precise
estimate \eqref{eq:QueueMaxStable} of the tail of the maximum of a one-dimensional
\(\Tree\)-walk we are able to get the order of the branching capacity of
a large \(d\)-dimensional ball in \(\Z^{d}\) for \(d>\frac{2\alpha}{\alpha-1}\). Interestingly, we also get that the branching capacity of this ball may simply be expressed as the Newtonian
capacity \(\mathrm{Cap}(B(0,R))\), but with a penalization coming from the
underlying tree given by \(\P(H(\Tree)>R^{2})\).

Note that the order of \(\mathrm{Cap}(B(0,R))\) is well known (see \emph{e.g.} \autocite[Proposition
6.5.2]{LawlerLimic2010RandomWalkModern}):

\begin{equation}\label{eq:NewtonianCapBall}{
\mathrm{Cap}(B(0,R)) \asymp R^{d-2},
}\end{equation}
and the asymptotics of
\(n \mapsto \P(H(\Tree)>n)\) has been established by Slack
\autocite{Slack1968BranchingProcessMean} and will be recalled at \eqref{eq:SlackResult}.

\begin{theorem}\label{thm:BranchingCapBall}

Consider a \(\Tree\)-walk in \(\Z^{d}\). We assume that there is
\(\delta > d\) such that \(\theta\) has a finite moment of order
\(\delta\). Then for \(d> \frac{2\alpha}{\alpha-1}\), we have
\begin{equation*}{
\bcp(B(0,R)) \asymp \mathrm{Cap}(B(0,R))\times  \P(H(\Tree) > R^{2}).
}\end{equation*}

\end{theorem}
We stress that the right-hand side of \Cref{thm:BranchingCapBall} is bounded from above and below by some
\(\left(d-\frac{2\alpha}{\alpha-1} \right)\)-varying functions of \(R\). 

This points out that in the $\alpha$-stable setting, $\bcp$ could be comparable to a Bessel-Riesz capacity (see \autocite{AsselahSchapiraSousi2025DerivativeFormulaCapacities} for a definition) with a kernel being a \(\left(\frac{2\alpha}{\alpha-1}-d \right)\)-varying function, however such result would require more than the estimate of $\bcp$ for balls.

\subsection{Outline}\label{outline}

In \Cref{branching-random-walk-and-invariant-trees}, we first recall some
facts about \(\alpha\)-stable BRW and the associated invariant
walk, and then we prove that there is a
dichotomy recurrence/transience for this invariant walk. \Cref{recurrence-and-transience-of-stable-brw-1} is dedicated to the recurrence criterion in low dimension and its applications, including the case of critical dimension. Finally, \Cref{capacity-of-transient-stable-brw-1} deals with the study of
branching capacity in transient dimension. We first establish the relations between branching capacity and visiting probabilities, and then we prove the estimate for branching capacity of balls.
We also provide \Cref{annexe-regularly-varying-functions} to recall all
the properties of regularly varying function required in this paper.

\begin{table}[htbp]
\caption{The mostly standard notation used in this paper is recalled here.}
\centering
\begin{tabular}{c p{10cm}}
\toprule
$\N = \{0,1,\ldots\}$ \text{ and } $\N^*=\N\!\setminus\!\{0\}$ & sets of natural integers and positive integers;\\
  $\lvert x \rvert$ & Euclidian norm (or absolute value) of $x \in \Z^d$;\\
  $\lVert x \rVert = \lvert \Sigma_{\theta}^{-1}x \rvert $ & norm adapted to walks with $\theta$-jumps;\\
 \(C,C',c,c',\ldots\) & generic constants with respect to some parameters which will be specified when they are not immediately clear from the context;\\
 \(f(x) \sim g(x)\) & asymptotic equivalence, which essentially means \(f(x)/g(x) \to 1\) as \(x\) goes to a limit point. if unspecified, this limit point is $+\infty$;\\
 \(f(x) \preceq g(x)\) & $g$ dominates $f$, \textit{i.e.} there is $C>0$ such that $f(x) \leq Cg(x)$ for large $x$;\\
 \(f(x) \asymp g(x)\) & $f$ and $g$ have the same order, \textit{i.e.} \(f(x) \preceq g(x)\) and \(g(x) \preceq f(x)\);\\
 $\xrightarrow[]{(\P)}$ & convergence in probability;\\
  $\xrightarrow[]{(d)}$ & convergence in distribution;\\
 \bottomrule
\end{tabular}
\label{tab:standard notation}
\end{table}

\paragraph{Acknowledgement.} I thank Igor Kortchemski and Bruno Schapira for their careful readings and their help for several computations, which notably shortened the proof of \Cref{lem:GreenFunctionApproximation}. I also owe many thanks to Perla Sousi for sharing her lecture note on Zhu's branching capacity in the finite variance setup.

\section{Branching random walk and invariant trees}\label{branching-random-walk-and-invariant-trees}

\subsection{Framework}\label{framework}

\subsubsection{Trees and tree-indexed walks}\label{tree-indexed-walk-and-brw}

In this paper, trees are ordered and rooted. A classical formalism for such trees is to see them as subtrees of Ulam-Harris' tree, see \textit{e.g.} \cite{LeGall2005RandomTreesApplications} for more details, however we will need a slight extension of this framework. Ulam-Harris' formalism is well suited to define a $\mu$-Bienaymé tree as a random tree such that every $u \in \Tree$ has a number of children distributed according to $\mu$ and independent of everything else. We can also define its size-conditioned version $\Tree_{n}$ as $\Tree$ conditioned to have $n+1$ vertices. However, Ulam-Harris' formalism does not fit our point of view for the invariant tree $\Tree_{\infty}$, because in this tree the distinguished vertex is not the common ancestor of the whole genealogy represented by this tree. A suitable formalism for such point of view is provided by Stufler in \cite{Stufler2019LocalLimitsLarge}, but here we choose to make an informal definition rather than introducing the whole formalism. 

As depicted in \Cref{fig:pointedtree}, we see a tree $T$ as a \emph{pointed genealogical tree}. First, the edges of the tree are oriented (from a parent to a child) and there is at most one vertex in $T$ with in-degree~$0$ while all other vertices of $T$ have in-degree~$1$. The vertex with in-degree~$0$, if it exists, is the root, \textit{i.e.} the common ancestor of the genealogy. Then, our tree is ordered in the sense that every vertex $u \in T$  has a first child, a second child, $\dots$ until its last child (if any). Finally, we say that $T$ is a pointed tree when it has a distinguished vertex, called its \textit{cursor} and denoted by $c_0$. When we consider a classical plane tree such as $\Tree$ or $\Tree_n$ without choosing a distinguished point, by convention we may see them as pointed genealogical tree by setting the cursor to be the root \textit{i.e.} $c_0=\varnothing$.

\begin{figure}[h!]

\begin{center}
\begin{tikzpicture}[scale=0.1]
\tikzstyle{every node}+=[inner sep=0pt]
\draw [black] (36.7,-18.2) circle (3);
\draw (36.7,-18.2) node {\(c_0\)};

\draw [black] (36.7,-31.3) circle (3);

\draw [black] (36.7,-45.3) circle (3);

\draw [black] (30,-5.8) circle (3);
\draw [black] (44.2,-5.8) circle (3);
\draw [black] (23,-18.3) circle (3);
\draw [black] (48.9,-18.6) circle (3);
\draw [black] (60.5,-18.7) circle (3);
\draw [black] (48.9,-31.3) circle (3);
\draw [black] (60.5,-5.8) circle (3);
\draw [black] (36.7,-42.3) -- (36.7,-34.3);
\fill [black] (36.7,-34.3) -- (36.2,-35.1) -- (37.2,-35.1);
\draw [black] (36.7,-28.3) -- (36.7,-21.2);
\fill [black] (36.7,-21.2) -- (36.2,-22) -- (37.2,-22);
\draw [black] (35.27,-15.56) -- (31.43,-8.44);
\fill [black] (31.43,-8.44) -- (31.37,-9.38) -- (32.25,-8.91);
\draw [black] (38.25,-15.63) -- (42.65,-8.37);
\fill [black] (42.65,-8.37) -- (41.81,-8.79) -- (42.66,-9.31);
\draw [black] (34.52,-29.24) -- (25.18,-20.36);
\fill [black] (25.18,-20.36) -- (25.41,-21.28) -- (26.1,-20.55);
\draw [black] (38.67,-43.04) -- (46.93,-33.56);
\fill [black] (46.93,-33.56) -- (46.03,-33.84) -- (46.78,-34.49);
\draw [black] (38.78,-29.14) -- (46.82,-20.76);
\fill [black] (46.82,-20.76) -- (45.91,-20.99) -- (46.63,-21.69);
\draw [black] (39.35,-29.9) -- (57.85,-20.1);
\fill [black] (57.85,-20.1) -- (56.91,-20.04) -- (57.38,-20.92);
\draw [black] (60.5,-15.7) -- (60.5,-8.8);
\fill [black] (60.5,-8.8) -- (60,-9.6) -- (61,-9.6);
\end{tikzpicture}

\caption{A pointed genealogical tree \(T\), where the directed edges indicate the genealogy, the relative positions of siblings indicate their order and \(c_0\) indicates a distinguished vertex (which is not the common ancestor of the whole tree in this example).}
\label{fig:pointedtree}
\end{center}
\end{figure}
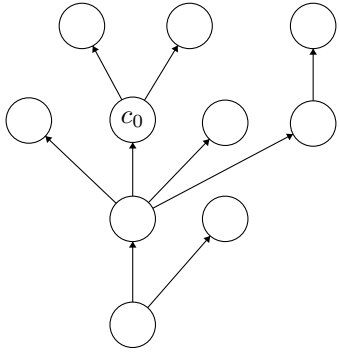

When $T$ is finite, it must have a root $\varnothing$ and its height is $H(T)=\max_{u \in T} d_{T}(\varnothing,u)$ where $d_{T}$ is the graph distance in the tree. An infinite tree $T$ may not have a root, as the cursor may have an infinite line of ancestors.
\par

Given a (pointed genealogical) tree $T$, a position $x \in \Z^{d}$ and a distribution $\theta$ on $\Z^{d}$, we denote by $W_{T}^{x}=\bigl(W_{T}^{x}(u)\bigr)_{u \in T}$ a random walk indexed by $T$, simply called a $T$-walk in the following, starting at $x$ and jumping according to $\theta$ distribution. Note that the starting position is the position of the cursor $c_0$. This $T$-walk is formally defined by

\begin{enumerate}[(i)]
  \item $W_{T}^{x}(c_0) = x$
  \item For all edges in $T$, say from a parent $v$ to a child $u$, the jump $W_{T}^{x}(u)-W_{T}^{x}(v)$ is distributed according to $\theta$ and independent of the jumps along the other edges.

\end{enumerate}

We will drop the superscript $x$ for $x=0$, but most of the time the walk itself (and its jump distribution $\theta$) will be implicit as we will write $T^x=\{ W_{T}^{x}(u), u \in T \} \subset \Z^{d}$ for the set of position visited by a $T$-walk starting from $x$. We also define the range of a $T$-walk, \textit{i.e.} the number of position visited in $\Z^d$, as $R^{\theta}(T)=\#T^{x}$  (it depends on $\theta$ but not on $x$).
\par

Let us recall that for a $\mu$-Bienaymé tree $\Tree$, a $\Tree$-walk is also called a Branching Random Walk (BRW). This comes from the branching property of $\Tree$, which enables to sample a BRW in a dynamical way instead of sampling the entire tree $\Tree$ and then the jumps. This does not apply to $\Tree_{n}$-walks and $\Tree_{\infty}$-walks but we may sometimes refer to them as BRW by extension.

\subsubsection{\texorpdfstring{Properties of \(\mu\)-Bienaymé trees and BRW}{Properties of \textbackslash mu-Bienaymé trees}}\label{properties-of-mu-bienayme-trees}

We gather here some known results that will turns out to be useful in the rest of this paper. Recall that $\mu \in \mathrm{dom}(\alpha)$ and as discussed in \Cref{offspring-attracted-to-an-alpha-stable-distribution} it means that despite their possibly infinite variance, \(\mu\)-distributed random variables verify a generalization of the Central Limit Theorem (see \eqref{eq:StableCLT}). We will mostly use that $\mu$ is characterized by \eqref{eq:GenMuStable}.

First, we will need an estimate on the height
\(H(\Tree)=\max_{u \in \Tree}\lvert u \rvert\) of the \(\mu\)-Bienaymé
tree \(\Tree\). The following result, due to Slack \autocite[Lemma
2]{Slack1968BranchingProcessMean}, generalizes Kolmogorov's estimate.
Assuming that $\mu \in \mathrm{dom}(\alpha)$ with \(\alpha \in (1,2]\), we have

\begin{equation}\label{eq:SlackResult}
  \mathbb P(H(\Tree) \geq n)^{\alpha-1}L\big(\mathbb P(H(\Tree) \geq n)^{-1}\big) \sim \frac1{\alpha-1}\frac1n,
\end{equation}
where \(L\) is the slowly varying function from \eqref{eq:GenMuStable}. In
particular, 
\[n \mapsto \P(H(\Tree)>n) \text{ is } \left(-\frac{1}{\alpha-1} \right)\text{-varying.}\]

This estimate on $H(\Tree)$ is also crucial to understand the maximal displacement of a $(\mu,\theta)$-BRW $W_{\Tree}$ in $\Z$ starting at $0$. Indeed, we know from \autocite{Subordination} that when $\theta$ has a finite variance $\sigma_{\theta}$ and a finite moment of order $\delta > \frac{2\alpha}{\alpha-1}$, the tail of $\max W_{\Tree}$ is given by

\begin{equation}\label{eq:QueueMaxStable}
{\mathbb P\left(  \frac{1}{\sigma_{\theta}}\max W_{\Tree} \geq  n \right) \sim_{n \to +\infty} \left( \frac{\alpha+1}{\alpha-1} \right)^{\frac{1}{\alpha-1}} \mathbb P(H(\Tree) > n^2).
}\end{equation} 

Finally, we will make use of some scaling limit results on the height of \(\Tree_{n}\).
When \(\mu\) has finite variance, it is known since Aldous
\autocite{Aldous1991ContinuumRandomTree} that
\(\frac{1}{\sqrt{  n }}\Tree_{n}\) (\emph{i.e.} \(\Tree_{n}\) where each
edge has length \(\frac{1}{\sqrt{ n }}\)) converges in distribution to a
continuum random tree, namely the Brownian continuum tree, and a similar
result for the \(\alpha\)-stable case was obtained by Duquesne
\autocite{Duquesne2003LimitTheoremContour} (see also
\autocite{Kortchemski2013SimpleProofDuquesnes}). To make this
convergence formal, one can see random trees as random compact metric
spaces and use the Gromov-Hausdorff distance (see \emph{e.g.}
\autocite{LeGall2005RandomTreesApplications} for a quick introduction).
However, in this paper we only need the following estimate on
\(H(\Tree_{n})\) which does not require such formalism. Assuming that
$\mu \in \mathrm{dom}(\alpha)$ with
\(\alpha \in (1,2]\), we have
\begin{equation}\label{eq:HeightEstimate}
  \frac{\ell(n)}{n^{\frac{\alpha-1}{\alpha}}} \times H(\Tree_{n}) \text{ converges in distribution,}
\end{equation}
where \(\ell\) is a slowly varying function (see \eqref{eq:StableCLT} for a characterization of \(\ell\)). In particular, this means that $\Tree_n$ are $\frac{\alpha}{\alpha-1}$-dimensional in the sense that the maximal $D>1$ for which they satisfy \Cref{Assum:MaillardDimension} is $D = \frac{\alpha}{\alpha-1}$ .

\subsection{Invariant trees and walks}\label{invariant-trees-and-walks}

As explained in
\Cref{the-infinite-walk-associated-to-branching-random-walk}, there are
several slightly different versions of the invariant tree introduced by
Le Gall \& Lin. Here the one we work
with, denoted by \(\Tree_{\infty}\), is the same as in Asselah, Schapira \& Sousi
\autocite{AsselahSchapiraSousi2025LocalTimesCapacity}.\footnote{To avoid confusion, note that in  \autocite{AsselahSchapiraSousi2025LocalTimesCapacity} the invariant is
denoted by \(\Tree\) and \(\mu\)-Bienaymé trees are denoted by \(\Tree_c\). Moreover, the distinguished vertex called the cursor of the invariant tree here is called its root in \autocite{AsselahSchapiraSousi2025LocalTimesCapacity}.} However, we see it with the point of view of pointed genealogical tree to provide a more natural interpretation of
a \(\Tree_{\infty}\)-walk. Indeed, we have already said that this tree can be seen as a local limit of \(\Tree_{n}\) \emph{pointed at a uniform
vertex but without modifying the genealogy}. Such results were first
obtained by Aldous \autocite{Aldous1991AsymptoticFringeDistributions}, so
\(\Tree_{\infty}\) is sometimes referred to as Aldous' sin-tree (which
enables to distinguish it from the local limit of \(\Tree_{n}\) obtained when looking around the root, known as Kesten's tree). In our setting of pointed genealogical trees
this convergence has been precisely established by Stufler
\autocite{Stufler2019LocalLimitsLarge}. We do not detail further the convergence toward $\Tree_{\infty}$ and directly describe this tree based on the spine \(\ldots \to c_{-2} \to c_{-1} \to c_{0}\) introduced in \Cref{the-infinite-walk-associated-to-branching-random-walk}.

\subsubsection{Sampling the invariant tree}\label{sampling-the-invariant-tree}

To sample \(\Tree_{\infty}\), we first need $\mu'$ the size-biased version of $\mu$. We recall that $\mu'$ is defined by \(\mu'(k)=k\mu(k)\) and its generating function is given by
\begin{equation}\label{eq:GenSizeBias}
    \gen_{\mu'}(x) = x\gen_{\mu}'(x).
\end{equation}

\begin{itemize}
\tightlist
\item
  Consider a semi-infinite line of vertices
  \((c_{0},c_{-1},c_{-2},\dots)\), called \emph{the spine of
  the tree} and denoted by \(\mathcal{S}\). Here $c_0$ is the cursor and for all $i \geq 0$ we see \(c_{-i}\) as the
  child of \(c_{-i-1}\), hence this spine is growing up from infinity.
\item
  The strict ancestors of \(c_{0}\), namely \(c_{-k}\) for
  all \(k \geq 1\), reproduce according to the size-biased distribution \(\mu'\) and \(c_{-k+1}\) is selected uniformly among
  those children. 
\item
  All other vertices, including \(c_{0}\), reproduce according to
  \(\mu\).
\end{itemize}

Note that with our convention, the out-degree of the cursor
\(c_{0}\) is distributed according to \(\mu\) but one must
add \(1\) to get its total degree since \(c_{0}\) has a
parent. 

As each vertex among $c_0$ and the children of the spine generates $\mu$-Bienaymé trees, which are finite, we can also describe $\Tree_{\infty}$ as a spine with finite trees grafted to every vertex of this spine, as depicted in \Cref{fig:InvariantTree}.

\subsubsection{Structure of $\Tree_{\infty}$}\label{future-and-past-of-the-root}

The genealogical structure defines a partial order on
\(\Tree_{\infty}\): \(u \to v\) if \(u\) is an ancestor of \(v\).
Two vertices \(u,v\) may not be comparable for \(\to\), but one can
always consider their last common ancestor \(u \wedge v\).

Since \(\Tree_{\infty}\) is ordered, this tree also comes with a total
order denoted by $\preceq$ on its vertices, namely the lexicographical (or depth-first) order
starting from infinity. Informally, this order coincides with a
depth-first exploration of the tree starting from the root
lying at infinity. A precise definition could be that
\(u \preceq v\) if and only if either \(u \to v\) or the
children \(u_{*}\) and \(v_{*}\) of \(u \wedge v\) such that \(u\)
(resp. \(v\)) descends from \(u_{*}\) (resp. \(v_{*}\)) satisfy
\(u_{*} \leq  v_{*}\) according to the order on the children of
\(u \wedge v\).

We can now consider two important subparts of \(\Tree_{\infty}\) depicted in \Cref{fig:futureAndPast} :

\begin{itemize}
\item
  the past of the cursor \(c_{0}\) is
  \(\Tree_{-} =\{ u \in \Tree_{\infty}: u \prec c_{0}\}\),
\item
  its future
  \(\Tree_{+}=\{  u \in \Tree_{\infty}: c_{0} \prec u \}\).
\end{itemize}

In words, \(\Tree_{-}\) consists of the spine
\(\mathcal{S}\setminus \{ c_0 \}\) together with the finite
trees grafted on the left of the spine. Following Zhu, we call those
finite trees \emph{past bushes} and we denote them by
\(\mathcal{B}_{k,-}=\{ u \in \Tree_{\infty}: c_{-k} \preceq u \prec c_{-k+1} \}\)
for \(k\geq 1\). On the contrary, \(\Tree_{+}\) does not contain the
spine \(\mathcal{S}\), but only consists of
\(\mathcal{B}_{0}\setminus \{ c_0 \}\) where \(\mathcal{B}_{0}\)
is the bush of descendants of \(c_0\), together with the
\emph{future bushes}, namely the bushes grafted on the right of the spine
and denoted by
\(\mathcal{B}_{k,+}= \{ u \in \Tree_{\infty}: u \wedge c_{-k+1} = c_{-k} \text{ and } c_{-k+1} \prec u \}\). Compare \Cref{fig:futureAndPast,fig:InvariantTree}.
We see that \(\Tree_{+}\) is actually not a tree, but \(\Tree_{+} \cup \mathcal{S}\) is and this is precisely the invariant tree
\(\Tree_{\mathrm{inv}}\) considered in
\autocite{LeGallLin2016RangeTreeindexedRandom}.

\begin{figure}[h]
  \begin{center}
  \includegraphics[height=9cm]{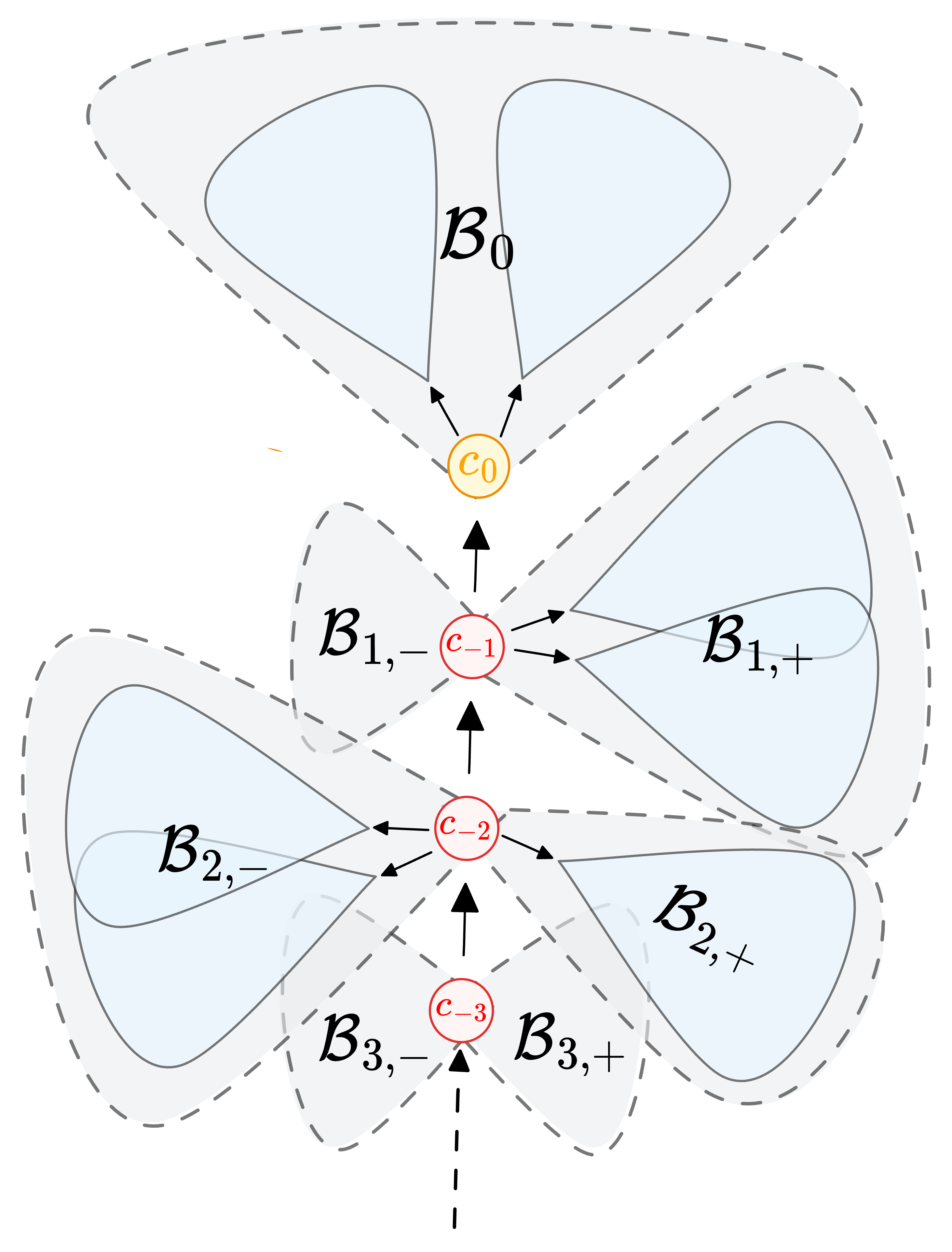}
  \caption{\centering A realization of Aldous' sin-tree \(\Tree_{\infty}\). Edges still indicate the genealogical order. On the spine, the red vertices $c_{-k}\ (k \geq 1)$  reproduce according to \(\mu'\) and the next vertex on the spine is uniformly chosen among their children. The cursor $c_0$, in orange, reproduces according to \(\mu\), as well as every other individual, hence all the trees depicted by light blue bubbles are \(\mu\)-Bienaymé trees. For all $k \geq 1$, all the finite trees grafted to a given side of $c_{-k}$ together with this vertex form a past or future bush $\mathcal B_{k,\pm}$. }\label{fig:InvariantTree}
  \end{center}
\end{figure}

Finally, we stress that \(\mathcal{B}_{0}\) is distributed as \(\Tree\),
while every other bush \(\mathcal{B}_{k,\pm}\) is distributed as a
slightly different tree called an adjoint \(\mu\)-Bienaymé tree and
denoted by \(\widetilde{\Tree}\). Indeed, due to the size-bias and the uniform location of the next child on the spine, the root of \(\widetilde{\Tree}\) has an offspring distributed
according to
\[\widetilde{\mu}(k)=\mu(\lllbracket  k+1,+\infty \lllbracket  ),\]
while all other individuals (if any) reproduce independently according
to \(\mu\). It is a simple matter to see that

\begin{equation}\label{eq:GenAdjoint}{
\gen_{\widetilde{\mu}}(x)=\frac{1-\gen_{\mu}(x)}{1-x}.
}\end{equation}

In a similar way, $\mathcal{B}_k =\mathcal{B}_{k,+} \cup \mathcal{B}_{k,-}$ is distributed as a tree $\Tree'$ where the root reproduces according to $\mu'*\delta_{-1}$ and all other individuals (if any) reproduce independently according to $\mu$.

\subsubsection{\texorpdfstring{\(\Tree_{\infty}\)-walk and invariance}{\textbackslash Tree\_\{\textbackslash infty\}-walk and invariance}}\label{tree_infty-walk-and-invariance}

We will make almost no use of this property, but the main motivation for the
introduction of \(\Tree_{\mathrm{inv}}\) in
\autocite{LeGallLin2016RangeTreeindexedRandom} is an invariance
property, which is also shared by \(\Tree_{\infty}\). Indeed, one can
consider the vertex coming right after \(c_{0}\) in
lexicographical order and declare it to be the new cursor, which means that we consider the ancestral
lineage of this new cursor to be the new spine. This gives a new pointed tree \(\Tree_{\infty}'\), and
this tree is actually distributed as \(\Tree_{\infty}\). This invariance
may be recovered as consequence of the construction of \(\Tree_{\infty}\) as a
local limit around a uniform vertex.

With our framework, this invariance automatically extends into an invariance in distribution of \(\Tree_{\infty}\)-walk
\(W_{\Tree_{\infty}}\) when one moves the cursor from $c_0$ to the next vertex and shifts all spatial positions so that the new cursor is still at $0 \in \Z^d$. However, recall that the jumps of a \(W_{\Tree_{\infty}}\) follows the
genealogical structure of \(\Tree_{\infty}\): If \(u\) is the parent of
\(v\) in \(\Tree_{\infty}\) then
\(W_{\Tree_{\infty}}(v)-W_{\Tree_{\infty}}(u)\) is distributed as
\(\theta\). This means that if we start at \(c_{0}\) and then goes down
the spine, we will not see a random walk \(S\) with jump distributed
according to \(\theta\) but its reversed \(\bar{S}\) where the jumps
follows \(\bar{\theta}(x)=\theta(-x)\). This observation motivates the framework of pointed genealogical trees, where the special orientation of the walk along the spine is a natural consequence of the genealogy of the tree.

\subsection{Recurrence and transience for infinite branching random walk}\label{recurrence-and-transience-for-infinite-branching-random-walk}

We now prove \Cref{prop:DichotomyRecTransienceAdlous}, \textit{i.e.} we show that the classical dichotomy between
transience and recurrence still holds for \(\Tree_{\infty}\)-walk. Let
us first state some preliminary lemmas about recurrence.

\begin{lemma}[$0$-$1$ law for visiting a vertex infinitely often]\label{lem:01lawBienayme}

Assume that \(\mu\) is critical and \(\theta\) is not supported by a
strict subgroup of \(\Z^d\), then 
\[\forall x \in \Z^d,\ \P\left(\sum_{u \in \Tree_{\infty}} \1_{W_{\Tree_{\infty}}(u)=x} = +\infty\right) \in \{ 0, 1 \}.\]
Moreover, this probability actually does not depend on x.

\end{lemma}

\begin{lemma}[Recurrence cannot be restricted to the past]\label{lem:RecurrenceAndFutureRecurrence}

If almost surely \(x\) is visited infinitely often by
\(W_{\Tree_{\infty}}\), then actually
\(\#\{  u \in \Tree_{+}: W_{\Tree_{\infty}}(u) = x \} = +\infty\) almost surely.

\end{lemma}

Finally, to relate $\P(\esc)$ and $\P(\esc_+)$, we adapt the method proposed in \autocite{LeGallLin2016RangeTreeindexedRandom} to express $\P(\esc_+)$.
We will need the following estimates on the probabilities that a \(\Tree\)-walk, a $\widetilde{\Tree}$-walk or a $\Tree'$-walk starting from \(0\) reach some position \(x \in \Z^{d}\). 

\begin{lemma}\label{lem:VisitingProbaV1}
Set \(p(x) = \P(x \in \Tree^0)\), \(\widetilde{p}(x) = \P(x \in \widetilde{\Tree}^0)\), and \(p'(x) = \P(x \in (\Tree')^0)\). Then we have
\begin{enumerate}[(i)]
    \item \( \widetilde{p}(x) \sim_{x \to \infty} p(x)^{\alpha-1}L(p(x)^{-1}). \)
    \item \( p'(x) \sim_{x \to \infty} \alpha p(x)^{\alpha-1}L(p(x)^{-1}). \)
\end{enumerate}
\end{lemma}

\Cref{prop:DichotomyRecTransienceAdlous} is then a consequence of those lemmas and the invariance
property of \(W_{\Tree_{\infty}}\).

\begin{proof}[Proof of \Cref{prop:DichotomyRecTransienceAdlous}.]

First, let us link the events \(\esc\) and \(\esc_{+}\). Since
\(\esc \subset \esc_{+}\), we obviously have
\(\P(\esc_{+})=0 \implies \P(\esc)=0\). To prove the converse, one can
follow Le Gall \& Lin's approach and
condition on the spine-indexed walk \(W_{\mathcal{S}}\) to
express \(\P(\esc)\) and \(\P(\esc_{+})\) according to a single random
walk \((S_{k})_{k}\). Indeed, given \(W_{\mathcal{S}}\), the restrictions of the walks to the bushes are independent and we get

\[ \P(\esc) = \P\bigl(0 \not\in W_{\Tree_{\infty}}(\mathcal{B}_0\setminus\{c_0\})\bigr) \E \biggl(\prod_{k \geq 1} \P(0 \not\in W_{\Tree_{\infty}}(\mathcal{B}_k) | W_{\mathcal{S}} )\biggr),\]
where we recall that $\mathcal{B}_k =\mathcal{B}_{k,+} \cup \mathcal{B}_{k,-}$ is the bush formed of $c_{-k}$ and the finite subtrees grafted to it in $\Tree_{\infty}$. Since $\mathcal{B}_k$ is distributed as $\Tree'$, $W_{\mathcal{S}}$ is distributed as \((-S_{k})_{k}\), and the only dependence between $W_{\mathcal{S}}$ and the walk on the bush $\mathcal{B}_k$ is contained in $W_{\mathcal{S}}(c_{-k})$, we deduce that $\P(\esc)=0$ if and only if almost surely
\(\prod_{k \geq 1} 1-p'(\bar{S}_{k}) = 0\), which is
equivalent to \(\sum_{k \geq 1} p'(S_{k}) = +\infty\) almost
surely.

A similar computation for \(\Tree_+\) instead of \(\Tree_{\infty}\) yields that \(\P(\esc_{+})=0\) if and only if almost surely
\(\sum_{k \geq 1 } \widetilde{p}(S_{k}) = +\infty\). But according to \Cref{lem:VisitingProbaV1}, we clearly see that \(\sum_{k \geq 1 } \widetilde{p}(S_{k}) = +\infty\) if and only if \(\sum_{k \geq 1 } p'(S_{k}) = +\infty\), thus we conclude that \(\P(\esc)=0\) if and only if \(\P(\esc_{+})=0\).

We can now prove the first statements of \Cref{prop:DichotomyRecTransienceAdlous}:

\begin{enumerate}[(i)]
\item Clearly, \(\P(\sum_{u \in \Tree_{\infty}} \1_{W_{\Tree_{\infty}}(u)=0} = +\infty) \leq 1 - \P(\esc)\), hence by \Cref{lem:01lawBienayme} \(\P(\esc)>0\) implies that almost surely there is no \(x \in \Z^d\) visited infinitely often by \(W_{\Tree_{\infty}}\).
\item Conversely, if \(\P(\esc)=0\) then \(\P(\esc_{+})=0\) and by the invariance property of \(W_{\Tree_{\infty}}\), almost surely we can construct a sequence of vertices \((\Omega_{k})_{k\geq 1}\) in \(\Tree_{+}\), increasing in lexicographical order, such that \(W_{\Tree_{\infty}}(\Omega_{k})=0\). This shows that \(\P(\sum_{u \in \Tree_{\infty}} \1_{W_{\Tree_{\infty}}(u)=0} = +\infty) = 1\) and again we conclude by \Cref{lem:01lawBienayme}.
\end{enumerate}

Finally, \Cref{lem:RecurrenceAndFutureRecurrence} gives the last assertion of \Cref{prop:DichotomyRecTransienceAdlous}.
\end{proof}

We end this section by proving the three lemmas.

\begin{proof}[Proof of \Cref{lem:01lawBienayme}.]
Fix \(x \in \Z^d\). Observe that altering a finite part of
\(\Tree_{\infty}\) (with the corresponding jumps on the edges of this
finite part) \emph{in a way such that the walk is modified on at most a
finite number of vertices} does not affect the fact that
\(W_{\Tree_{\infty}}\) visits \(x\) infinitely often. For instance,
exchanging two jumps on the spine is such an alteration, but modifying a
single jump on the spine is not as it alters the entire walk beyond this
vertex.

For \(i \geq 1\), set \(J_{i}=W_{\Tree_{\infty}}(c_{-(i-1)})-W_{\Tree_{\infty}}(c_{-i})\) to be the \(i\)-th jump on the spine
\(\mathcal{S}\) of \(\Tree_{\infty}\), and set \(\mathcal{B}_i\) to be the bush grafted on both sides of the spine at \(c_{-i}\), rooted at this vertex. Finally, \(W_{\mathcal{B}_i}\) is the
\(\mathcal{B}_i\)-indexed walk obtained by assigning position \(0\) to \(c_{-i}\) and then jumping with the jumps already assigned to the
edges of \(\mathcal{B}_i\). By construction \((J_{i},W_{\mathcal{B}_i})_{i \geq 1}\) are
i.i.d. and are enough to reconstruct \(W_{\Tree_{\infty}}\) except for
the finite bush grafted to the cursor $c_0$ of \(\Tree_{\infty}\). According to
the preliminary remark, the event
\(\bigl(\sum_{u \in \Tree_{\infty}} \1_{W_{\Tree_{\infty}}(u)=x} = +\infty \bigr)\)
is measurable with respect to \((J_{i},W_{\mathcal{B}_i})_{i \geq 1}\) and is
invariant by finite permutation of the indices \(i\), hence by
Hewitt-Savage \(0\)-\(1\) law it has probability \(0\) or \(1\).

Now assume that there is a \(x_{0} \in \Z^d\) which is visited
infinitely often by \(W_{\Tree_{\infty}}\) almost surely. There are two
possibilities:

\begin{enumerate}
    \item We may have \(\P\bigl(\#\{ u \in \mathcal{S}: W_{\Tree_{\infty}}(u) =x_{0} \} =+\infty\bigr) > 0\), in which case the random walk indexed by the spine \(\mathcal{S}\) is recurrent and by assumption it is irreducible, hence almost surely it visits every \(x \in \Z^d\) infinitely often.
    \item Else, since all bushes grafted to the spine are finite, almost surely we must have infinitely many bushes that contains a \(u \not\in \mathcal{S}\) such that \(W_{\Tree_{\infty}}(u)=x_{0}\). Denote by \((\mathcal{B}_{[i]})_{i \geq 1}\) the subsequence of bushes hitting $x_0$. For each, \(i \geq 1\), we explore the tree \(\mathcal{B}_{[i]}\) in lexicographical order, until we find a vertex outside of the spine with position $x_0$ denoted by \(u_{i}\). Then, by the branching property the descendants of \(u_{i}\) form a \(\mu\)-Bienaymé tree \(T_{i}\) independent of the vertices explored earlier and of every other bushes, together with the corresponding jumps. Since \(\theta\) generates an irreducible random walk, for all \(x \in \Z^d\) we have \(\P(x \in \Tree^{x_{0}}) >0\), hence by Borel-Cantelli lemma there is infinitely many \(T_{i}\) that contains a vertex \(v\) such that \(W_{\Tree_{\infty}}(v)=x\).
\end{enumerate}
In both cases we see that when almost surely \(x_0\) is visited infinitely often, then this holds for all \(x \in \Z^{d}\).
\end{proof}

A simple extension of the argument above also yields \Cref{lem:RecurrenceAndFutureRecurrence}.

\begin{proof}[Proof of \Cref{lem:RecurrenceAndFutureRecurrence}.]

We argue as earlier by considering two possibilities:

\begin{enumerate}
    \item If \(W_{\mathcal{S}}\), the random walk indexed by the spine \(\mathcal{S}\), is recurrent, then almost surely it visits \(x\) infinitely often and we have an i.i.d. sequences of bushes included in \(\Tree_{+}\) grafted to those vertices of the spine that visit \(x\). Borel-Cantelli Lemma guaranties that almost surely there are infinitely many of those bushes that also visit \(x\), hence we have the desired result.
    \item Else, as earlier we must have infinitely many bushes \((\mathcal{B}_{[i]})_{i \geq 1}\) that contains a \(u\) such that \(W_{\Tree_{\infty}}(u)=x\). Each of those bushes can be split into its part in \(\Tree_{+}\), denoted by \(\mathcal{B}_{[i],+}\), and its part in \(\Tree_{-}\setminus \mathcal{S}\), denoted by \(\mathcal{B}_{[i],-}\). Given the random walk \(W_{\mathcal{S}}\), the \(\mathcal{B}_{[i]}\)-indexed walks are independent, and by symmetry we have
\begin{equation*}{
\P(\exists u \in \mathcal{B}_{[i],+}: W_{\Tree_{\infty}}(u)=x \ \vert\ W_{\mathcal{S}})=\P(\exists u \in \mathcal{B}_{[i],-}: W_{\Tree_{\infty}}(u)=x \ \vert\ W_{\mathcal{S}}) \geq \frac{1}{2}, }
\end{equation*}
hence by Borel-Cantelli lemma, given \(W_{\mathcal{S}}\), almost surely there are infinitely many \(B_{[i],+}\) that visit \(x\), and the desired result follows.
\end{enumerate}
\end{proof}

Finally, we deal with \Cref{lem:VisitingProbaV1}. Actually, we only prove the second assertion here. The first one can be obtained in a very similar way by replacing \eqref{eq:GenSizeBias} by \eqref{eq:GenAdjoint}. We will also prove later (see \Cref{lem:FirstEstimatesVisitingProba}) an extended version of this first assertion, the reader may find there a detailed proof. 
\begin{proof}[Proof of \Cref{lem:VisitingProbaV1}.] 

We decompose \(\Tree'\) into its subtrees grafted to the root,
and we use that according to \eqref{eq:GenSizeBias} the offspring distribution $\mu'*\delta_{-1}$ of this root has generating
function \(\gen_{\mu}'(x)\). For \(x \neq 0\), not hitting $x$ means having no subtree grafted to the root that hits $x$, and those subtrees are classical $\mu$-Bienaymé trees thus we get

\begin{equation*}{
1-p'(x) = \gen_{\mu}'\big(1-\mathbb E(p(x-S_{1})\big). 
}\end{equation*}

The same decomposition at the root applied to $\Tree$ yields $1-p(x)=1-\gen_{\mu}\big(1-\mathbb E(p(x-S_{1})\big)$, and we obviously have $p(x) \to 0$ as $x \to \infty$ hence by criticality of $\mu$ we have $\mathbb E(p(x-S_{1})) \sim p(x)$ . Moreover, we also have from
\eqref{eq:GenMuStable} and the Monotone Density Theorem (\Cref{thm:MonotoneDensity}) that
\[1-\gen_{\mu}'(1-s) \sim_{s \to 0} \alpha s^{\alpha-1}L\left(\frac{1}{s} \right).\]
Putting everything together, we see that \(p'(x) \sim_{x \to \infty} \alpha p(x)^{\alpha-1}L(p(x)^{-1}). \)
\end{proof}

\section{Recurrence and transience of stable-BRW}\label{recurrence-and-transience-of-stable-brw-1}

\subsection{Recurrence in low dimension}\label{recurrence-in-low-dimension}

To prove that \(\Tree_{n}\)-walks are recurrent in dimension \(d \leq 4\)
when \(\mu\) has a finite variance and \(\theta\) has a finite moment of
order \(3\), Le Gall \& Lin relied on a
\(2^\text{nd}\) moment method to bound \(\P\left(0 \in \Tree^x \right)\)
for large \(x\), which is sufficient to show that
\(\P \bigl(\esc_{+} \bigr) = 0\). However, as soon as
\(\sigma_{\mu}^2 = +\infty\), we cannot adapt this argument due to lack
of moments. Here, we propose a different approach based on the control
of the maximal displacement \(\max W_{\Tree_{n}}\) of a
\(\Tree_{n}\)-walk in \(\Z\), which actually works in a more general setting. The main idea is to
adapt a control due to Maillard \autocite{Maillard2016MaximumTreeindexedRandom}.
Let us first review this result before adapting it to our setting. The recurrence criterion \Cref{thm:RecurrenceCriterion} and its corollaries about \(\Tree_{n}\)-walks will directly follow from this.

\subsubsection{Maillard's result}\label{maillards-result}

We consider a sequence of random trees \(T_{n}\) such that \(\#T_{n}=n+1\). In \autocite{Maillard2016MaximumTreeindexedRandom}, Maillard is concerned with
\(T_{n}\)-walks in \(\Z\), denoted by \(W_{n}\) in the following, with a
heavy-tailed jump distribution. Namely, he assumes that the jumps in \(W_{n}\) are i.i.d. copies of
a variable \(J\) such that
\begin{enumerate}[(i)]
    \item \(J\) is centered;
    \item \(x \mapsto \P(J>x)\)
and \(x \mapsto \P(\lvert J \rvert>x)\) are \((-\beta)\)-varying for
some \(\beta>0\).
\end{enumerate}
He proves that the maximum  \(\max \lvert W_{n} \rvert\) is achieved by one big jump during the walk when $\beta$ is sufficiently small compared to the dimension of the trees $T_n$. More precisely, he considers the crude notion of dimension given by \Cref{Assum:MaillardDimension}, which roughly says that the trees are at least $D$-dimensional when for all $\varepsilon>0$ we have $\mathrm{volume} \geq \mathrm{height}^{D-\varepsilon}$ in those trees.

A simplified version of \autocite[Theorem
2.1]{Maillard2016MaximumTreeindexedRandom} is that when \Cref{Assum:MaillardDimension}
holds for some $D>1$ and \(\beta< 2D\), then there is a slowly varying function \(h\)
such that

\begin{equation}\label{eq:MaillardResult}
  \frac{\max \lvert W_{n} \rvert}{n^{1/\beta}h(n)} \text{ converges to a Pareto distribution.}
\end{equation}
In particular, this gives us the order of \(\max \lvert W_{n} \rvert\)
in this big-jump regime.

\subsubsection{Adaptation to light-tailed jumps and recurrence criterion}\label{adaptation-to-light-tailed-jumps-and-recurrence-criterion}

We are interested in \(T_{n}\)-walks where the jump distribution has a
lighter tail, such as nearest-neighbour jumps, hence we cannot apply
directly the previous result, however we can couple heavy-tailed jumps
and light-tailed jumps to get the following result.

\begin{proposition}\label{prop:MaillardLightTailVersion}

Consider \(W_{n}'\) a \(T_{n}\)-walk in \(\Z\), where \((T_{n})\)
satisfies \Cref{Assum:MaillardDimension} for some $D>1$. Assume that the jumps in \(W_{n}'\) are i.i.d. copies of
a centered variable \(J'\) such that for all \(\delta < 2D\), $\E(\lvert J' \rvert^{\delta}) < +\infty$.

Then for all \(\varepsilon>0\) we have

\begin{equation*}{
 \frac{\max \lvert W_{n}' \rvert }{n^{\frac{1}{2D}+\varepsilon}} \xrightarrow[n \to +\infty]{(\P)} 0.
}\end{equation*}

\end{proposition}

Since a \(T_{n}\)-walk \(W_{T_{n}}\) in \(\Z^{d}\) is always contained
in the ball \(B(0, \max \lVert W_{T_{n}} \rVert)\), the recurrence criterion \Cref{thm:RecurrenceCriterion} immediately follows. We first prove this criterion and its application to \(T_{n}=\Tree_{n}\) before coming back on \Cref{prop:MaillardLightTailVersion}.

\begin{proof}[Proof of \Cref{thm:RecurrenceCriterion} and \Cref{cor:RecLowDim}.]
We decompose the \(T_{n}\)-walk \(W_{T_{n}}\) into its \(d\) coordinates
by writing \(W_{T_{n}}(u)=(W_{n}^{(1)}(u),\dots,W_{n}^{(d)}(u))\).
Observe that by assumption on $\theta$, each \(W_{n}^{(i)}\) is a \(T_{n}\)-walk in \(\Z\) with a jump distribution satisfying the assumptions of \Cref{prop:MaillardLightTailVersion}. Hence, as soon as $d>2D$
we have \(\frac{1}{d} > \frac{1}{2D}\) and can apply \Cref{prop:MaillardLightTailVersion} to get

\begin{equation*}{
\frac{\max \lVert W_{T_{n}} \rVert }{n^\frac{1}{d}} \asymp \max_{1 \leq i \leq d} \left( \frac{\max \lvert W_{n}^{(i)} \rvert}{n^\frac{1}{d}}  \right) \xrightarrow[n \to +\infty]{(\P)} 0. 
}\end{equation*}
Then, every point visited by the walk is included in
\(B(0,\max\lVert W_{\Tree_{n}} \rVert)\) thus we have

\begin{equation*}{
\frac{R^{\theta}(T_{n})}{n} \preceq   \left(\frac{\max\lVert W_{T_{n}} \rVert } {n^{1/d}} \right)^{d} \xrightarrow[n \to +\infty]{(\P)}0.
}\end{equation*}

This concludes the proof of \Cref{thm:RecurrenceCriterion}. To obtain \Cref{cor:RecLowDim}, let us recall that in the $\alpha$-stable regime, the size-conditioned Bienaymé trees $\Tree_n$ satisfy \Cref{Assum:MaillardDimension} with $D=\frac{\alpha}{\alpha-1}$ (see \eqref{eq:HeightEstimate}), thus we can directly apply \Cref{thm:RecurrenceCriterion} to obtain the recurrence. Then, $\P(\esc)=\P(\esc_+)=0$ follows from \eqref{eq:LinkFiniteInfinite} and \Cref{prop:DichotomyRecTransienceAdlous}.
\end{proof}

\begin{remark}[Effect of big jumps on recurrence]\label{rmk:BigJumpEffect}

With the original result of Maillard, we also see that when \((T_{n})\) satisfies
\Cref{Assum:MaillardDimension} but \(\theta\) is heavy-tailed, we can still prove
recurrence of \(T_{n}\)-walks provided that the dimension is low enough.
For instance if \(\theta\) has a finite moment of order \(\beta\) for
some \(\beta<2D\), then \(T_{n}\)-walks in \(\Z^{d}\) are recurrent for
\(d<\beta\). However we do not have result on the transience in
dimension \(d>\beta\) (for a maximal \(\beta\)) in this heavy-tailed
regime.
\end{remark}

\begin{proof}[Proof of \Cref{prop:MaillardLightTailVersion}.]

Fix \(\varepsilon>0\), and choose \(\beta \in (0,2D)\) such that
\(\frac{1}{2D} <\frac{1}{\beta} < \frac{1}{2D} +\varepsilon\). Now consider a
random variable \(J\) independent from \(J'\) and such that
\begin{enumerate}[(i)]
    \item \(J\) is centered;
    \item \(x \mapsto \P(J>x)\)
and \(x \mapsto \P(\lvert J \rvert>x)\) are \((-\beta)\)-varying for
this \(\beta\).
\end{enumerate}
We claim that the sum
\(J+J'\) is centered and such that \(\P(J+J' > x) \sim \P(J > x)\) and
\(\P(\lvert J+J' \rvert > x)\sim \P(\lvert J \rvert>x)\), so those tails
distribution functions are also \((-\beta)\)-varying.

Indeed, for any \(a \in (0,1)\) we can decompose according to the events
\(\bigl(\lvert J' \rvert \leq ax \bigr)\) and
\(\bigl(\lvert J' \rvert > ax \bigr)\). Then, as $\lvert J'\rvert$ has finite moment of order slightly more than $\beta$, 
\(\P(\lvert J' \rvert > ax) = o\bigl(\P(\lvert J \rvert> x) \bigr)\) for large $x$. This leads to

\begin{equation*}{
 \frac{\P(\lvert J+J' \rvert > x)}{\P(\lvert J \rvert > x)} = \frac{\P(\lvert  J+J' \rvert >x \ \& \ \lvert  J' \rvert  \leq ax)}{\P(\lvert J \rvert > x)} + o(1).
}\end{equation*}

On the event \(\bigl(\lvert J' \rvert \leq ax \bigr)\), we have
\(\bigl(\lvert J \rvert> (1+a)x \bigr) \subset \bigl(\lvert J+J' \rvert > x \bigr) \subset \bigl(\lvert J \rvert > (1-a)x \bigr)\)
by the triangle inequality, hence we can conclude that

\begin{equation*}{
 \begin{aligned}
\varliminf_{ x \to +\infty } \frac{\P(\lvert J+J' \rvert > x)}{\P(\lvert J \rvert > x)} &\geq \varliminf_{ x \to +\infty } \frac{\P(\lvert J \rvert >(1+a)x)}{\P(\lvert J \rvert >x)}
=  (1+a)^{-\beta}; \\
\varlimsup_{ x \to +\infty } \frac{\P(\lvert J+J' \rvert > x)}{\P(\lvert J \rvert > x)} & \leq \varlimsup_{ x \to +\infty } \frac{\P(\lvert J \rvert >(1-a)x)}{\P(\lvert J \rvert >x)} =  (1-a)^{-\beta}.
\end{aligned}
}\end{equation*}

Finally, making \(a\) tends to \(0\) gives
\(\P(\lvert J+J' \rvert> x) \sim \P(\lvert J \rvert> x)\). The same
argument also gives \(\P(J+J' > x) \sim \P(J >x)\).

Now, assign independent copies of \(J\) and \(J'\) to the \emph{same}
tree \(T_{n}\), so that we can consider \(3\) walks indexed by this
tree: \(W_{n}\) where the jumps are copies of \(J\), \(W_{n}'\) where
the jumps are copies of \(J'\), and \(W_{n}+W_{n}'\) (where the jumps
are copies of \(J+J'\)). We simply write

\begin{equation*}{
 \max \lvert W_{n}' \rvert \leq \max \lvert W_{n} \rvert +\max\lvert W_{n} +W_{n}' \rvert,
}\end{equation*}
and then we apply Maillard's result \eqref{eq:MaillardResult} to both \(W_{n}\) and \(W_{n}+W_{n}'\)
to get that, by choice of \(\beta\), we have

\begin{equation*}{
 \frac{\max \lvert W_{n} \rvert +\max\lvert W_{n} +W_{n}' \rvert}{n^{\frac{1}{2D } +\varepsilon}} \xrightarrow[n \to +\infty]{(\P)} 0. 
}\end{equation*}
This proves \Cref{prop:MaillardLightTailVersion}.
\end{proof}

\subsubsection{Recurrence in the Cauchy case}\label{recurrence-in-the-cauchy-case}

The generality of \Cref{thm:RecurrenceCriterion} enables us to look at some other
tree-indexed walks. Here, we apply it to \(\mu\)-Bienaymé trees
\(\Tree_{n}\) where \(\mu\) is assumed to have an even heavier tail
(while still being critical). More precisely, in this subsection only,
\(\mu\) satisfies \Cref{Assum:Cauchy} instead of being attracted to an
\(\alpha\)-stable tree with \(\alpha \in (1,2]\). As explained in
introduction, this assumption is slightly stronger than requiring that
\(\mu\) is attracted to a spectrally positive Cauchy distribution (which
is stable of index \(\alpha=1\)).

In this setting, there is no non-trivial scaling limit of the trees
\(\Tree_{n}\), but \autocite[Theorem
1.3]{Addario-BerryDonderwinkelKortchemski2025CriticalTreesAre} also
proves that with high probability the height \(H(\Tree_n)\) of
\(\Tree_{n}\) grows as a slowly varying function as \(n \to +\infty\).
This implies that \(\Tree_{n}\) are \(+\infty\)-dimensional, in the
sense that they satisfy \Cref{Assum:MaillardDimension} for all $D > 1$.

As a consequence, when \Cref{Assum:Cauchy} holds, if we use a centered jump
distribution \(\theta\) with finite moment of order \(\delta\) for all
\(\delta>0\), we can apply \Cref{thm:RecurrenceCriterion} for every $D>1$ to obtain that \(\Tree_{n}\)-walks
are recurrent in \(\Z^d\) whatever \(d\) is, in the sense that

\begin{equation*}{
\frac{R^{\theta}(\Tree_{n})}{n} \xrightarrow[n \to +\infty]{(\P)} 0.
}\end{equation*}
This proves \Cref{cor:RecCauchy}.

\subsection{\texorpdfstring{At the critical dimension \(d=\frac{2\alpha}{\alpha-1}\)}{At the critical dimension d=\textbackslash frac\{2\textbackslash alpha\}\{\textbackslash alpha-1\}}}\label{at-the-critical-dimension-dfrac2alphaalpha-1}

\subsubsection{Recurrent behaviour at critical dimension}\label{recurrent-behaviour-at-critical-dimension}

Whenever \(\frac{2\alpha}{\alpha-1}\) is an integer, \emph{i.e.}
\(\alpha= \frac{m}{m-2}\) for some \(m\geq 4\), our general bound on
\(\max W_{\Tree_{n}}\) obtained in \Cref{prop:MaillardLightTailVersion} is too rough and cannot
tell us anything about the recurrence or transience of $\Tree_{n}$-walks in
\(\Z^{\frac{2\alpha}{\alpha-1}}\). However, in the specific setting of
\(\mu\)-Bienaymé trees where $\mu \in \mathrm{dom}(\alpha)$ with \(\alpha \in (1,2]\), much finer
results are available. In particular a result due to Marzouk \autocite[Theorem
1]{Marzouk2020ScalingLimitsDiscrete} gives the scaling limit of
\(W_{\Tree_{n}}\) in the regime that interests us. We first use this to discuss the exact asymptotic of \(\max W_{\Tree_{n}}\), then we turn to transience and recurrence at criticality and prove \Cref{prop:RecCriticalGaussian,prop:NoUniversalityAtCriticality}.

Actually, Marzouk dealt with the case \(d=1\). Nonetheless, obtaining the scaling limit in higher dimension is a simple
corollary of the case \(d=1\) as this one-dimensional result directly gives tightness for the
multidimensional walk, so we directly consider the multidimensional
extension of \autocite[Theorem 1]{Marzouk2020ScalingLimitsDiscrete}.
First, we reproduce the necessary and sufficient condition on $\theta$ given by Marzouk to ensure that the $\Tree_n$-walks converge toward a Brownian snake.

\begin{assumption}\label{Assum:MomentAssumJumpMarzouk}
\begin{itemize}
    \item The jump distribution satisfies
    \begin{equation*}{
 \theta\left( \left\{  x \in \Z^d: \lVert x \rVert \geq r^{\frac{\alpha-1}{2\alpha}}\ell(r)^{-1/2}  \right\} \right) = o_{r \to +\infty}\left( \frac{1}{r} \right),
}\end{equation*}
where $\ell$ is the slowly varying function involved in the generalized CLT satisfied by $\mu$ (see \eqref{eq:StableCLT}).
\end{itemize}

\end{assumption}
Observe that since \(\ell\) is slowly varying, we can apply an asymptotic inversion (\Cref{prop:asymptoticInversion}) to express this assumption as a bound on the tail $\theta\left( \left\{  x \in \Z^d: \lVert x \rVert \geq r  \right\} \right)$, and then we see that \Cref{Assum:MomentAssumJumpMarzouk} is almost equivalent to the moment assumption on $\theta$ required for \Cref{cor:RecLowDim}. More precisely, we have in general

\begin{equation*}
    \exists \delta > \frac{2\alpha}{\alpha-1}, \sum_{x \in \Z^d} \lVert x \rVert^{\delta}\theta(x) < +\infty  \Longrightarrow \Cref{Assum:MomentAssumJumpMarzouk} \Longrightarrow \forall \delta < \frac{2\alpha}{\alpha-1}, \sum_{x \in \Z^d} \lVert x \rVert^{\delta}\theta(x) < +\infty,
\end{equation*}
and this may be refined depending on the precise asymptotics of $\ell$. The relevant consequence of \autocite[Theorem
1]{Marzouk2020ScalingLimitsDiscrete} here is that when \Cref{Assum:MomentAssumJumpMarzouk}
holds,

\begin{equation}\label{eq:CVMaximumMarzouk}
  \frac{\max \lVert W_{\Tree_{n}} \rVert }{n^{\frac{\alpha-1}{2\alpha}} \ell(n)^{-1/2} } \text{ converges in distribution.}
\end{equation}

Based on this and Le Gall \& Lin's result from \autocite{LeGallLin2016RangeTreeindexedRandom} about the range in dimension $d=4$ given that $\mu$ has a finite variance, we prove that when $\alpha=2$ and $\theta$ satisfies \Cref{Assum:MomentAssumJumpMarzouk}, $\Tree_n$-walks in $\Z^4$ must be recurrent, which is \Cref{prop:RecCriticalGaussian}.

\begin{proof}[Proof of \Cref{prop:RecCriticalGaussian}.]
As said earlier, the case where $\mu$ has a finite variance has already been covered by Le Gall \& Lin, so we focus on the case where $\alpha=2$ but $\mu$ has infinite variance. In this case, observe that $\ell(n) \to +\infty$ as $n \to +\infty$. Indeed, this is equivalent to $L(n) \to +\infty$ as $n \to +\infty$ (see \Cref{offspring-attracted-to-an-alpha-stable-distribution} and the references therein), but $L(n)$ can easily be related to the variance of $\mu$ when $\alpha=2$. For instance, one can deduce from \eqref{eq:GenMuStable} and the Monotone Density Theorem (\Cref{thm:MonotoneDensity}) that
\begin{equation*}
    1-\gen_{\mu}'(1-s) \sim_{s \to 0}  2sL\left(\frac{1}{s} \right),
\end{equation*}
thus we get 
\begin{equation*}
    L(n) \sim \frac{1-\gen_{\mu}'(1-\frac{1}{n})}{2\frac{1}{n}} \to \frac{\gen_{\mu}''(1)}{2} = +\infty \text{ as } n \to +\infty.
\end{equation*}
Then, since $\ell(n) \to +\infty$ we deduce from \eqref{eq:CVMaximumMarzouk} that 
\begin{equation*}
  \frac{\max \lVert W_{\Tree_{n}} \rVert }{n^{\frac{\alpha-1}{2\alpha}} } \xrightarrow[n \to +\infty]{(\P)} 0.
\end{equation*}
As in the proof of the recurrence criterion \Cref{thm:RecurrenceCriterion}, this gives for $d=\frac{2\alpha}{\alpha-1}$
\begin{equation*}{
\frac{R^{\theta}(\Tree_{n})}{n} \preceq   \left(\frac{\max\lVert W_{\Tree_{n}} \rVert } {n^{1/d}} \right)^{d} \xrightarrow[n \to +\infty]{(\P)}0,
}\end{equation*}
which is the desired result.
\end{proof}

However, this argument is not sufficient to fully understand whether we have
recurrence or transience at critical dimension
\(\frac{2\alpha}{\alpha-1}\) when \(\alpha < 2\), because \(\ell\) may
have wilder behaviour in this case. We can nonetheless assert that the
universal behavior stated in \Cref{prop:RecCriticalGaussian} breaks down when
\(\alpha<2\). Namely, our argument will tell us that recurrence is possible at
critical dimension. But the transience criterion established by Le Gall \& Lin can actually be applied
at critical dimension in some cases, without modifying their proof,
showing that transience is also possible. We quickly recall their
criterion and then prove this absence of universal behaviour, which is \Cref{prop:NoUniversalityAtCriticality}.

\subsubsection{Le Gall and Lin's transience criterion}\label{le-gall-and-lins-transience-criterion-at-critical-dimension}

Consider the visiting probabilities \(p(x) = \P(x \in \Tree^0)\) and \(\widetilde{p}(x) = \P(x \in \widetilde{\Tree}^0)\) defined in \Cref{lem:VisitingProbaV1}. As recalled in the proof of \Cref{prop:DichotomyRecTransienceAdlous}, one can express \(\P(\esc_{+})\) by decomposing \(\Tree_{+}\) as a collection of independent trees grafted on a spine and it leads to 
\begin{equation*}
    \P(\esc_{+})=0 \text{ if and only if } \sum_{k \geq 1} \widetilde{p}(S_{k}) = +\infty \text{ almost surely,}
\end{equation*}
where $S$ is a classical random walk with $\theta$-jumps.

The sufficient condition given by Le Gall \& Lin to ensure that
\(\P(\esc_{+})>0\) is simply to check that
\(\E(\sum_{k \geq 1} \widetilde{p}(S_{k}) ) < +\infty\). They obtained that this expectation is finite when $d> \frac{2\alpha}{\alpha-1}$, and actually their argument can be directly extended to cover the case where $d= \frac{2\alpha}{\alpha-1}$ and the additional \Cref{Assum:SufficientConditionTransience} holds. For completeness, we detail this in the proof of \Cref{prop:NoUniversalityAtCriticality}.

\begin{proof}[Proof of \Cref{prop:NoUniversalityAtCriticality}.]
The second assertion is similar to the proof of \Cref{prop:RecCriticalGaussian}: $L(n) \to +\infty$ is equivalent to $\ell(n) \to +\infty$ and we conclude with \eqref{eq:CVMaximumMarzouk}.

For the first assertion, we prove that given \Cref{Assum:SufficientConditionTransience} we have \(\E(\sum_{k \geq 1} \widetilde{p}(S_{k}) ) < +\infty\), as it implies \(\P(\esc_{+})>0\) as desired. This
expectation may be rewritten as

\begin{equation*}{
\E \left(  \sum_{k \geq 1} \widetilde{p}(S_{k}) \right) = \sum_{k \geq 1} \sum_{x \in \Z^d} \widetilde{p}(x)\P(S_{k}=x) = \left( \sum_{x \in \Z^d} \widetilde{p}(x)g_{\theta}(x) \right) - \widetilde{p}(0).
}\end{equation*}

To upper bound this, combine the relation between $\widetilde{p}$ and $p$ given in \Cref{lem:VisitingProbaV1} with the bound $p(x) \leq g_{\theta}(x)$ (obtained by $1^{\text{st}}$ moment method) to deduce that

\begin{equation}\label{eq:UpperBounfVisitingProbability}
  \widetilde{p}(x) \preceq g_{\theta}(x)^{\alpha-1}L(g_{\theta}(x)^{-1}).
\end{equation}

As a consequence of this and the standard estimate \eqref{eq:GreenEstimates} on $g_{\theta}$ (which holds given the moment assumption on $\theta$), we have

\begin{equation*}{
\E \left(  \sum_{k \geq 1} \widetilde{p}(S_{k}) \right) \leq C\sum_{x \in \Z^d} g_{\theta}(x)^{\alpha} L(g_{\theta}(x)^{-1}) \leq C \sum_{n \geq 1}n^{d-1-\alpha(d-2)}L(n^{d-2}).
}\end{equation*}
where \(d=\frac{2\alpha}{\alpha-1}\). At this critical dimension, we actually have
\(d-2= \frac{2}{\alpha-1}\) and $d=\alpha(d-2)$, thus this sum is finite if and only if
\Cref{Assum:SufficientConditionTransience} holds, which proves \Cref{prop:NoUniversalityAtCriticality}.
\end{proof}

\section{Branching capacity for transient stable BRW}\label{capacity-of-transient-stable-brw-1}

We now focus on the transient situation, \emph{i.e.} in this section we
will always assume that \(d> \frac{2\alpha}{\alpha-1}\), or
\(d=\frac{2\alpha}{\alpha-1}\) and \Cref{Assum:SufficientConditionTransience} holds. Moreover, in
order to study branching capacity we will always assume that the
asymptotics of the Green function \(g_{\theta}\) is given by
\eqref{eq:GreenEstimates}, which requires that the centered distribution \(\theta\)
has at least a finite moment of order \(d-1\) according to Le Gall \& Lin 
\autocite{LeGallLin2016RangeTreeindexedRandom}. This also means that for
\((S_{n})_{n}\) with jumps distributed as \(\theta\),
\(\P(\lVert S_{1} \rVert>\lVert x \rVert) =o(\lVert x \rVert^{2-d})=o(g_{\theta}(x))\).
In a first part, dedicated to \Cref{thm:BranchingCapacityVisitingProba}, we do not require any
further moment assumption on \(\theta\), thus we slightly lightens the
moment assumption required by Zhu in \autocite{Zhu2017CriticalBranchingRandom}
to study branching capacity.

We provide \Cref{tab:notation branching} to help keep track of the notation. Also, we will use the following unusual convention. Given a path $\gamma=(\gamma(0),\ldots,\gamma(k))$ in $\Z^d$, $x,y \in \Z^d$ and a subset $B \subset \Z^d$, we write $\gamma:x \to y$ to denote that the path starts at $\gamma(0)=x$ and ends at $\gamma(k)=y$. Then, we write 
$$\gamma:x \to y \subset B$$ 
to denote that $\gamma$ stays in $B$ \textit{expect maybe at its endpoints}, \textit{i.e.} $\forall 1 \leq i \leq k-1, \gamma(i) \in B$. In particular, this still makes sense when $x \not\in B$ or $y \not\in B$.

\begin{table}[htbp]
\caption{Notation used throughout this section}
\centering
\begin{tabular}{c p{10cm}}
\toprule

  $S=(S_n)_n$ & classical random walks with $\theta$-jumps;\\
 \(H_{A} = \inf\{ n \geq 0 : S_{n} \in A\}\) & hitting time for the set $A \subset \Z^d$ by $S$;\\
\(H_{A}^{+}=\inf \{  n \geq  1 : S_{n} \in A\}\) & first return time for the set $A \subset \Z^d$ by $S$;\\
\( s(\gamma)= \P_{\gamma(0)}(\forall 1 \leq i \leq n, S_i= \gamma(i)\) & probability that $S$ realizes a path $\gamma=(\gamma(0),\ldots,\gamma(k))$ in $\Z^d$; \\
\(g_{\theta}(y,x) = g_{\theta}(x-y) = \sum_{\gamma : y \to x} s(\gamma)\) & Green function of $S$, which satisfies \eqref{eq:GreenEstimates};\\
$\bar{S}=(\bar{S}_n)_n$ & classical random walks with $\bar{\theta}$-jumps where $\bar{\theta}(x)=\theta(-x)$;\\
\(\bar{H}_{A} = \inf\{ n \geq 0 : \bar{S}_{n} \in A\}\) & hitting time for the set $A \subset \Z^d$ by $\bar{S}$; \\
\( \bar{s}(\gamma)= \P_{\gamma(0)}(\forall 1 \leq i \leq n, \bar{S}_i= \gamma(i)\) & probability that $\bar{S}$ realizes a path $\gamma=(\gamma(0),\ldots,\gamma(k))$ in $\Z^d$; \\
$\bar{g_{\theta}}(x)=g_{\theta}(-x)$ & Green function of $\bar{S}$; \\
 \hline

$\preceq$ & lexicographical (or depth-first) order in a tree;\\
$\mathcal{B}_0 \stackrel{(d)}{=} \Tree$ & bush of $\Tree_{-}$ grafted to $c_0$ and distributed as a $\mu$-Bienaymé tree;\\
$\mathcal{B}_{k,-} \stackrel{(d)}{=}\widetilde{\Tree}$ & bush of $\Tree_{-}$ grafted to $c_{-k}$ and distributed as an adjoint $\mu$-Bienaymé tree as defined in \Cref{future-and-past-of-the-root};\\
\(p(x,A)=\P(\Tree^x \cap A \neq \emptyset)\) & probability that a $(\mu,\theta)$-BRW starting from \(x\) visits \(A \subset \Z^d\); \\
\(\widetilde{p}(x,A)=\P(\widetilde{\Tree}^x \cap A \neq \emptyset)\) & probability that an adjoint $(\mu,\theta)$-BRW starting from \(x\) visits \(A\); \\
\(\Omega_{\text{in}}^{A} = \min \{ u \in \Tree: W_{\Tree}(u) \in A \}\) & hitting vertex of \(A\) by a $(\mu,\theta)$-BRW; \\
\(\Omega_{\text{out}}^{A}= \max \{ u \in \Tree: W_{\Tree}(u) \in A \}\) & exit vertex of \(A\) by a $(\mu,\theta)$-BRW;\\
\(p^{0}(x,A) =\P((\Tree\setminus \{ \varnothing \})^x \cap A \neq \emptyset)\) & probability that a $(\mu,\theta)$-BRW starting from \(x\) returns to \(A\);\\
\(\widetilde{p}^{0}(x,A) =\P((\widetilde{\Tree}\setminus \{ \varnothing \})^x \cap A \neq \emptyset)\) & probability that an adjoint $(\mu,\theta)$-BRW starting from \(x\) returns to \(A\);\\
\hline

\(s_{A}(\gamma)=s(\gamma)\prod_{i=0}^{k-1} \bigl(1-\widetilde{p}(\gamma(i),A) \bigr)\) & probability that a random walk with $\theta$-jumps and killing rate $\widetilde{p}(\cdot,A)$ realizes the path \(\gamma\);\\
\(g_A(y,x) = \sum_{\gamma : y \to x} s_A(\gamma)\) & Green function of the previous killed random walk; \\
\(s_{A}^{0}(\gamma)=s(\gamma)\prod_{i=0}^{k-1} \bigl(1-\widetilde{p}^{0}(\gamma(i),A) \bigr)\) & probability that a random walk with $\theta$-jumps and killing rate $\widetilde{p}^{0}(\cdot,A)$ realizes the path \(\gamma\);\\
\(g_A^{0}(y,x) = \sum_{\gamma : y \to x} s_A^{0}(\gamma)\) & Green function of the previous killed random walk; \\
 \bottomrule
\end{tabular}
\label{tab:notation branching}
\end{table}

\subsection{Branching capacity and visiting probability}\label{branching-capacity-and-visiting-probability}

\paragraph{\texorpdfstring{Main ideas leading to
\Cref{thm:BranchingCapacityVisitingProba}}{Main ideas leading to }}\label{main-ideas-leading-to-thm:BranchingCapacityVisitingProba}

As reminded in \Cref{capacity-of-transient-stable-brw}, the branching
capacity is based on a analogy with a Newtonian capacity. In particular,
in order to prove that we have a branching version of \eqref{eq:newtonianCapacityInterpretation}, it
will be useful to keep in mind how to prove \eqref{eq:newtonianCapacityInterpretation} first:
consider a finite subset \(A \subset \Z^{d}\), and \(x \in \Z^{d}\),
then a \emph{last passage (in \(A\)) decomposition} and the Markov
property gives

\begin{equation}\label{eq:ClassicalLastPassageDecomposition}{
\P_{x}(H_{A}<+\infty) = \sum_{a \in A} \sum_{n \in \N}\P_{x}(S_{n}=a)\P_{a}(H_{A}^{+}=+\infty) = \sum_{a \in A} g_{\theta}(a-x)\P_{a}(H_{A}^{+}=+\infty),
}\end{equation}
and for fixed \(a\), one can apply \eqref{eq:GreenEstimates} to get that
\(g_{\theta}(x-a) \sim_{x \to \infty} g_{\theta}(x)\) and obtain
\eqref{eq:newtonianCapacityInterpretation}. Moreover, this approach shows that the asymptotics given
by \eqref{eq:newtonianCapacityInterpretation} is uniform in \(A\) in the following sense: assuming
that \(A \subset B(0,  \frac{\lVert x \rVert}{2})\), every \(a \in A\)
is such that
\(\frac{1}{2}\lVert x \rVert \leq \lVert x-a \rVert\leq  \frac{3}{2}\lVert x \rVert\)
and \eqref{eq:GreenEstimates} leads to

\begin{equation}\label{eq:ClassicalHittingProbaEstimate}{
\P_{x}(H_{A}<+\infty) \asymp g_{\theta}(x)\sum_{a \in A}\P_{a}(H_{A}^{+}=+\infty) = g_{\theta}(x)\mathrm{Cap}(A)
}\end{equation}
where the hidden constants do not depend on \(A\) or \(x\).

\begin{remark}\label{rmk:SymmetryCapacity}

Note that by symmetry, instead of \eqref{eq:ClassicalLastPassageDecomposition} we could also have made
a \emph{first passage (in \(A\)) decomposition} and use the probability
that a \emph{reversed random walk} \((\bar{S}_{n})_{n}\) never comes
back to \(A\) to get a similar result. This shows that the capacity
associated with \(\bar{\theta}\) is the same as the one associated with
\(\theta\).

\end{remark}

For a \(\Tree\)-walk, the Markov property is not available anymore but
we will still perform a \emph{first passage decomposition} and a \emph{last
passage decomposition} in order to prove the branching equivalent of
\eqref{eq:newtonianCapacityInterpretation}. More precisely, we set

\begin{itemize}
\item
  \(p(x,A)=\P(\Tree^x \cap A \neq \emptyset)\) to be the probability
  that a BRW starting from \(x\) visits \(A\).
\item
  \(\widetilde{p}(x,A)=\P(\widetilde{\Tree}^{x} \cap A \neq \emptyset)\)
  to be the corresponding probability for an adjoint BRW.
\end{itemize}

Note that this notation extends the one from
\Cref{lem:VisitingProbaV1} in the following way:
\(p(x)=p(0,\{ x \})=p(-x,\{ 0 \})\) and
\(\widetilde{p}(x)=\widetilde{p}(0,\{ x \})\). We also consider the
vertices corresponding to the hitting time and exit time of the set
\(A\):

\begin{itemize}
\item
  \(\Omega_{\text{in}}^{A} = \min \{ u \in \Tree: W_{\Tree}(u) \in A \}\)
  is the \emph{hitting vertex} of \(A\);
\item
  \(\Omega_{\text{out}}^{A}= \max \{ u \in \Tree: W_{\Tree}(u) \in A \}\)
  is the \emph{exit vertex} of \(A\).
\end{itemize}

Recall that we use lexicographical order on \(\Tree\), and we
arbitrarily set
\(\Omega_{\text{in}}^{A}=\Omega_{\text{out}}^{A} = \partial \not\in \Tree\)
when \(\Tree^{x} \cap A = \emptyset\).

\begin{proposition}[asymptotics of visiting probabilities]\label{prop:AsymptoticsVisitingProba}

For all finite \(A \subset \Z^d\) and \(a \in A\), we have

\begin{equation*}{
\begin{aligned}
\frac{\P(\Tree^{x} \cap A \neq \emptyset \text{ and } W_{\Tree}^{x}(\Omega_{\text{in}}^{A})=a)}{g_{\theta}(x)} &\xrightarrow[x \to \infty]{} \P(\Tree_{-}^{a} \cap A = \emptyset); \\
& \\
\frac{\P(\Tree^{x} \cap A \neq \emptyset \text{ and } W_{\Tree}^{x}(\Omega_{\text{out}}^{A})=a)}{g_{\theta}(x)} &\xrightarrow[x \to\infty]{} \P(\Tree_{+}^{a} \cap A = \emptyset).
\end{aligned}
}\end{equation*}

\end{proposition}

Summing over \(a \in A\) in \Cref{prop:AsymptoticsVisitingProba} directly gives \Cref{thm:BranchingCapacityVisitingProba},
as well as \Cref{prop:BranchingCapacityFromPast} which is the branching equivalent of
\Cref{rmk:SymmetryCapacity}.

An important tool to obtain \Cref{prop:AsymptoticsVisitingProba} is Zhu's observation that the
path from the origin to the first vertex in \(A\) may be expressed by a
non-branching random walk with a killing rate given by a BRW with a
size-bias at the root. 

\begin{proposition}[\protect{\autocite[Proposition~5.1]{Zhu2017CriticalBranchingRandom}}]\label{prop:ZhuKilledWalk}

On the event \((\Tree^{x} \cap A \neq \emptyset)\), set
\(\Gamma_{\mathrm{in}}=(W_{\Tree}^{x}(\varnothing)=x,\dots, W_{\Tree}^{x}(\Omega_{\mathrm{in}}^{A}))\),
\emph{i.e.} \(\Gamma_{\mathrm{in}}\) is the path in \(\Z^{d}\) starting
from \(x\) and following the steps on the branch from the root to the
hitting vertex of \(A\). Then for any deterministic path \(\gamma\) in
\(\Z^d\) with length $\lvert \gamma \rvert = k$, starting at \(\gamma(0)=x\) and ending at \(\gamma(k) \in A\),

\begin{equation*}{
\P(\Tree^{x} \cap A \neq  \emptyset \text{ and } \Gamma_{\mathrm{in}} = \gamma) = s(\gamma)\prod_{i=0}^{k-1} \bigl( 1-\widetilde{p}(\gamma(i),A) \bigr),
}\end{equation*}
where
\(s(\gamma)=\prod_{i=0}^{k-1} \theta \bigl(\gamma(i+1)-\gamma(i) \bigr)\)
is the probability that a classical random walk with jumps distribution
\(\theta\) realizes the path \(\gamma\).

\end{proposition}

\begin{remark}
We quickly recall Zhu's argument for completeness. On the event \((\Tree^{x} \cap A \neq \emptyset)\), we consider the genealogical line from $\varnothing$ to $\Omega_{\mathrm{in}}^{A}$ as a finite spine and we consider the children of those vertices located at the left and the right of this spine. We let $(l_i,r_i)_{0 \leq i \leq k-1}$ denote those numbers of children. Then we compute the probability of seeing a path $\gamma=(\gamma(0),\ldots,\gamma(k))$ with $\gamma(i) \in A \Longleftrightarrow i=k$  by summing over all possibilities on $(l_i,r_i)_{0 \leq i \leq k-1}$. We have to take into account that the left children cannot produce a subtree reaching $A$, by definition of $\Omega_{\mathrm{in}}^{A}$, thus we get
\begin{align*}
\P(\Tree^{x} \cap A \neq  \emptyset \text{ and } \Gamma_{\mathrm{in}} = \gamma) &= s(\gamma) \times \prod_{i=0}^{k-1}\sum_{l_i,r_i \geq 0} \mu(1+l_i+r_i) \E\bigl(1-p(\gamma(i)+S_1,A)\bigr)^{l_i} \\
&= s(\gamma) \prod_{i=0}^{k-1}\sum_{l_i\geq 0} \widetilde{\mu}(l_i) \E\bigl(1-p(\gamma(i)+S_1,A)\bigr)^{l_i}, \\
&= s(\gamma)\prod_{i=0}^{k-1} \bigl( 1-\widetilde{p}(\gamma(i),A) \bigr),
\end{align*}
where we simply used that $\widetilde{\mu}(k) = \mu(\llbracket k+1,+\infty\llbracket)$ and a decomposition of $\widetilde{\Tree}$ at its root.
\end{remark}

As mentioned above, the term
$$s_{A}(\gamma):=s(\gamma)\prod_{i=0}^{k-1} \bigl(1-\widetilde{p}(\gamma(i),A) \bigr)$$
is the probability that a killed random walk realizes the path
\(\gamma\), with a killing probability at each site \(y \in \Z^{d}\)
given by \(\widetilde{p}(y,A)\). We will make use of the corresponding
Green function \(g_{A}\) defined by

\begin{equation*}{
g_A(x,y)=\sum_{\gamma:x \to y} s_A(\gamma).
}\end{equation*}

According to \Cref{prop:ZhuKilledWalk} and the fact that \(s_{A}(\gamma)=0\)
whenever the path \(\gamma\) encounters \(A\) before its endpoint, for
any \(a \in A\) we may rewrite the probability we want to estimate as

\begin{equation}\label{eq:HittingProbaAsGreenFunction}{
\P(\Tree^{x} \cap A \neq  0 \text{ and } W^{x}_{\Tree}(\Omega_{\mathrm{in}}^{A})=a) = \sum_{\gamma:x \to a} s_A(\gamma) = g_{A}(x,a).
}\end{equation}

\Cref{prop:ZhuKilledWalk} may also be seen as a spinal decomposition of a \(\Tree\)-walk,
as \(s_{A}(\gamma)\) is the probability that a tree-indexed walk, where
the underlying tree has a finite spine of length
\(\lvert \gamma \rvert\) and adjoint Bienaymé trees grafted to each
vertex of this spine, realizes \(\gamma\) along its spine while the
bushes stay away from \(A\). This event is illustrated in \Cref{fig:SpinalDecomposition} below.

\begin{figure}[!h]
    \centering
    \includegraphics[width=0.4\linewidth]{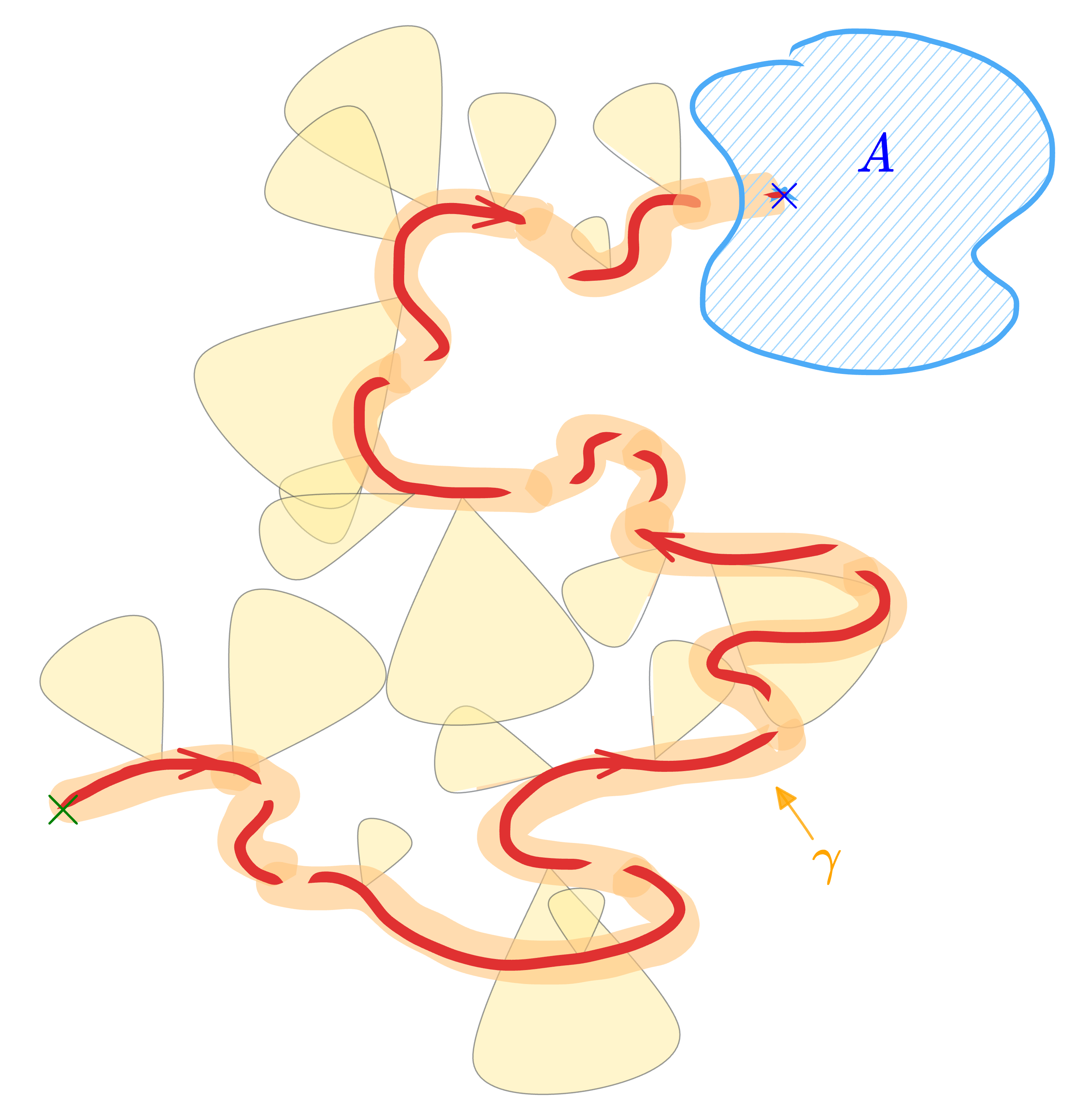}
    \caption{\centering A tree-indexed walk with a spine realizing a path $\gamma \subset A^{c}$ (in red) ending in $A$. The discontinuities of this path emphasize that the walk might make big jumps. Moreover, its adjoint Benaymé trees grafted to the left of the spine only  occupy positions (the yellow bubbles) outside of $A$.}
    \label{fig:SpinalDecomposition}
\end{figure}

This latter interpretation will link
\(\P(\Tree^{x} \cap A \neq  0 \text{ and } W^{x}_{\Tree}(\Omega_{\mathrm{in}}^{A})=a)\)
with \(\Tree_{-}\), by letting the spine grows to infinity, while the
former interpretation enables to make some approximation by the standard
Green function \(g_{\theta}\) through the following \Cref{lem:GreenFunctionApproximation} .
Combining the two of them will give \Cref{prop:AsymptoticsVisitingProba}.

\begin{lemma}\label{lem:GreenFunctionApproximation}

\begin{equation*}{
\lim_{ R \to +\infty } \sup_{\lVert x \rVert, \lVert y \rVert \geq  R } \left\vert\frac{g_A(x,y)}{g_{\theta}(x,y)}-1\right\vert \to 0,
}\end{equation*}
where we recall that \(g_{\theta}(x,y)=g_{\theta}(y-x)\).

\end{lemma}

This lemma is actually the technical part of the proof, where the
parameter \(\alpha\) will play a role. This comes from a control of the killing
rate involved in \(g_{A}\). Indeed, we have already bounded this killing rate in a special case with \eqref{eq:UpperBounfVisitingProbability}, and this can be extended into
the following lemma. 

\begin{lemma}[First estimates on the asymptotic probability to visit a finite set $A$]\label{lem:FirstEstimatesVisitingProba}

Fix \(x \in \Z^d\) and \(A \subset \Z^d\) finite.

\begin{enumerate}[(i)]
\item
the probability that a branching random walk visits \(A\) is bounded above as follows:
\begin{equation*}{
  p(x,A) \preceq   \#A \times g_{\theta}(x).
  }\end{equation*}

\item
  Visiting probabilities for a branching random walk and an adjoint
  branching random walks are asymptotically related:
\begin{equation*}{
\widetilde{p}(x,A) \sim_{x \to \infty} p(x,A)^{\alpha-1}L(p(x,A)^{-1}).
}\end{equation*}
\end{enumerate}

As a consequence
\(\widetilde{p}(x,A) \preceq g_{\theta}(x)^{\alpha-1}L(g_{\theta}(x)^{-1})\)
(where the hidden constant depends on \(A\)).

\end{lemma}

This last upper bound enables to control the killing rate at a given
point, but then we must integrate this over all possible paths. Zhu
achieved this by first considering some close points \(x,y\) and
excluding some bad paths and then extending this to a more general
situation by showing that the other contributions are negligible. This
approach could work here, but we propose an alternative
proof seemingly more straightforward where we directly integrate over all possible
paths.

A last technical point to get \Cref{prop:AsymptoticsVisitingProba} is the control of the big
jumps of the walk. We will need the following result.

\begin{lemma}[Overshoot Lemma]\label{lem:OvershootLemma1}

There is a constant \(C\) such that for all $a \in \Z^d$ and all radius \(r> \lVert a \rVert, r'>0\), 
\begin{equation*}{
\sum_{y:\lVert y \rVert > r+r' } \sum_{\gamma:y \to a \subset B(0,r)} s(\gamma) \leq  Cr^{2}\P_{0}(\lVert S_{1} \rVert > r').
}\end{equation*}

\end{lemma}
Let us first prove \Cref{prop:AsymptoticsVisitingProba} before coming back to those lemmas.

\begin{proof}[Proof of \Cref{prop:AsymptoticsVisitingProba}.]

We first relate the probabilities \(s_{A}\) with a \(\Tree_{-}\)-walk. Choose a large radius \(R>0\) so that
\(A \subset B(0,R)\) and decompose this \(\Tree_{-}\)-walk according to \(\sigma_{R}\)
the first vertex of the spine (first but in \emph{reversed genealogical
order}) which steps outside of \(B(0,R)\). Since
\(\sigma_{R+1} \to \sigma_{R}\) and for any vertex
\(v \in \Tree_{-}\), there is \(R\) such that \(\sigma_{R} \to v\), we
have the monotone convergence

\begin{equation*}{
\P(\Tree_{-}^{a} \cap A = \emptyset) = \lim_{ R \to +\infty } \P(\forall u \in \Tree_{-} \text{ such that } \sigma_{R} \to  u, W_{\Tree_{-}}^{a}(u) \not\in A).
}\end{equation*}

Let us now focus on the subtree formed by the descendants of \(\sigma_{R}\) in
\(\Tree_{-}\), which we consider as a tree rooted at \(\sigma_{R}\) and with a distinguished point still at $c_0$. We recall that the past bushes $\mathcal{B}_{k,-}$ grafted to the spine of \(\Tree_{-}\) are independent and distributed as an adjoint Bienaymé tree $\widetilde{\Tree}$. Thus on the event $(\forall u \in \Tree_{-} \text{ such that } \sigma_{R} \to  u, W_{\Tree_{-}}^{a}(u) \not\in A)$, we
see a tree-indexed walk with a finite spine (the genealogical line from $\sigma_R$ to $c_0$) which realizes along this spine some path \(\gamma \subset B(0,R)\) ending at \(a\) while the
adjoint trees grafted on this spine avoid \(A\). See \Cref{fig:LastPassageDecomposition} for an illustration. By comparing it with \Cref{fig:SpinalDecomposition}, we get that for a given path
\(\gamma\), the probability of seeing this path \(\gamma\) precisely is
\(s_{A}(\gamma)=s(\gamma)\prod_{i=0}^{k-1} \bigl(1-\widetilde{p}(\gamma(i),A) \bigr)\).
A decomposition according to which site \(y\) is occupied by
\(\sigma_{R}\) gives

\begin{equation}\label{eq:DecomposingAvoidAinPast}{
\P(\Tree_{-}^{a} \cap A =\emptyset) = \lim_{ R \to +\infty } \sum_{y: \lVert y \rVert > R} \sum_{\gamma:y \to a \subset B(0,R)} s_{A}(\gamma).
}\end{equation}

\begin{figure}[!h]
    \centering
    \includegraphics[width = 0.75\linewidth]{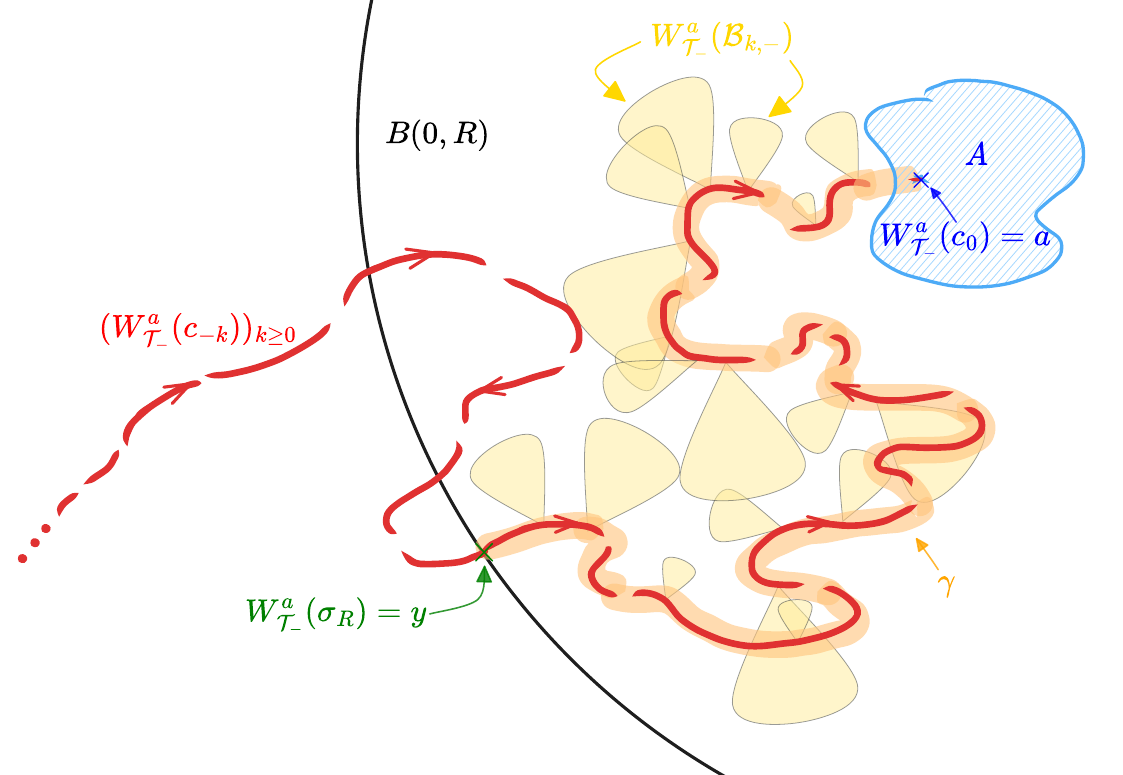}
    \caption{\centering A decomposition of a walk $W^a_{\Tree_{-}}$, indexed by $\Tree_{-}\cup\{c_0\}$, according to its \textit{last passage} outside of the ball $B(0,R)$ along the spine. The red path is the walk along the spine, with arrows to recall the genealogical order along this spine. By construction, it ends at $W_{\Tree_{-}}^{a}(c_0)=a$. The position $W_{\Tree_{-}}^{a}(\sigma_R)=y$ is the last (in genealogical order) position of the walk outside of $B(0,R)$, and $\gamma$ denotes the path realized by the spine from $\sigma_R$ to $c_0$. Finally, the yellow bubbles corresponds to the positions occupied by the past bushes grafted to the left of the spine.}
    \label{fig:LastPassageDecomposition}
\end{figure}

Recalling \eqref{eq:HittingProbaAsGreenFunction}, we can perform a similar decomposition of
\(\P(\Tree^{x} \cap A \neq  0 \text{ and } W^{x}_{\Tree}(\Omega_{\mathrm{in}}^{A})=a)\)
by splitting any path from \(x\) to \(a\) at its last step outside of
\(B(0,R)\) and we get

\begin{equation*}{
g_{A}(x,a) = \sum_{y: \lVert y \rVert > R}g_A(x,y)\sum_{\gamma:y\to a \subset B(0,R)} s_A(\gamma).
}\end{equation*}
Fix some \(\varepsilon>0\). For large \(\lVert x \rVert\), we can take
\(R'=\varepsilon\lVert x \rVert > R\) to apply \Cref{lem:OvershootLemma1} and get

\begin{equation*}{
\sum_{y: \lVert y \rVert > R+\varepsilon \lVert x \rVert }g_A(x,y)\sum_{\gamma:y\to a \subset B(0,R)} s_A(\gamma) \leq  Cg_{\theta}(0)R^{2}\P\left( \lVert S_{1} \rVert > \varepsilon\lVert x \rVert   \right).
}\end{equation*}
In addition, by assumption on $\theta$ we have
\(\P(\lVert S_{1} \rVert>\lVert x \rVert) =o(\lVert x \rVert^{2-d})=o(g_{\theta}(x))\).
We deduce that

\begin{equation*}{
\begin{aligned}
\varlimsup_{ x \to \infty } \frac{g_{A}(x,a)}{g_{\theta}(x)} &= \varlimsup_{ x \to \infty } \sum_{y: R <\lVert y \rVert <R+\varepsilon \lVert x \rVert  } \frac{g_A(x,y)}{g_{\theta}(x)} \sum_{\gamma:y\to a \subset B(0,R)} s_A(\gamma) \\
& \leq  \left( \varlimsup_{ x \to \infty } \sup_{y: R <\lVert y \rVert <R+\varepsilon \lVert x \rVert  } \frac{g_A(x,y)}{g_{\theta}(x)} \right)\times \sum_{y: \lVert y \rVert > R}\sum_{\gamma:y \to a \subset B(0,R)} s_{A}(\gamma) \\
& \leq  (1 - \varepsilon)^{2-d} \times \sum_{y: \lVert y \rVert > R}\sum_{\gamma:y \to a \subset B(0,R)} s_{A}(\gamma).
\end{aligned}
}\end{equation*}
where we used $g_A(x,y) \leq g_{\theta}(x,y)$ and \eqref{eq:GreenEstimates}. First send \(\varepsilon\) to \(0\), and then send \(R\) to \(+\infty\)
with \eqref{eq:DecomposingAvoidAinPast} to obtain

\begin{equation*}{
\varlimsup_{ x \to \infty } \frac{g_{A}(x,a)}{g_{\theta}(x)} \leq  \P(\Tree_{-}^{a} \cap A =\emptyset).
}\end{equation*}

In the same way ,
\begin{equation*}{
\begin{aligned}
\varliminf_{ x \to \infty } \frac{g_{A}(x,a)}{g_{\theta}(x)} 
& \geq  \left( \varliminf_{ x \to \infty } \inf_{y: R <\lVert y \rVert <R+\varepsilon \lVert x \rVert  } \frac{g_A(x,y)}{g_{\theta}(x)} \right)\times \sum_{y: \lVert y \rVert > R}\sum_{\gamma:y \to a \subset B(0,R)} s_{A}(\gamma) \\
& \geq \left( \varliminf_{ x \to \infty }\inf_{y: R <\lVert y \rVert <R+\varepsilon \lVert x \rVert  } \frac{g_{\theta}(x,y)}{g_{\theta}(x)} \right)\times  \inf_{\lVert x' \rVert,\lVert y \rVert \geq R  } \frac{g_{A}(x',y)}{g_{\theta}(x',y)} \times \sum_{y: \lVert y \rVert > R}\sum_{\gamma:y \to a \subset B(0,R)} s_{A}(\gamma) \\
& \geq  (1 + \varepsilon)^{2-d}\times  \inf_{\lVert x' \rVert,\lVert y \rVert \geq R  } \frac{g_{A}(x',y)}{g_{\theta}(x',y)} \times \sum_{y: \lVert y \rVert > R}\sum_{\gamma:y \to a \subset B(0,R)} s_{A}(\gamma).
\end{aligned}
}\end{equation*}
Again, we send \(\varepsilon\) to \(0\), and then thanks to \Cref{lem:GreenFunctionApproximation} and \eqref{eq:DecomposingAvoidAinPast} we send \(R\) to \(+\infty\). It gives the desired limiting lower bound.

\begin{equation*}{
\varliminf_{ x \to \infty } \frac{g_{A}(x,a)}{g_{\theta}(x)} \geq  \P(\Tree_{-}^{a} \cap A =\emptyset).
}\end{equation*}
Thus, recalling \eqref{eq:HittingProbaAsGreenFunction}, we actually have the first assertion of \Cref{prop:AsymptoticsVisitingProba}, namely 
\begin{equation*}
    g_A(x,a) \sim \P(\Tree_{-}^{a} \cap A =\emptyset)\times g_{\theta}(x) \text{ as } x \to \infty.
\end{equation*}

The argument is quite similar to get the asymptotics of
\(\P(\Tree^{x} \cap A \neq \emptyset \text{ and } W_{\Tree}^{x}(\Omega_{\text{out}}^{A})=a)\).
On the event \((\Tree^{x} \cap A \neq  \emptyset)\), one can consider
the path \(\Gamma_{\mathrm{out}}\) realised between the root and
\(\Omega_{\mathrm{out}}^{A}\) and, as shown by Zhu, we have an
equivalent of \Cref{prop:ZhuKilledWalk} :

\begin{equation*}{
\P(\Tree^{x} \cap A \neq  \emptyset \text{ and } \Gamma_{\mathrm{out}} = \gamma) = s(\gamma) \bigl( 1-p^{0}(\gamma(k),A) \bigr) \prod_{i=0}^{k-1} \bigl( 1-\widetilde{p}^{0}(\gamma(i),A) \bigr),
}\end{equation*}
where
\begin{itemize}
    \item $k$ is the length of the path $\gamma$,
    \item \(p^{0}(x,A) =\P((\Tree\setminus \{ \varnothing \})^x \cap A \neq \emptyset)\),
    \item \(\widetilde{p}^{0}(x,A) =\P((\widetilde{\Tree}\setminus \{ \varnothing \})^x \cap A \neq \emptyset)\).
\end{itemize}
Note that for \(y \not\in A\), \(p^{0}(y,A)=p(y,A)\) and the same holds
for adjoint trees.

By setting
\(s_{A}^{0}(\gamma)= s(\gamma) \prod_{i=0}^{k-1} \bigl(1-\widetilde{p^{0}}(\gamma(i),A) \bigr)\)
to interpret this as a killed random walk, this can be written as
\(\P(\Tree^{x} \cap A \neq  \emptyset \text{ and } \Gamma_{\mathrm{out}}=\gamma) = \bigl(1-p^{0}(\gamma(k),A) \bigr)s_{A}^{0}(\gamma)\)
and then

\begin{equation*}{
\P(\Tree^{x} \cap A \neq  0 \text{ and } W^{x}_{\Tree}(\Omega_{\mathrm{out}}^{A})=a) = \sum_{\gamma:x \to a} \bigl(1-p^{0}(\gamma(k),A) \bigr)s^{0}_A(\gamma) = \bigl(1-p^{0}(a,A) \bigr)g^{0}_{A}(x,a).
}\end{equation*}

Moreover, \(\bigl(1-p^{0}(\gamma(k),A) \bigr)s_{A}^{0}(\gamma)\) can
also be interpreted as the probability that some tree-indexed random
walk realizes the path \(\gamma\) on its spine while the Bienaymé tree
grafted to the last vertex on the spine and the adjoint Bienaymé trees
grafted to the rest of the spine do not reach \(A\), except maybe on
the spine itself. By decomposing \(\Tree_{+}\) in the same way that we
decomposed \(\Tree_{-}\) this leads to

\begin{equation*}{
\P(\Tree_{+}^{a} \cap A =\emptyset) = \bigl( 1-p^{0}(a,A) \bigr)\lim_{ R \to +\infty } \sum_{y: \lVert y \rVert > R} \sum_{\gamma:y \to a \subset B(0,R)} s_{A}^{0}(\gamma).
}\end{equation*}

Since \(s_{A}\) and \(s_{A}^{0}\) are almost the same, one can also
adapt \Cref{lem:GreenFunctionApproximation} to the corresponding Green function \(g_{A}^{0}\),
thus we are able to argue in the exact same way to get

\begin{equation*}{
\P(\Tree^{x} \cap A \neq \emptyset \text{ and } W_{\Tree}^{x}(\Omega_{\text{out}}^{A})=a) \sim g_{\theta}(x) \P(\Tree_{+}^{a} \cap A = \emptyset).
}\end{equation*}
This concludes the proof of \Cref{prop:AsymptoticsVisitingProba}.
\end{proof}

We now prove the technical \Cref{lem:OvershootLemma1,lem:FirstEstimatesVisitingProba,lem:GreenFunctionApproximation}.

\begin{proof}[Proof of \Cref{lem:OvershootLemma1}.]
Note that by reversing time, we can express the sum with the reversed
random walk \((\bar{S}_{n})_{n}\) :

\begin{equation*}{
\sum_{y :\lVert y \rVert > r+r' } \sum_{\gamma:y \to a \subset B(0,r)} s(\gamma) = \P_{a}(\bar{S}_{\bar{H}_{>r}} \not\in B(0,r+r')),
}\end{equation*}
where $\bar{H}_{>r}$ is the hitting time of $B(0,r)^c$ by $\bar{S}$.
By considering the position \(z \in B(0,r)\) occupied by the walk before
leaving this ball we get

\begin{equation*}{
\begin{aligned}
\P_{a}(\bar{S}_{\bar{H}_{>r}} \not\in B(0,r+r')) &= \sum_{z \in B(0,r)} \sum_{\gamma:a \to z \subset B(0,r)} \bar{s}(\gamma)\P_{z}(\lVert \bar{S}_{1} \rVert > r+r' ) \\
& \leq \P_{0}(\lVert S_{1} \rVert > r') \sum_{z \in B(0,r)}g_{\theta}(z,a) \\
& \leq C r^{2} \P_{0}(\lVert S_{1} \rVert > r'),
\end{aligned}
}
\end{equation*}
where $\bar{s}(\gamma)$ is the probability that $\bar{S}$ realizes a path $\gamma$. We have also used  that \(\sum_{z \in B(0,r)}g_{\theta}(z,a)\), which is the
total time spent in \(B(0,r)\) by a (reversed) walk started at \(a\), is
bounded above by \(Cr^{2}\) due to \eqref{eq:GreenEstimates}.
\end{proof}

\begin{proof}[Proof of \Cref{lem:FirstEstimatesVisitingProba}.]

The first assertion is a direct application of the \(1^{\text{st}}\)
moment method: for any \(a \in A\), we have
\(\P(a \in \Tree^{x}) \leq  g_{\theta}(a-x)\). For the second one, we
decompose \(\widetilde{\Tree}\) into its subtrees grafted to the root,
and we use that the offspring distribution of this root has generating
function \(\frac{1-\gen_{\mu}(s)}{1-s}\). For \(x \not\in A\) we
get

\begin{equation*}{
1-\widetilde{p}(x,A) = \frac{1-\gen_{\mu}\big(1-\mathbb E(p(x+S_{1},A))\big)}{\mathbb E(p(x+S_{1},A))}. 
}\end{equation*}

Set \(q(x)=\mathbb E(p(x+S_{1},A))\), clearly \(q(x) \to 0\) as
\(\lVert x \rVert \to +\infty\) hence from \eqref{eq:GenMuStable} we get

\begin{equation}\label{eq:AjointVisitingProba}{
\widetilde{p}(x) = \frac{\gen_{\mu}\left(1-q(x)\right)-1 +q(x)}{q(x)} \sim q(x)^{\alpha-1}L\bigl(q(x)^{-1}\bigr).
}\end{equation}

Moreover, the same decomposition at the root applied to \(\Tree\) yields
\(1-p(x,A) = \gen_{\mu}(1-q(x))\), thus \(q(x) \sim p(x,A)\) and we have
the desired equivalent for \(\widetilde{p}\). Combining this with the
first assertion and Potter's bound (\Cref{prop:PotterBounds}) gives the last bound.
\end{proof}

\begin{proof}[Proof of \Cref{lem:GreenFunctionApproximation}.]

First, we clearly have \(g_A(x,y) \leq g_{\theta}(x,y)\), so we actually
have to establish that 

\begin{equation}\label{eq:GoalUpperbound1}
    \text{when } \lVert x \rVert,\lVert y \rVert \geq  R,\ g_{\theta}(x,y)-g_A(x,y) \leq   \varepsilon(R)g_{\theta}(x,y),
\end{equation}
where $\varepsilon$ is a function such that $\varepsilon(R) \to 0$ for large $R$.

Fix some \(R>0\). We note that, independently of this \(R\), we have

\begin{equation*}{
\begin{aligned}
g_{A}(x,y) &= \sum_{n \in \N} \E_{x}\left( \1_{S_{n}=y}\prod_{i=0}^{n-1}(1-\widetilde{p}(S_{i},A)) \right) \\
& \geq  \sum_{n \in \N} \E_{x}\left( \1_{S_{n}=y}\left( 1-\sum_{i=0}^{n-1}\widetilde{p}(S_{i},A) \right) \right) \\
& \geq  g_{\theta}(x,y) - \sum_{i \in \N}  \E_{x} \left( \widetilde{p}(S_{i},A) \sum_{n > i} \1_{_{S_{n}=y}}\right) \\
& \geq  g_{\theta}(x,y) - \sum_{i \in \N} \E_{x}(\widetilde{p}(S_{i},A)g_{\theta}(S_{i},y)).
\end{aligned}
}\end{equation*}
As a consequence,

\begin{equation}\label{eq:BoundGreenDifference}{
g_{\theta}(x,y)-g_A(x,y) \leq  \sum_{i \in \N} \E_{x}(\widetilde{p}(S_{i},A)g_{\theta}(S_{i},y)) \leq  \sum_{z \in \Z^{d}} \widetilde{p}(z,A)g_{\theta}(x,z)g_{\theta}(z,y).
}\end{equation}

To control the left-hand side, we bound $\widetilde{p}(\cdot,A)$ from above with a more regular function $\pi$. More precisely, in sight of \Cref{lem:FirstEstimatesVisitingProba} and \Cref{prop:MonotoneEquivalent} (and \Cref{Assum:SufficientConditionTransience} in the special case $d=\frac{2\alpha}{\alpha-1}$) we have the following statement.

\begin{claim}\label{claim:pi}
There is a non-increasing $-(\alpha-1)(d-2)$-varying function $\pi$ satisfying $\sum_{r \in \N} r \pi(r) < +\infty$ and such that
\begin{equation*}
    \forall z \in \Z^{d},\ \widetilde{p}(z,A) \leq \pi(\lVert z \rVert ).
\end{equation*}
Note that $\pi$ may depend on $A$.
\end{claim}

Given \eqref{eq:BoundGreenDifference} and \Cref{claim:pi}, to prove \eqref{eq:GoalUpperbound1} amounts to show that when $\lVert x \rVert, \lVert y \rVert \geq  R$ we have

\begin{equation}\label{eq:GoalUpperbound2}
    \sum_{z \in \Z^{d}} \pi(\lVert z \rVert)g_{\theta}(z-x)g_{\theta}(y-z) \leq  \varepsilon(R) g_{\theta}(y-x).
\end{equation}

Without loss of generality, we may assume that $R \leq  \lVert  y \rVert \leq  \lVert  x \rVert$.  In particular  this implies $\lVert x-y \rVert \leq 2\lVert x \rVert$. 

As a preliminary result, we make a simple computation to show the existence of a constant $C$ such that for all $a,x \in \Z^{d}$ with $\lVert a \rVert \leq  2 \lVert x \rVert$, we have

\begin{equation}\label{eq:BoundPartialConvolution}
    \sum_{z \in B\left( 0,\frac{\lVert x \rVert}{2} \right)} g_{\theta}(z)g_{\theta}(a-z) \leq C \lVert x \rVert^{2}g_{\theta}(a).
\end{equation}
Indeed, when $z \in B(a,\lVert a \rVert/2)$ we have $g_{\theta}(z) \leq  C_{1}g_{\theta}(a)$  and else we have $g_{\theta}(a-z) \leq  C_{2}g_{\theta}(a)$, thus we get

\begin{equation*}
    \begin{aligned}
\sum_{z \in B\left( 0,\frac{\lVert x \rVert}{2} \right)} g_{\theta}(z)g_{\theta}(a-z) &\leq C_{1}g_{\theta}(a)\sum_{z \in B\left( a,\frac{\lVert a \rVert}{2} \right)} g_{\theta}(a-z) + C_{2}g_{\theta}(a)\sum_{z \in B\left( 0,\frac{\lVert x \rVert}{2} \right)} g_{\theta}(z) \\
& \leq  C'_{1}g_{\theta}(a)\lVert a \rVert ^{2} +C'_{2}g_{\theta}(a)\lVert x \rVert ^{2} \\
& \leq  Cg_{\theta}(a)\lVert x \rVert ^{2}.
\end{aligned}
\end{equation*}

We then control $\sum_{z \in \Z^{d}} \pi(\lVert z \rVert)g_{\theta}(z-x)g_{\theta}(y-z)$ by splitting the sum according to whether $z \in B\left( x,\frac{\lVert x \rVert}{2} \right)$ or not.

For the first part where $z \in B\left( x,\frac{\lVert x \rVert}{2} \right)$, we have $\pi(\lVert z \rVert) \leq  C \pi(\lVert x \rVert)$ by Potter's bound (\Cref{prop:PotterBounds}), hence we get

\begin{equation*}
    \begin{aligned}
\sum_{z \in B\left( x,\frac{\lVert x \rVert}{2} \right)} \pi(\lVert z \rVert)g_{\theta}(z-x)g_{\theta}(y-z) &\leq C \pi(\lVert x \rVert ) \sum_{z \in B\left( 0,\frac{\lVert x \rVert}{2} \right)} g_{\theta}(z)g_{\theta}((y-x)-z) \\
&  \leq C \lVert x \rVert ^{2}\pi(\lVert x \rVert ) \times  g_{\theta}(y-x),
\end{aligned}
\end{equation*}
where we applied \eqref{eq:BoundPartialConvolution}  to $a=y-x$.

For the remaining part where $z \not\in B\left( x, \frac{\lVert x \rVert}{2} \right)$, we have $\lVert x-z \rVert \geq \frac{\left\lVert  x  \right\rVert}{2} \geq \frac{\lVert x-y \rVert}{4}$ which implies $g_{\theta}(z-x) \leq C g_{\theta}(y-x)$. We may thus write

\begin{equation*}
    \sum_{z \not\in B\left( x, \frac{\lVert x \rVert}{2} \right)} \pi(\lVert z \rVert)g_{\theta}(z-x)g_{\theta}(y-z) \leq  Cg_{\theta}(y-x) \underbrace{ \sum_{z \in \Z^{d}} \pi(\lVert z \rVert )g_{\theta}(y-z) }_{ (\pi*g_{\theta})(y) }.
\end{equation*}

Combining the last two displays yields

\begin{equation*}
    \sum_{z \in \Z^{d}} \pi(\lVert z \rVert)g_{\theta}(z-x)g_{\theta}(y-z) \leq  C \bigl(\lVert x \rVert ^{2}\pi(\lVert x \rVert ) +  (\pi*g_{\theta})(y)\bigr) \times g_{\theta}(y-x),
\end{equation*}
where $(\pi*g_{\theta})(y) = \sum_{z \in \Z^{d}} \pi(\lVert z \rVert )g_{\theta}(y-z)$. Hence it remains to show that 

\begin{equation}\label{eq:GoalUpperbound3}
    \lVert x \rVert ^{2}\pi(\lVert x \rVert ) +  (\pi*g_{\theta})(y) \leq \varepsilon(R),
\end{equation}
for some function $\varepsilon$ going to $0$ at $+\infty$ in order to get \eqref{eq:GoalUpperbound2}. We prove this with $\varepsilon(R) = C \sum_{r \geq 1} \frac{r^{d-1}\pi(r)}{(R \vee r)^{d-2}}$, which does vanish at $+\infty$ by Lebesgue's dominated convergence since $\sum_{r \in \N} r \pi(r) < +\infty$ (recall \Cref{claim:pi}).

First, we simply observe that since $\sum_{r \in \N} r \pi(r) < +\infty$ we can apply Karamata's \Cref{thm:Karamata} to get

\begin{equation*}
    \lVert x \rVert ^{2}\pi(\lVert x \rVert ) \leq C\sum_{r \geq \lVert x \rVert } r\pi(r) \leq  C \sum_{r\geq  R} r\pi(r) \leq  C \sum_{r \geq 1} \frac{r^{d-1}\pi(r)}{(R \vee r)^{d-2}}.
\end{equation*}

Then we deal with $(\pi*g_{\theta})(y)$  by splitting again the sum between $z \in B\left( y, \frac{\lVert y \rVert}{2} \right)$ and $z \not\in B\left( y, \frac{\lVert y \rVert}{2} \right)$. As earlier, for the first part we have

\begin{equation*}
    \sum_{z \in B\left( y, \frac{\lVert y \rVert}{2} \right)} \pi(\lVert z \rVert ) g_{\theta}(y-z) \leq C \lVert y \rVert^{2} \pi(\lVert y \rVert ) \leq C \sum_{r \geq 1} \frac{r^{d-1}\pi(r)}{(R \vee r)^{d-2}}.
\end{equation*}

Finally, on the remaining part, $\lVert y-z \rVert > \frac{\lVert y \rVert}{2}$ implies both $\lVert y-z \rVert > \frac{R}{2}$ and $\lVert y-z \rVert > \frac{\lVert z \rVert}{3}$ hence we have $g_{\theta}(y-z) \leq \frac{C}{(R \vee \lVert z \rVert)^{d-2}}$. This gives

\begin{equation*}
    \sum_{z \not\in B\left( y, \frac{\lVert y \rVert}{2} \right)} \pi(\lVert z \rVert ) g_{\theta}(y-z) \leq C \sum_{r \geq 1} \frac{r^{d-1}\pi(r)}{(R \vee r)^{d-2}},
\end{equation*}
which ends the proof of \eqref{eq:GoalUpperbound3}. We conclude that \eqref{eq:GoalUpperbound2} and thus \eqref{eq:GoalUpperbound1} hold, which is the desired result.
\end{proof}

\subsection{Visiting probabilities for adjoint and infinite branching random walk}\label{visiting-probability-for-adjoint-and-infinite-branching-random-walk}

The main result on the visiting probability of \(\Tree\)-walk,
\Cref{thm:BranchingCapacityVisitingProba}, enables to study the corresponding
\(\Tree_{\infty}\)-walk. Let us first consider the alternative Green
function introduced in \Cref{eq:AlternativeGreenFunction}.

\begin{proof}[Proof of
\Cref{prop:AlternativeGreenFunction}.]

By decomposing our infinite tree into its finite bushes, we see that

\begin{equation}\label{eq:DecompositionAlternativeGreen}{
G(x) = \sum_{k \geq 1} \sum_{y \in \Z^d} \P(\bar{S}_{k}=y)\widetilde{p}(x-y) = (g_{\bar{\theta}}*\widetilde{p})(x) - \widetilde{p}(x).
}\end{equation}
Moreover, \Cref{thm:BranchingCapacityVisitingProba} and \Cref{lem:FirstEstimatesVisitingProba} immediately give that for
every finite set \(A \subset \Z^{d}\),

\begin{equation}\label{eq:AsymptoticAdjointVisitingProba}{
\widetilde{p}(x,A) \sim_{x \to \infty} \bcp(A)^{\alpha-1}g_{\theta}(x)^{\alpha-1}L(g_{\theta}(x)^{-1}).
}\end{equation}
In particular, this is a \(-(\alpha-1)(d-2)\)-varying function. Taking
\(A=\{ 0 \}\) yields the asymptotics of \(\widetilde{p}(x)\), and then
one can directly compute the asymptotics of
\((g_{\bar{\theta}}*\widetilde{p})(x)\) via a generic result on convolution of regularly varying functions, \Cref{lem:ConvolutionRegular}. We get that \((g_{\bar{\theta}}*\widetilde{p})(x)\) is
\(-((\alpha-1)(d-2)-2)\)-varying, thus \(\widetilde{p}(x)\) is
negligible in front of this term and we get \Cref{prop:AlternativeGreenFunction}.
\end{proof}

\begin{proof}[Proof of
\Cref{thm:VisitingProbabilityInfinite}.]

Following Sousi's approach, we decompose \(\Tree_{\infty}\) into a
collection of finite bushes. We can again use \Cref{thm:BranchingCapacityVisitingProba}, or more
precisely its consequence \eqref{eq:AsymptoticAdjointVisitingProba}, to estimate that
\(\Tree_{\infty}\) reaches some finite set \(A \subset \Z^{d}\) in the
future or in the past.

First, we claim that it is sufficient to prove the desired result for a
simpler infinite tree, say \(\Tree_{*}\), where there is a spine and
every vertex on the spine reproduces according to \(\widetilde{\mu}\),
including the cursor. Indeed, since
\(\P(\mathcal{S}^{x} \cap A \neq  \emptyset) \asymp \P(\Tree^{x} \cap A \neq  \emptyset) \asymp \frac{1}{\lVert x \rVert^{d-2}}\)
and
\(\P(\widetilde{\Tree}^{x} \cap A \neq  \emptyset) \asymp \frac{L(\lVert x \rVert^{d-2})}{\lVert x \rVert^{(\alpha-1)(d-2)}}\)
are negligible in front of \(G(x)\), as soon as
\(\P(\Tree_{*}^{x} \cap A \neq  \emptyset) \asymp G(x)\) we have for
large \(x\)

\begin{equation*}{
\P(\Tree_{+}^{x} \cap A \neq  \emptyset) \sim \P(\Tree_{-}^{x} \cap A \neq  \emptyset) \sim \P(\Tree_{*}^{x} \cap A \neq  \emptyset).
}\end{equation*}

We can then decompose this tree \(\Tree_{*}\) according to the first
vertex on the spine whose bush (still denoted by \(\mathcal{B}_{k}\) for
\(k \geq 0\)) hits \(A\). This gives

\begin{equation*}{
\begin{aligned}
\P(\Tree_{*}^{x} \cap A \neq  \emptyset) &= \sum_{n \in \N} \P(\forall k < n, W_{\Tree_{*}}^{x}(\mathcal{B}_{k}) \cap A = \emptyset \text{ and } W_{\Tree_{*}}^{x}(\mathcal{B}_{n}) \cap A \neq   \emptyset) \\
& =\sum_{n \in \N} \sum_{y \in \Z^{d}} \P(W_{\Tree_{*}}^{x}(c_{-n})=y \text{ and } \forall k < n, W_{\Tree_{*}}^{x}(\mathcal{B}_{k}) \cap A = \emptyset)\widetilde{p}(y,A) \\
&= \sum_{y \in \Z^d} \widetilde{p}(y,A) \sum_{\gamma:x \to y} \bar{s}(\gamma) \prod_{i=0}^{\lvert \gamma \rvert-1 } (1-\widetilde{p}(\gamma(i),A)) \\
&=\sum_{y \in \Z^{d}}  \bar{g}_{A}(x,y)\widetilde{p}(y,A),
\end{aligned}
}\end{equation*}
where \(\bar{s}(\gamma)\) is the probability that a walk with
\(\bar{\theta}\)-jumps realizes the path \(\gamma\), and \(\bar{g}_{A}\)
is the Green function for the associated killed random walk.
\Cref{lem:GreenFunctionApproximation} also applies to \(\frac{\bar{g}_{A}}{g_{\bar{\theta}}}\),
and \(g_{\bar{\theta}}(z)=g_{\theta}(-z)\) so we can approximate the
previous sum by the convolution product
\((\widetilde{p}(\cdot,A)*g_{\theta})(x)\). As for \Cref{prop:AlternativeGreenFunction}, the
technical \Cref{lem:ConvolutionRegular} and \eqref{eq:AsymptoticAdjointVisitingProba} gives that

\begin{equation*}{
(\widetilde{p}(\cdot,A)*g_{\theta})(x) \sim \left(\frac{\bcp(A)}{\bcp(\{ 0 \})}\right)^{\alpha-1} (\widetilde{p}*g_{\theta})(x) \sim \left(\frac{\bcp(A)}{\bcp(\{ 0 \})}\right)^{\alpha-1} G(x).
}\end{equation*}

Thus the last piece to finish this proof is to make rigorous the
approximation of $\sum_{y \in \Z^{d}}  \bar{g}_{A}(x,y)\widetilde{p}(y,A)$ by \((\widetilde{p}(\cdot,A)*g_{\theta})(x)\). We do this by showing that

\begin{equation}\label{eq:ApproximationGoalForAlternativeGreen}{
\sum_{y \in \Z^{d}} g_{\bar{\theta}}(x,y)\widetilde{p}(y,A) \left\lvert  1-\frac{\bar{g}_{A}(x,y)}{g_{\bar{\theta}}(x,y)}  \right\rvert = o_{x \to \infty} \big((\widetilde{p}(\cdot,A)*g_{\theta})(x)\big).
}\end{equation}

To obtain this last point, fix some large radius \(R\) and split the sum
between \(\lVert y \rVert \leq  R\) and \(\lVert y \rVert > R\). This
leads to

\begin{equation*}{
\sum_{y \in \Z^{d}} g_{\bar{\theta}}(x,y)\widetilde{p}(y,A) \left\lvert  1-\frac{\bar{g}_{A}(x,y)}{g_{\bar{\theta}}(x,y)}  \right\rvert \leq   C(R) g_{\bar{\theta}}(x) + \sup_{\lVert z \rVert,\lVert y \rVert \geq  R } \left\lvert  1-\frac{\bar{g}_{A}(z,y)}{g_{\bar{\theta}}(z,y)}  \right\rvert (\widetilde{p}(\cdot,A)*g_{\theta})(x),
}\end{equation*}
where \(C(R)\) is a constant with respect to \(x\), which will not
matter since
\(g_{\theta}(x)=o \big((\widetilde{p}(\cdot,A)*g_{\theta})(x)\big)\).
We get

\begin{equation*}{
\varlimsup_{ x \to \infty } \frac{\sum_{y \in \Z^{d}} g_{\bar{\theta}}(x,y)\widetilde{p}(y,A) \left\lvert  1-\frac{\bar{g}_{A}(x,y)}{g_{\bar{\theta}}(x,y)}  \right\rvert}{(\widetilde{p}(\cdot,A)*g_{\theta})(x)} \leq  \sup_{\lVert z \rVert,\lVert y \rVert \geq  R } \left\lvert  1-\frac{\bar{g}_{A}(z,y)}{g_{\bar{\theta}}(z,y)}  \right\rvert,
}\end{equation*}
so by \Cref{lem:GreenFunctionApproximation} sending \(R\) to \(+\infty\) gives \eqref{eq:ApproximationGoalForAlternativeGreen} and
concludes this proof.
\end{proof}

\subsection{\texorpdfstring{\(\alpha\)-stable branching capacity of balls}{\textbackslash alpha-stable branching capacity of balls}}\label{alpha-stable-branching-capacity-of-balls-1}

Let us finally turn to the proof of \Cref{thm:BranchingCapBall}, \textit{i.e.} the estimate on the (\(\alpha\)-stable) branching
capacity of a ball \(B(0,R)\). To do so, we actually prove a branching
analog of \eqref{eq:ClassicalHittingProbaEstimate} in this case, namely that

\begin{proposition}\label{prop:BranchingHittingProbalargeBall}
\begin{equation*}{
\P(\Tree^{x}\cap B(0,R) \neq  \emptyset) \asymp \P(H(\Tree) > R^{2}) \mathrm{Cap}(B(0,R))g_{\theta}(x),
}\end{equation*}
where we recall that \(\asymp\) means here that the two terms have the same order for
large \(x\) and \(R\), up to some constants that do not depend on \(x\)
or \(R\).
\end{proposition}
\Cref{thm:BranchingCapBall} is a direct consequence of this since
\(\bcp(B(0,R))=\lim_{ x \to \infty } \frac{\P(\Tree^{x}\cap B(0,R) \neq  \emptyset)}{g_{\theta}(x)}\)
(\Cref{thm:BranchingCapacityVisitingProba}).

We will again rely on the spinal decomposition \Cref{prop:ZhuKilledWalk}, hence to
ease notation, we set
\begin{itemize}
    \item \(p_{R}(x) =p(x,B(0,R))\) and \(\widetilde{p}_{R}(x) =\widetilde{p}(x,B(0,R))\),
    \item \(s_{R}(\gamma)=s_{B(0,R)}(\gamma)= s(\gamma)\prod_{i=0}^{k-1} \bigl( 1-\widetilde{p}_{R}(\gamma(i))\) and \(g_{R}(x,y)=\sum_{\gamma:x \to y}s_{R}(\gamma)\).
\end{itemize}

Moreover, we will also need to control some big jumps with the following
result.

\begin{lemma}[Second overshoot lemma]\label{lem:OvershootLemma2}

We assume that there is \(\delta >  d\) such that
\(\P(\lVert S_{1} \rVert>u) \preceq u^{-\delta}\). There is a constant
\(C\) such that for \(r' \geq r\) and \(\lVert x \rVert > r+r'\), we
have

\begin{equation*}{
\sum_{y :\lVert y \rVert < r } \sum_{\gamma:x \to y \subset B(0,r+r')} s(\gamma) \leq  C \frac{r'^{d-\delta}}{\lVert x \rVert ^{d-2}}
}\end{equation*}

\end{lemma}

With this in hands, we will be able to split \(p_{R}(x)\) into two
parts. The first one accounts for going from \(x\) to close to
\(B(0,R)\) and will cost \(\mathrm{Cap}(B(0,R))g_{\theta}(x)\), while
the other correspond to going from the neighborhood or \(B(0,R)\) to
this ball and will add the penalization \(\P(H(\Tree)>R^{2})\). However,
we will first prove an upper bound for \Cref{prop:BranchingHittingProbalargeBall}, and then use it to get the next
lemma, which gives a better control of the killing rate.

\begin{lemma}\label{lem:UniformBoundAjointVisitingProba}

Fix \(\beta \in \bigl(2, (\alpha-1)(d-2) \bigr)\), which exists as soon
as \(d > \frac{2\alpha}{\alpha-1}\). There is \(C \in (0,+\infty)\) such
that for all \(R>0\) large enough, for all \(y\) such that
\(\lVert y \rVert > R\), one has \begin{equation*}{
\widetilde{p}_{R}(y) \leq  C \times  \left( \frac{R}{\lVert y \rVert} \right)^{\beta}\times   R^{-2}.
}\end{equation*}

\end{lemma}

The computations are quite similar from those leading to \Cref{lem:FirstEstimatesVisitingProba} ,
but here we must also keep track of the dependence in \(R\), hence we
state this result as a new lemma. This will then enable us to get the
corresponding lower bound to prove \Cref{prop:BranchingHittingProbalargeBall}.

\begin{proof}[Proof of \Cref{lem:OvershootLemma2}.]

First note that

\begin{equation*}{
\sum_{y :\lVert y \rVert < r } \sum_{\gamma:x \to y \subset B(0,r+r')} s(\gamma) = \P_{x}(H_{\leq r+r'}<+\infty \text{ and } S_{H_{\leq r+r'}} \in B(0,r)),
}\end{equation*}
where $H_{\leq r+r'}$ is the hitting time of $B(0,r+r')$ by $S$.
By considering the position \(z \not\in B(0,r+r')\) occupied by the walk
before entering this ball, we get 
\begin{equation*}
\P_{x}(H_{\leq r+r'}<+\infty \text{ and } S_{H_{\leq r+r'}} \in B(0,r)) = \sum_{z:\lVert z \rVert > r+r'}  \sum_{\gamma : z \to x \subset B(0,r+r')^c} \P_{z}(S_{1} \in B(0,r)) \times \bar{s}(\gamma)
\end{equation*} 
where $\bar{s}(\gamma)$ is again the probability that a reversed random walk $\bar{S}$ realizes the path $\gamma$. We still have
\(\sum_{\gamma : z \to x \subset B(0,r+r')^c} \bar{s}(\gamma) \leq g_{\theta}(x,z)\),
but since \(r'\geq r\) we also have
\(\lVert z \rVert -r \geq \frac{\left\lVert z  \right\rVert}{2}\), hence
\(\P_{z}(S_{1} \in B(0,r)) \leq  C\times  \P_{0}(\lVert S_{1} \rVert \geq  \lVert z \rVert)\).

Finally, it remains to bound from above
\(\sum_{z:\lVert z \rVert >r+r'} \P_{0}(\lVert S_{1} \rVert \geq \lVert z \rVert )g_{\theta}(x,z)\).
We do it by considering 
\begin{itemize}
\item \(I_{1}=\left\{  z: \lVert z \rVert> r+r' \text{ and } \lVert x-z \rVert \leq  \frac{\left\lVert  x  \right\rVert}{2}  \right\}\);
\item \(I_{2}=\left\{  z: r+r' < \lVert z \rVert \leq  2 \lVert x \rVert \text{ and } \lVert x-z \rVert >  \frac{\left\lVert  x  \right\rVert}{2}  \right\}\);
\item \(I_{3} = \{  z:\lVert z \rVert> 2\lVert x \rVert \}\).
\end{itemize}

For all \(z \in I_{1}\),
\(\lVert z \rVert \geq  \frac{\left\lVert  x  \right\rVert}{2}\), thus
we get the following bound :

\begin{equation*}{
\sum_{z \in I_{1}} \P_{0}(\lVert S_{1} \rVert \geq \lVert z \rVert )g_{\theta}(x,z) \preceq \lVert x \rVert ^{-\delta}\sum_{y:\lVert y \rVert \leq \frac{\lVert x \rVert}{2} } g_{\theta}(y) \preceq \lVert x \rVert ^{2-\delta}.}
\end{equation*}

On the set \(I_{2}\), since \(\lVert z-x \rVert\) is of order
\(\lVert  x \rVert\), we have
\(g_{\theta}(x,z) \leq \frac{C_{2}}{\lVert x \rVert^{d-2}}\) and then

\begin{equation*}{
\sum_{z \in I_{2}} \P_{0}(\lVert S_{1} \rVert \geq \lVert z \rVert )g_{\theta}(x,z) \preceq \frac{1}{\lVert x \rVert ^{d-2}}\sum_{z:\lVert z \rVert > r+r' } \lVert z \rVert^{-\delta}   \preceq \frac{r'^{d-\delta}}{\lVert x \rVert ^{d-2}}.}
\end{equation*}

On the last set \(I_{3}\), \(\lVert x-z \rVert\) is of order
\(\lVert z \rVert\) hence we have

\begin{equation*}{
\sum_{z \in I_{3}} \P_{0}(\lVert S_{1} \rVert \geq \lVert z \rVert )g_{\theta}(x,z) \preceq \sum_{z:\lVert z \rVert  > 2\lVert x \rVert } \lVert z \rVert^{2-\delta-d} \preceq \lVert x \rVert ^{2-\delta}.}
\end{equation*}

Finally, since \(d-\delta<0\) and \(r' < \lVert x \rVert\), we see that
\(\lVert x \rVert^{2-\delta}=\frac{\left\lVert  x  \right\rVert^{d-\delta}}{\lVert x \rVert^{d-2}} \leq  \frac{r'^{d-\delta}}{\lVert x \rVert ^{d-2}},\)
and the desired result follows.
\end{proof}

\begin{proof}[Proof of the upper bound for \Cref{prop:BranchingHittingProbalargeBall}.]
The first step is to recall from \Cref{prop:ZhuKilledWalk} that

\begin{equation*}{
p_{R}(x) := \P(\Tree^{x}\cap B(0,R) \neq  \emptyset) = \sum_{y: \lVert y \rVert \leq R} \sum_{\gamma:x \to y } s_{R}(\gamma).}
\end{equation*}
Then, we decompose any path \(\gamma\) according to the first time it
enters \(B(0,4R)\) as depicted in \Cref{fig:FirstPassageDecomposition}. This gives

\begin{equation}\label{eq:HittingProbaFirstTimeDecomposition}{
p_{R}(x) = \sum_{z: \lVert z \rVert <4R} \left( \sum_{\gamma_{1}:x \to z \subset B(0,4R)^{c}} s_{R}(\gamma_{1}) \right) \P(\Tree^{z}\cap B(0,R) \neq  \emptyset).}
\end{equation}

\begin{figure}[!h]
    \centering
    \includegraphics[width = 0.75\linewidth]{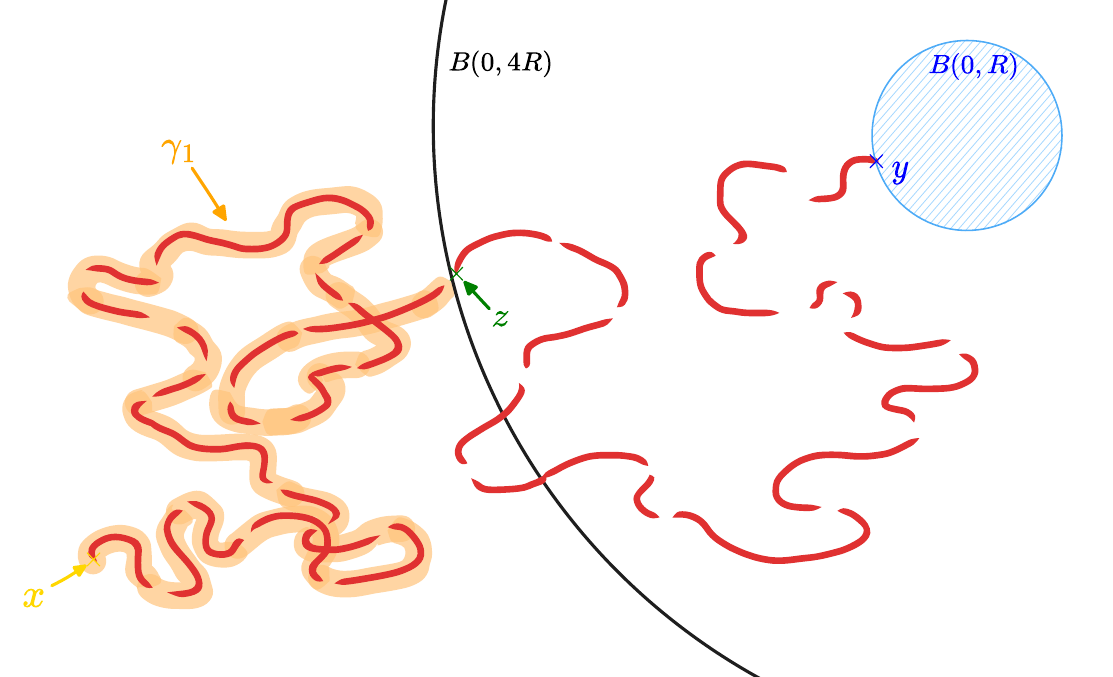}
    \caption{\centering A decomposition of trajectories from $x$ to $B(0,R)$ for the killed random walk. We decompose according to the first position $z \in B(0,4R)$ reached by the trajectories: $\gamma_1$ denotes the path from $x$ to $z$, then the remaining portion can be any path from $z$ to $y \in B(0,R)$.}
    \label{fig:FirstPassageDecomposition}
\end{figure}

The part \(\lVert z \rVert \leq 2R\) actually is negligible due to the
cost of a big jump from \(B(0,4R)^{c}\) to \(B(0,2R)\). Indeed by \Cref{lem:OvershootLemma2}, we
get

\begin{equation*}{
\sum_{z: \lVert z \rVert <2R} \left( \sum_{\gamma_{1}:x \to z \subset B(0,4R)^{c}} s_{R}(\gamma_{1}) \right) \P(\Tree^{z}\cap B(0,R) \neq  \emptyset) \leq C R^{d-\delta} g_{\theta}(x).}
\end{equation*}

We can also control the remaining part by proving the
existence of \(C \in (0,+\infty)\) such that for all \(z\) with
\(2R < \lVert z \rVert <4R\),

\begin{equation}\label{eq:HittingProbaBoundedByHeight}{
\P(\Tree^{z}\cap B(0,R) \neq  \emptyset) \leq   C\times  \P(H(\Tree)>R^{2}).}
\end{equation}
Indeed, on the event \(\Tree^{z}\cap B(0,R) \neq  \emptyset\), since
\(\lVert z \rVert > 2R\) we clearly have
\(\max_{u \in \Tree} \lVert W_{\Tree}^{z}(u) -z\rVert \geq R\). By
equivalence of \(\lVert \cdot \rVert\) and
\(\lVert \cdot \rVert_{\infty}\), we see that
\[\P(\max_{u \in \Tree} \lVert W_{\Tree}^{z}(u) -z\rVert \geq R) \leq  C_{1} \P(\max_{u \in \Tree} \lvert W_{\Tree}'(u)\rvert \geq R),\]
where \(W_{\Tree}'\) is a \(1\)-dimensional \(\Tree\)-walk starting from
\(0\) with centered jumps with a finite moment of order $\delta > d > \frac{2\alpha}{\alpha-1}$. According to \eqref{eq:QueueMaxStable}, in this \(1\)-dimensional situation
\(\P(\max_{u \in \Tree} \lvert W_{\Tree}'(u)\rvert \geq R) \sim C_{2} \P(H(\Tree)>R^{2})\),
so we directly get \eqref{eq:HittingProbaBoundedByHeight}.

Now, we also have the simple bound

\begin{equation*}{
\sum_{z: 2R <  \lVert z \rVert < 4R} \sum_{\gamma:x \to z \subset B(0, 4R)^{c}} s_{R}(\gamma) \leq  \P_{x}(H_{B(0,4R)}<+\infty).}
\end{equation*}

Putting everything together, we get the following upper bound:

\begin{equation}\label{eq:GlobalUpperBoundHittingProba}{
\begin{aligned}
p_{R}(x) &\leq  C\times   \P(H(\Tree) > R^{2})\P_{x}(H_{B(0,4R)}<+\infty) + C R^{d-\delta} g_{\theta}(x) \\
& \leq  C'\times   (\P(H(\Tree) > R^{2}) \mathrm{Cap}(B(0,R))+R^{d-\delta})g_{\theta}(x) \\
& \leq  C''\times \P(H(\Tree) > R^{2}) \mathrm{Cap}(B(0,R))g_{\theta}(x)
\end{aligned}}
\end{equation}
where we used \eqref{eq:ClassicalHittingProbaEstimate} to estimate the hitting probability for the
classical walk, as well as the fact that according to the standard
estimates \eqref{eq:NewtonianCapBall} and \eqref{eq:SlackResult},
\(R \mapsto \P(H(\Tree) > R^{2}) \mathrm{Cap}(B(0,R))\) is bounded below
by a \(\left( d-\frac{2\alpha}{\alpha-1} \right)\)-varying function
hence \(R^{d-\delta}=o(1)\) is negligible in front of this function.
\end{proof}

\begin{proof}[Proof of
\Cref{lem:UniformBoundAjointVisitingProba}.]
Recall from \eqref{eq:AjointVisitingProba}, in \Cref{lem:FirstEstimatesVisitingProba} that an application of the
branching property yields

\begin{equation*}{
\widetilde{p}_{R}(y) = \frac{\gen_{\mu}\left(1-q_{R}(y)\right)-1 +q_{R}(y)}{q_{R}(y)},}
\end{equation*}
where \(q_{R}(y)=\E(p_{R}(y+S_{1}))\). Since
\(\E(g_{\theta}(y+S_{1}))=g_{\theta}(y)-\1_{\{ 0 \}}(y) \leq  g_{\theta}(y)\)
we see that \eqref{eq:GlobalUpperBoundHittingProba} also holds with \(q_{R}\) instead of
\(p_{R}\). Then we use that
\(\Phi:s \to \frac{\gen_{\mu}(1-s)-(1-s)}{s}\) is non-decreasing and
regularly varying (at \(0\)) to obtain

\begin{equation*}{
\widetilde{p}_{R}(y) \leq  \Phi\left( C \left( \frac{R}{\lVert y \rVert } \right)^{d-2} \P(H(\Tree)>R^{2})\right) \leq C' \Phi\left( \left( \frac{R}{\lVert y \rVert } \right)^{d-2} \P(H(\Tree)>R^{2})\right) .}
\end{equation*}

Recall from \eqref{eq:GenMuStable} and \eqref{eq:SlackResult} that \(\Phi\) is
\((\alpha-1)\)-varying (at \(0\)) and

\begin{equation*}{
\Phi \bigl( \P(H(\Tree)>R^{2}) \bigr) \sim \frac{R^{-2}}{\alpha-1}.}
\end{equation*}

By choice of \(\beta\) we have \(\frac{\beta}{d-2} < \alpha-1\), thus by
Potter's bound (\Cref{prop:PotterBounds}) for all \(x<y\) small enough one has
\(\frac{\Phi(x)}{\Phi(y)} \leq 2 (\frac{x}{y})^{\beta/(d-2)}\). Here
\(\left( \frac{R}{\lVert y \rVert } \right)^{d-2} \P(H(\Tree)>R^{2}) \leq \P(H(\Tree)>R^{2}) \xrightarrow[R \to +\infty]{} 0\),
hence for \(R\) large enough we get

\begin{equation*}{
\frac{\Phi\left( \left( \frac{R}{\lVert y \rVert } \right)^{d-2} \P(H(\Tree)>R^{2})\right)}{\Phi \bigl( \P(H(\Tree)>R^{2}) \bigr) } \leq 2 \left( \frac{R}{\lVert y \rVert} \right)^{\beta}.}
\end{equation*}
Combining this with the previous display gives the desired result.
\end{proof}

\begin{proof}[Proof of the lower bound for \Cref{prop:BranchingHittingProbalargeBall}.]

We use the same decomposition as earlier (illustrated by \Cref{fig:FirstPassageDecomposition}) but this time we decompose paths according to their hitting time of \(B(0,\lambda R)\)
for some \(\lambda>2\) to be fixed:

\begin{equation}\label{eq:HittingProbaFirstDecompositionLambda}{
p_{R}(x) = \sum_{z: \lVert z \rVert <\lambda R} \left( \sum_{\gamma_{1}:x \to z \subset B(0,\lambda R)^{c}} s_{R}(\gamma_{1}) \right) \P(\Tree^{z}\cap B(0,R) \neq  \emptyset).}
\end{equation}

Let us first prove that there is \(c>0\) such that, for \(R\) large
enough and for all \(z\) with \(\lVert z \rVert < \lambda R\), we have

\begin{equation}\label{eq:HittingProbaFirstLowerBound}{
\P(\Tree^{z}\cap B(0,R) \neq  \emptyset) \geq  c\times  \P(H(\Tree)>R^{2}).}
\end{equation}
To establish the lower bound, we consider the events
\(\Tree^{z}\cap B(0,R) \neq  \emptyset \text{ and } H(\Tree) > R^{2}\).
When \(H(\Tree)> R^{2}\), it is sufficient to look at the walk along a
branch of maximal length to get an important probability of reaching
\(B(0,R)\). Actually, we can even restrict our attention to a single
vertex at height \(R^{2}\). This leads to

\begin{equation*}{
\P(\Tree^{z}\cap B(0,R) \neq  \emptyset \text{ and } H(\Tree) > R^{2}) \geq  \P(H(\Tree)>R^{2})\P_{z}(S_{R^{2}} \in B(0,R)).}
\end{equation*}
Then, by the local central limit theorem we see that
\(\min_{z:\lVert z \rVert < \lambda R}\P_{z}(S_{R^{2}} \in B(0,R))\) is
bounded away from \(0\): First, we denote by
\(f_{\Sigma_{\theta}^{2}}(x)\) the density of
\(\norm(0,\Sigma_{\theta}^{2})\), and we observe that there is
\(c_{1}>0\) such that for all \(k\) with
\(\lVert k \rVert < (\lambda+1)R\), we have
\(f_{R^{2}\Sigma_{\theta}^{2}}(k) \geq \frac{c_{1}}{R^{d}}\). The local
central limit theorem tells us that there is \(\varepsilon(R^{2})\)
going to \(0\) as \(R \to +\infty\) such that
$$\forall k \in \Z^d,\ \lvert \P_{0}(S_{R^{2}}=k) - f_{R^{2}\Sigma_{\theta}^{2}}(k)\rvert \leq \frac{\varepsilon(R^{2})}{R^{d}}$$
thus for all \(z\) with
\(\lVert z \rVert < \lambda R\), we have

\begin{equation*}{
\P_{z}(S_{R^{2}} \in B(0,R)) =  \sum_{k \in B(-z,R)}\P_{0}(S_{R^{2}}=k) \geq  \sum_{k \in B(-z,R)} \left( \frac{c_{1}}{R^{d}}-\frac{\varepsilon(R^{2})}{R^{d}}  \right) \geq  c_{2}(c_{1}-\varepsilon(R^{2})),}
\end{equation*}
which concludes the proof of \eqref{eq:HittingProbaFirstLowerBound} .

It remains to prove the second lower bound, namely that there is
\(c \in (0,1)\) such that

\begin{equation}\label{eq:HittingProbaSecondLowerBound}{
\sum_{z: \lVert z \rVert <\lambda R} \sum_{\gamma:x \to z \subset B(0,\lambda R)^{c}} s_{R}(\gamma) \geq  c\times  \P_{x}(H_{B(0,\lambda R)}<+\infty).}
\end{equation}
From the definition of \(s_{R}\), we see that we actually have to
lower bound the following expectation:

\begin{equation}\label{eq:FormulationAsExpectation}{
\frac{\sum_{z: \lVert z \rVert <\lambda R} \sum_{\gamma:x \to z \subset B(0,\lambda R)^{c}} s_{R}(\gamma)}{\P_{x}(H_{B(0,\lambda R)}<+\infty)} = \E_{x}\left( \prod_{i=0}^{H_{B(0,\lambda R)}-1} (1-\widetilde{p}_{R}(S_{i})) \Big\vert H_{B(0,\lambda R)}<+\infty \right).}
\end{equation}
To do so, we set
\(\mathcal{A}_{i} = \{ y \in \Z^{d}: R\times 2^{i} \leq  \lVert y \rVert < R\times 2^{i+1}\}\)
and define
\(T_{i,R}=\E\left( \sum_{k =0}^{H_{B(0,\lambda R)}-1} \1_{S_{k} \in \mathcal{A}_{i}} \right)\)
to be the time spent by \(S\) in the annulus \(\mathcal{A}_{i}\) before
hitting the ball \(B(0,\lambda R)\). By \Cref{lem:UniformBoundAjointVisitingProba}, we see that for
any
\(y \in \mathcal{A}_{i},\ \widetilde{p}_{R}(y)\leq 2^{-i\beta}R^{-2}\),
hence we have

\begin{equation*}{
\prod_{i=0}^{H_{B(0,\lambda R)}-1} (1-\widetilde{p}_{R}(S_{i})) \geq \prod_{i: 2^{i+1}>\lambda} (1- 2^{-i\beta}R^{-2})^{T_{i,R}} \geq 1- \sum_{i: 2^{i+1}>\lambda}  \frac{T_{i,R}}{2^{i\beta}R^{2}},
}\end{equation*}
where we used that
\(1-\prod_{j \in J}(1-p_{j}) \leq  \sum_{j \in J} p_{j}\) for any
collection of weights \((p_{j})_{j \in J} \in [0,1]^{J}\).

It now remains to upper bound
\(\E_{x}(T_{i,R} \vert H_{B(0,\lambda R)} <+\infty)\) to show that this
last sum is small enough. We first get rid of the conditioning, by
noticing that

\begin{equation*}{
\begin{aligned}
\E_{x}(T_{i,R} \vert H_{B(0,\lambda R)} <+\infty) &= \frac{1}{\P_{x}(H_{B(0,\lambda R)} <+\infty)} \sum_{k \in \N} \P_{x}(S_{k} \in \mathcal{A}_{i} \text{ and } k < H_{B(0,\lambda R)} < +\infty  ) \\
& \leq  \frac{1}{\P_{x}(H_{B(0,\lambda R)} <+\infty)} \sum_{k \in \N} \E_{x}(\1_{S_{k} \in \mathcal{A}_{i}} \P_{S_{k}} (H_{B(0,\lambda R)} < +\infty))  \\
&\leq  \frac{\sup_{y \in \mathcal{A}_{i}} \P_{y}(H_{B(0,\lambda R)} < +\infty)}{\P_{x}(H_{B(0,\lambda R)} < +\infty)} \sum_{k \in \N}\P_{x}(S_{k} \in \mathcal{A}_{i}).
\end{aligned}
}\end{equation*}

\(\sum_{k \in \N}\P_{x}(S_{k} \in \mathcal{A}_{i}) = \sum_{y \in \mathcal{A}_{i}} g_{\theta}(y-x)\)
is the expected time spent in \(\mathcal{A}_{i}\) by an unconditioned
walk started at \(x\), and we can bound this quantity with \eqref{eq:GreenEstimates}
as follows. Consider \(i_{0}\) such that \(x \in \mathcal{A}_{i_{0}}\).
\begin{itemize}
\item For all \(i \geq  i_{0}+2\), since
\(\forall y \in \mathcal{A}_{i}, \lVert y-x \rVert \geq  2^{i-1}R\) we
get
\(\sum_{y \in \mathcal{A}_{i}} g_{\theta}(y-x) \leq  C_{1} \times (2^{i}R)^{2}\).
\item For all \(i \leq i_{0}-2\), since
\(\forall y \in \mathcal{A}_{i},\lVert y-x \rVert \geq  \frac{\left\lVert  x  \right\rVert}{2}\),
we get this time
\(\sum_{y \in \mathcal{A}_{i}} g_{\theta}(y-x) \leq  C_{2} \times \frac{(2^{i}R)^{d}}{\lVert x \rVert^{d-2}}\).
\item For \(i \in \{ i_{0}-1,i_{0},i_{0}+1 \}\), we split the sum according
to whether \(\lVert y-x \rVert \leq  2^{i}R\) or not, and we use that
\(\sum_{z:\lVert z \rVert < a} \frac{1}{1 \vee \lVert z \rVert^{d-2}} \preceq a^{2}\)
to get
\(\sum_{y \in \mathcal{A}_{i}} g_{\theta}(y-x) \leq C_{3} \times (2^{i}R)^{2}\).
\end{itemize}

Finally, the remaining pre-factor can easily be bounded with a standard
result on the Newtonian capacity, namely \eqref{eq:ClassicalHittingProbaEstimate}, to get
\(\frac{\sup_{y \in \mathcal{A}_{i}} \P_{y}(H_{B(0,\lambda R)} < +\infty)}{\P_{x}(H_{B(0,\lambda R)} < +\infty)} \leq  C_{4}\times \left( \frac{\left\lVert  x  \right\rVert}{2^{i}R} \right)^{d-2}\)
where \(C_{4}\) does not depend on \(\lambda\) either. Hence we conclude
that there is a constant \(C\) (independent of \(x,R,i\) and
\(\lambda\)) such that

\begin{equation*}{
\E_{x}(T_{i,R} \vert H_{B(0,\lambda R)} <+\infty) \leq  C \times  (2^{i}R)^{2}.
}\end{equation*}
As a consequence,

\begin{equation*}{
\sum_{i: 2^{i+1}>\lambda}  \frac{\E_{x}(T_{i,R} \vert H_{B(0,\lambda R)} <+\infty)}{2^{i\beta}R^{2}} \leq C \times \sum_{i: 2^{i+1}>\lambda} \left( \frac{1}{2^{\beta-2}} \right)^{i}. 
}\end{equation*}

Thus we may choose \(\lambda\) to get the right-hand side \(<1\), and
the expectation in \Cref{eq:FormulationAsExpectation} is then bounded away from \(0\),
uniformly in \(x\) and \(R\).

This concludes the proof of \eqref{eq:HittingProbaSecondLowerBound}. Combined with \eqref{eq:HittingProbaFirstDecompositionLambda} and \eqref{eq:HittingProbaFirstLowerBound}, it finally gives that there is \(c>0\) and \(\lambda>0\)
such that for \(x\) and \(R\) large enough,

\begin{equation*}{
\begin{aligned}
p_{R}(x) &\geq  c\times   \P(H(\Tree) > R^{2})\P_{x}(H_{B(0,\lambda R)}<+\infty) \\
& \geq  c'\times   \P(H(\Tree) > R^{2}) \mathrm{Cap}(B(0,R))g_{\theta}(x),
\end{aligned}
}\end{equation*}
which is the desired lower bound for \Cref{prop:BranchingHittingProbalargeBall}.
\end{proof}

\appendix

\section{On regularly varying functions}\label{annexe-regularly-varying-functions}

\subsection{\texorpdfstring{Offspring attracted to an \(\alpha\)-stable distribution}{Offspring attracted to an \textbackslash alpha-stable distribution}}\label{offspring-attracted-to-an-alpha-stable-distribution}

In this paper, we have to use slowly (and regularly) varying functions because of the assumption that \(\mu\) is in the domain of attraction of an
\(\alpha\)-stable distribution, denoted by $\mu \in \mathrm{dom}(\alpha)$, where \(\alpha \in (1,2]\). Indeed, this assumption means
that given i.i.d. random variables \(X_1, X_2,\ldots\)
\(\mu\)-distributed, there is a slowly varying function \(\ell\) such
that

\begin{equation}\label{eq:StableCLT}{
\frac{\sum_{i=1}^n X_{i} - n}{n^{1/\alpha}\ell(n)} \xrightarrow[n \to \infty]{(d)} Y
}\end{equation}
where the limit is a spectrally positive \(\alpha\)-stable law
characterized by:
\(\E(e^{-\lambda Y})=\exp(\lambda^{\alpha}), \forall~\lambda~\geq~0\).

In the case where \(\mu\) has a finite variance, this assumption holds with \(\ell(n)=\frac{\sigma_{\mu}}{\sqrt{ 2 }}\), however in full generality we may not have an explicit expression for \(\ell\) and we can only rely on the fact that it is a slowly varying function.

We mainly use an alternative characterization of a critical offspring $\mu \in \mathrm{dom}(\alpha)$: \(\mu\) satisfies
\eqref{eq:StableCLT} if and only if there is another slowly varying function
\(L\) such that the generating function \(\gen_{\mu}\) of \(\mu\)
satisfies

\begin{equation*}{
 \gen_{\mu}(1-s)-(1-s) = s^{\alpha}L\left( \frac{1}{s} \right).
}\end{equation*}

See for instance \autocite{BjornbergStefansson2015RandomWalkRandom}. The
slowly varying functions \(\ell\) and \(L\) are different but can be
related: \(L(x) \to +\infty\) if and only if \(\ell(x) \to +\infty\),
and when \(\sigma_{\mu}^2 < +\infty\), one has
\(L(x) \to \frac{\sigma_{\mu}^2}{2}\) and
\(\ell(x) \to \frac{\sigma_{\mu}}{\sqrt{ 2 }}\) as \(x \to +\infty\).

\subsection{Definitions and properties of regularly varying functions}\label{definitions-and-properties}

We give here a brief account of the properties of regularly varying
functions, and refer to
\autocite{BinghamGoldieTeugels1987RegularVariation} for a complete
survey on those functions.

\begin{definition}\label{def:regular}

A measurable function \(f:(0,+\infty)\mapsto (0,+\infty)\) is said to be

\begin{itemize}
\item
  \emph{slowly varying} at \(+\infty\) when

  \begin{equation*}{
  \forall \lambda>0,\ \frac{f(\lambda x)}{f(x)} \xrightarrow[x \to \infty]{} 1;
  }\end{equation*}
\item
  \emph{regularly varying with index \(\beta \in \R\)} (or
  \emph{\(\beta\)-varying}) at \(+\infty\) when

  \begin{equation*}{
  \forall \lambda>0,\ \frac{f(\lambda x)}{f(x)} \xrightarrow[x \to \infty]{} \lambda^{\beta},
  }\end{equation*}

  or equivalently when \(x \mapsto x^{-\beta}f(x)\) is a slowly varying
  function at \(+\infty\).
\end{itemize}

\end{definition}

\begin{remark}\label{rmk:UniformCvThm}
According to the Uniform Convergence Theorem \autocite[Theorem~1.2.1]{BinghamGoldieTeugels1987RegularVariation}, as soon as a function $f$ satisfies the definition above, we have that for all $\lambda_{1},\lambda_{2}>0$, the convergence
 \begin{equation*}{
  \forall \lambda>0,\ \frac{f(\lambda x)}{f(x)} \xrightarrow[x \to \infty]{} \lambda^{\beta},
  }\end{equation*}
actually holds uniformly in $\lambda \in [\lambda_{1},\lambda_{2}]$.
\end{remark}

\begin{remark}

By change of variables, one can define regularly varying functions at
any \(a \in [-\infty,+\infty]\) in a similar way for any function \(f\)
defined on a neighbourhood of \(a\). Here by default a regularly varying
function will be regularly varying at \(+\infty\), and we will only
write down explicitly the limiting point when it is more convenient to
deal with regularly varying functions at \(a \neq +\infty\).

\end{remark}

An obvious example of a \(\beta\)-varying function is
\(x \mapsto x^{\beta}\), and actually any regularly varying function
should be thought as roughly asymptotic to a power function (and any
slowly varying function as roughly constant) since many asymptotic
properties of power functions extend to regularly varying functions. In
this direction, the following result enables to bound regularly varying
functions by power functions.

\begin{proposition}[Potter's bound, \protect{\autocite[Theorem 1.5.6]{BinghamGoldieTeugels1987RegularVariation}}]\label{prop:PotterBounds}

Let \(f\) be a \(\beta\)-varying function. For all \(C>1\) and
\(\varepsilon>0\), there is \(A>0\) such that

\begin{equation*}{
 \forall x,y \geq A,\ \frac{f(x)}{f(y)} \leq C\times\left( \frac{x}{y} \right)^{\beta+\varepsilon} \vee \left( \frac{x}{y} \right)^{\beta-\varepsilon}.
}\end{equation*}

In particular, for all \(\varepsilon>0\) we have
\(x^{\beta-\varepsilon} \ll f(x) \ll x^{\beta+\varepsilon}\).

\end{proposition}

\begin{remark}

Note the following useful consequence: If \(f\) is \(\beta\)-varying
with \(\beta>0\), and if we have \(p_{1}(x)\) and \(p_{2}(x)\) such that
\(p_{1}(x) \to +\infty\) and \(p_{1}(x) \preceq p_{2}(x)\) as
\(x \to +\infty\), then we also have \(f(p_{1}(x)) \preceq f(p_{2}(x))\).

\end{remark}

In the same direction, a regularly varying function with non-zero index
may not be a monotone function but can always be approximated by one:

\begin{proposition}[Monotone equivalent, \protect{\autocite[Theorem 1.5.3]{BinghamGoldieTeugels1987RegularVariation}}]\label{prop:MonotoneEquivalent}

Let \(f\) be a \(\beta\)-varying function. If \(\beta \neq  0\) then
there is a monotone function \(\phi\) such that
\(f(x) \sim_{x \to +\infty} \phi(x)\). Note that \(\phi\) must be
non-decreasing when \(\beta>0\) and non-increasing when \(\beta<0\).

\end{proposition}

Given two regularly varying functions \(f\) and \(g\) such that
\(g(x) \to +\infty\), \(f \circ g\) is clearly regularly varying at
\(+\infty\) with
\(\text{index}(f \circ g) = \text{index}(f) \times\text{index}(g)\).
This stability by composition also comes with the existence of
(asymptotic) inverse function for regularly varying function with
positive index.

\begin{proposition}[Asymptotic inversion, \protect{\autocite[Theorem 1.5.12]{BinghamGoldieTeugels1987RegularVariation}}]\label{prop:asymptoticInversion}

Let \(f\) be a regularly varying function with index \(\beta >0\), then
there is a \(\frac{1}{\beta}\)-varying function \(g\) such that

\begin{equation*}{
 f(g(x)) \sim g(f(x)) \sim x \ \text{ as } x \to +\infty.
}\end{equation*}
Moreover, \(g\) is unique up to asymptotic equivalence.

\end{proposition}

An asymptotic inverse of a function \(f\) can be expressed by means
of the generalized inverse of monotone functions, but also by means of
the notion of conjugate for slowly varying functions, see
\autocite{BinghamGoldieTeugels1987RegularVariation} for more details on
this.

Finally, the following results show that regularly varying functions can
be integrated, and sometimes differentiated, just as power functions.

\begin{theorem}[Karamata's theorem, direct half, \protect{\autocite[section 1.6]{BinghamGoldieTeugels1987RegularVariation}}]\label{thm:Karamata}

Let \(\ell\) be a slowly varying function. There is \(A>0\) such that
\(\ell\) is locally bounded on \([A,+\infty)\), hence we can integrate
and we have the following estimates:

\begin{itemize}
\item
  For \(\beta > -1\),

  \begin{equation*}{
  \int_{A}^x t^{\beta}\ell(t)\dd t \sim_{x \to +\infty} \frac{x^{\beta+1}}{\beta+1} \ell(x).
  }\end{equation*}
\item
  For \(\beta < -1\),

  \begin{equation*}{
  \int_{x}^{+\infty} t^{\beta}\ell(t) \dd t \sim_{x \to +\infty} \frac{x^{\beta+1}}{-\beta-1}\ell(x).
  }\end{equation*}
\item
  For \(\beta=-1\), the function
  \(x \mapsto \int_{A}^x \frac{\ell(t)}{t}\dd t\) is slowly varying and
  such that

  \begin{equation*}{
  \frac{1}{\ell(x)}\int_{A}^x \frac{\ell(t)}{t}\dd t \xrightarrow[x \to +\infty]{} 0. 
  }\end{equation*}

  Moreover if \(\int_{A}^{+\infty} \frac{\ell(t)}{t}\dd t <+\infty\)
  then the function
  \(x \mapsto \int_{x}^{+\infty} \frac{\ell(t)}{t}\dd t\) is also slowly
  varying and such that

  \begin{equation*}{
  \frac{1}{\ell(x)}\int_{x}^{+\infty} \frac{\ell(t)}{t}\dd t \xrightarrow[x \to +\infty]{} +\infty.
  }\end{equation*}
\end{itemize}

\end{theorem}

\begin{theorem}[Monotone Density Theorem, \protect{\autocite[Theorem 1.7.2]{BinghamGoldieTeugels1987RegularVariation}}]\label{thm:MonotoneDensity}

Let \(f\) be a positive measurable function, locally bounded on
\([A,+\infty)\). Assume that \(f\) is ultimately monotone.

\begin{itemize}
\item
  If there is \(C>0\), \(\beta>0\) and a slowly varying function
  \(\ell\) such that

  \begin{equation*}{
  \int_{A}^{x}f(t)\dd t \sim_{x \to +\infty} C x^{\beta} \ell(x),
   }\end{equation*}

  then \(f(x) \sim_{x \to +\infty} \beta C x^{\beta-1} \ell(x)\).
\item
  If \(\int_{A}^{+\infty}f(t)\dd t < +\infty\) and in addition there is
  \(C>0\), \(\beta < 0\) and a slowly varying function \(\ell\) such
  that

  \begin{equation*}{
  \int_{x}^{+\infty} f(t)\dd t \sim_{x \to +\infty} C x^{\beta} \ell(x),
  }\end{equation*}

  then \(f(x) \sim_{x \to +\infty} -\beta C x^{\beta-1} \ell(x)\).
\end{itemize}

\end{theorem}

\subsection{Convolution of regularly varying functions}\label{convolution-of-regularly-varying-functions}

\begin{lemma}\label{lem:ConvolutionRegular}

Consider \(f,g: \Z^d \to \R_+\) and \(\lVert \cdot \rVert\) an Euclidian
norm on \(\R^d\). Assume that
\(f(x)\sim \Vert x\Vert^{-a} \tilde f(\Vert x\Vert)\) and
\(g(x)\sim \Vert x\Vert^{-b} \tilde g(\Vert x\Vert)\) when
\(\Vert x\Vert\to+\infty\), where \(\tilde f, \tilde g\) are slowly
varying functions and \(a,b >0\) satisfy \(\max(a,b) < d < a+b\). Then
there is \(C=C(d,a,b)\) such that

\begin{equation*}{
f*g(x) \sim_{x\to \infty} C\Vert x\Vert^df(x)g(x).
}\end{equation*}

\end{lemma}

\begin{proof}
Let us first extend \(f,g\) to \(\R^{d}\) by constancy (say,
\(f(x_{1},\dots,x_{d})=f(\lfloor x_{1} \rfloor, \dots \lfloor x_{d} \rfloor)\)).
This and the assumption \(d < a+b\) enable to write

\begin{equation}\label{eq:IntegralFormulation}{
\forall x \in \R^{d}, f*g(x) = \sum_{z \in \Z^d} f(z)g(x-z) = \int_{\R^d} f(y)g(x-y)\dd y < +\infty.
}\end{equation}

We then choose \(u \in \R^{d}\) such that \(\lVert u \rVert=1\) and, for
a given \(x\), \(\theta_{x}\) a rotation (for \(\lVert \cdot{} \rVert\))
mapping \(u\) to \(\frac{x}{\lVert x \rVert}\). We also set, for any
\(v \in \R^{d}\) and \(r>0\), \(V(v,r)=B(0,r)\cup B(v,r)\).

The change of variables \(y = \Vert x\Vert\theta_{x}(z)\) yields

\begin{equation}\label{eq:ChangeVariables}{
\begin{aligned}
\int_{\R^d \setminus V(x,\sqrt{ \lVert x \rVert  })}f(y)g(x-y)\dd y = \Vert x\Vert^{d}f(x)g(x) \int_{\R^d\setminus V\left( u,\lVert x \rVert^{-1/2} \right)} \frac{f\bigl(\Vert x\Vert\theta_{x}(z)\bigr)}{f(x)} \frac{g\bigl( \Vert x\Vert\theta_{x}(u-z) \bigr)}{g(x)}\dd z.
\end{aligned}
}\end{equation}

Observe that, for any choice of \(x \mapsto \theta_{x}\),
\(\frac{f\bigl(\Vert x\Vert\theta_{x}(z)\bigr)}{f(x)} \frac{g\bigl(\Vert x\Vert\theta_{x}(u-z) \bigr)}{g(x)} \to \frac{1}{\Vert z\Vert^a} \frac{1}{\Vert u-z\Vert^b}\)
as \(\lVert x \rVert \to +\infty\). Moreover, one can fix
\(\varepsilon>0\) such that \(\max(a,b)+\epsilon <d\) and
\(a+b -2\varepsilon > d\), and then, since
\(\lVert y \rVert \geq \lVert x \rVert^{1/2}\) ensures that
\(\frac{f(y)}{\lVert y \rVert^{a}\tilde{f}(y)}\) is bounded away from
\(0\) and \(+\infty\) for large \(x\), we can apply Potter's bound (\Cref{prop:PotterBounds}) to
get that \(\forall z \not\in V(u,\Vert x\Vert^{-1/2})\),

\begin{equation*}{
\frac{f(\Vert x\Vert\theta_{x}(z))}{f(x)} \leq c \max(\lVert z \rVert^{-(a+\varepsilon)},\lVert z \rVert^{-(a-\varepsilon)}),
}\end{equation*}
and in the same way for large \(x\) and
\(z \not\in V(u,\Vert x\Vert^{-1/2})\) we have

\begin{equation*}{
\frac{g\bigl( \Vert x\Vert\theta_{x}(u-z) \bigr)}{g(x)} \leq c' \max(\lVert u-z \rVert ^{-(b+\varepsilon)},\lVert u-z \rVert ^{-(b-\varepsilon)}).
}\end{equation*}

This means that for \(x\) large enough, for all
\(z \in V(u,\lVert x \rVert^{-1/2})\) we have
\(\frac{f\bigl(\Vert x\Vert\theta_{x}(z)\bigr)}{f(x)} \frac{g\bigl(\Vert x\Vert\theta_{x}(u-z) \bigr)}{g(x)} \leq  c''\max(\lVert z \rVert^{-(a+\varepsilon)},\lVert z \rVert^{-(a-\varepsilon)})\times \max(\lVert u-z \rVert ^{-(b+\varepsilon)},\lVert u-z \rVert ^{-(b-\varepsilon)})\)
and by choice of \(\varepsilon\) we can apply Lebesgue's dominated
convergence theorem in \Cref{eq:ChangeVariables} to get

\begin{equation}\label{eq:DominatedConvergence}{
\int_{\R^d \setminus V(x,\Vert x\Vert ^{1/2})}f(y)g(x-y)\dd y \sim_{\Vert x\Vert \to + \infty} C(d,a,b)\Vert x\Vert^{d}f(x)g(x),
}\end{equation}
where
\(C(d,a,b)=\int_{\R^d} \frac{1}{\Vert z\Vert^a} \frac{1}{\Vert u-z\Vert^b}\dd z\)
does not depend on our choice of a unitary vector \(u\) (indeed this
integral is left unchanged by rotations).

To conclude, it simply remains to control what happens on
\(B(0,\lVert x \rVert^{1/2})\) and \(B(x,\Vert x\Vert^{1/2})\). This is
done by using Karamata's \Cref{thm:Karamata} and the uniform convergence theorem for
regularly varying functions (see \Cref{rmk:UniformCvThm}). For instance on
\(B(0,\lVert x \rVert^{1/2})\) we have

\begin{equation}\label{eq:Around0}{
\begin{aligned}
\int_{B(0,\Vert x\Vert ^{1/2})} f(y)g(x-y)\dd y 
&\preceq g(x) \int_{B(0,\Vert x\Vert ^{1/2})} f(y)\dd y\\
&\preceq g(x) \int_{0}^{\Vert x\Vert^{1/2}}t^{d-1-a}\tilde{f}(t)\dd t\\
&\preceq g(x)\Vert x\Vert^{(d-a)/2}\tilde{f}(\Vert x\Vert^{1/2} ),
\end{aligned}
}\end{equation}
and since \(d-a>0\), \Cref{eq:Around0} finally gives

\begin{equation}\label{eq:Around0Conclusion}{
\int_{B(0,\Vert x\Vert ^{1/2})} f(y)g(x-y)\dd y  =o \left(\Vert x\Vert ^d g(x)f(x)\right).
}\end{equation}
In the same way, from \(d-b>0\) we get

\begin{equation}\label{eq:AroundxConclusion}{
\int_{B(x,\Vert x\Vert^{1/2})} f(y)g(x-y)\dd y   =o \left(\Vert x\Vert ^d g(x)f(x)\right).
}\end{equation}
Gathering \Cref{eq:DominatedConvergence}, \Cref{eq:Around0Conclusion} and \Cref{eq:AroundxConclusion} proves
\Cref{lem:ConvolutionRegular}
\end{proof}

\printbibliography

\end{document}